\documentclass[10pt]{article}
\usepackage{amsfonts,color}
\usepackage{amssymb}
 \usepackage{footnote}
\usepackage{mathtools}
\usepackage{amsthm,wasysym}
\usepackage[english]{babel}
\usepackage{bbm}
\usepackage{colonequals}

\usepackage{graphicx}
\usepackage{bm}
\usepackage{amsfonts}
\usepackage{hyperref}
\usepackage{subcaption}
\usepackage{amsmath}
\usepackage{nicefrac}
\usepackage{gensymb}
\usepackage{enumerate}
\usepackage[nameinlink,capitalize]{cleveref}

\usepackage[nottoc,numbib]{tocbibind}

\usepackage{pgf,tikz,pgfplots}
\pgfplotsset{compat=1.14}
\usepackage{mathrsfs}
\usetikzlibrary{arrows}

\DeclareMathOperator{\at}{\bigg\vert}
\newcommand{\vb}[1]{\mathbf{#1}}
\DeclareMathOperator{\Dc}{\textrm{D}^\mathrm{curl}}
\DeclareMathOperator{\curl}{\mathrm{curl}_{\mathrm{2D}}}

\DeclareMathOperator{\di}{\mathrm{div}}
\DeclareMathOperator{\sym}{\mathrm{sym}}
\DeclareMathOperator{\skw}{\mathrm{skew}}
\DeclareMathOperator{\Le}{\mathit{L}^2}
\DeclareMathOperator{\Lez}{\mathit{L}^2_0}
\DeclareMathOperator{\Hone}{\mathit{H}^1}
\DeclareMathOperator{\Honez}{\mathit{H}_0^1}
\DeclareMathOperator{\Hc}{\mathit{H}(\mathrm{curl}}
\DeclareMathOperator{\Hcz}{\mathit{H}_0(\mathrm{curl}}
\DeclareMathOperator{\V}{\mathit{V}}
\DeclareMathOperator{\U}{\mathit{U}}

\DeclareMathOperator{\Q}{\mathit{Q}}
\DeclareMathOperator{\X}{\mathit{X}}

\DeclareMathOperator{\Pe}{\mathit{P}}
\DeclareMathOperator{\C}{\mathit{C}}

\newcommand{\dd}{\mathrm{d}}
\newcommand{\mue}{\mu_\mathrm{e}}
\newcommand{\mumi}{\mu_\mathrm{micro}}
\newcommand{\muma}{\mu_\mathrm{macro}}
\newcommand{\Lc}{L_\mathrm{c}}

\allowdisplaybreaks[1]

\makeindex
                                                     
\newtheorem{theorem}{Theorem}
\newtheorem{lemma}{Lemma}
\newtheorem{remark}{Remark}
\usepackage{chngcntr}
\counterwithin{theorem}{section}
\counterwithin{lemma}{section}
\counterwithin{remark}{section}

\newcommand{\AS}[1]{{\color{black} #1}}

   \makeatletter
\let\@fnsymbol\@arabic
\makeatother

\crefname{Problem}{Problem.}{Problem.}

\title{A hybrid $\Hone \times \Hc)$ finite element formulation for a relaxed micromorphic continuum \AS{model of antiplane shear}
}
\author{\normalsize{Adam Sky}\thanks{Corresponding author: Adam Sky, Chair of Statics and Dynamics, Technische Universit\"at Dortmund, August-Schmidt-Str. 8, 44227 Dortmund, Germany, email: Adam.chejanovsky@tu-dortmund.de}
	, \quad
	\normalsize{Michael Neunteufel}\thanks{Michael Neunteufel, Institute for Analysis and Scientific Computing, Technische Universit\"at Wien, Wiedner Hauptstr. 8-10 , 1040 Wien, Austria, email: michael.neunteufel@tuwien.ac.at}
	, \quad
	\normalsize{Ingo M\"unch}\thanks{Ingo M\"unch, Head of Chair of Statics and Dynamics, Technische Universit\"at Dortmund, August-Schmidt-Str. 8, 44227 Dortmund, Germany, email: ingo.muench@tu-dortmund.de}
	, \quad
	\normalsize{Joachim Sch\"oberl}\thanks{Joachim Schöberl, Institute for Analysis and Scientific Computing, Technische Universit\"at Wien, Wiedner Hauptstr. 8-10 , 1040 Wien, Austria, email: joachim.schoeberl@tuwien.ac.at}
	\\
	and \quad
	\normalsize{Patrizio Neff}\thanks{Patrizio Neff,  \ \ Head of Chair for Nonlinear 
		Analysis and Modelling, Faculty of Mathematics, Universit\"{a}t Duisburg-Essen,
		Thea-Leymann Str. 9, 45127 Essen, Germany, email: patrizio.neff@uni-due.de}
}

\begin{document}

\maketitle

\begin{abstract}
One approach for the simulation of metamaterials is to extend an associated continuum theory concerning its kinematic equations, and the relaxed micromorphic continuum represents such a model. It incorporates
the Curl of the nonsymmetric microdistortion in the free energy function. This suggests the existence of
solutions not belonging to $\Hone$, such that standard nodal $\Hone$-finite elements yield unsatisfactory convergence
rates and might be incapable of finding the exact solution. Our approach is to use base functions stemming
from both Hilbert spaces $\Hone$ and $\Hc)$, demonstrating the central role of such combinations for this
class of problems. For simplicity, a reduced two-dimensional relaxed micromorphic continuum \AS{describing antiplane shear} is introduced,
preserving the main computational traits of the three-dimensional version. This model is then used for the
formulation and a multi step investigation of a viable finite element solution, encompassing examinations of
existence and uniqueness of both standard and mixed formulations and their respective convergence rates.
\\
\vspace*{0.25cm}

\end{abstract}

\tableofcontents

\section{Introduction}
Materials with a pronounced microstructure such as metamaterials, see e.g. \cite{Madeo2018,Neff_low,AIVALIOTIS201999,d'Agostino2020}, porous media, composites etc., activate micro-motions which are not accounted for in classical continuum mechanics, where each material point is equipped with only three translational degrees of freedom. Therefore, several approaches to model such materials can be found in literature, such as multi-scale finite element methods \cite{multi,ABDULLE2006,Eidel2018} or generalized continuum theories. The latter can be classified into higher gradient theories \cite{highergradient,Kirchner2006,Mindlin1968,Neff2009} and so called micromorphic continuum theories \cite{Steigmann2012,Neff2007}. These theories extend the kinematics of the material point. Depending on the extension one obtains for example micropolar \cite{Munch2011,Jeong2010,Munch2017}, microstretch \cite{Romeo2020} or microstrain \cite{Forest2006,Hutter2016} theories. In its most general setting, as introduced by Eringen und Mindlin \cite{book, article}, a micromorphic continuum theory allows the material point to undergo an affine distortion independent of its macroscopic deformation arising from the displacement field. Consequently, in the micromorphic theory a material point is considered with $3 + 9 = 12$ degrees of freedom, of which the microdistortion $\bm{P}$ encompasses $9$. The various micromorphic theories differ in their proposition of the free energy functional. While classical theories incorporate the full gradient of the microdistortion $\nabla \bm{P}$ into the energy function \cite{Neff2014}, the relaxed micromorphic theory \cite{Neff2014,Neff2020,Neff_wave,Madeo2016} considers only $\text{Curl}\bm{P}$. The incorporation of the Curl of the microdistortion, formally known as the dislocation density, into the free energy functional relaxes the continuity assumptions on the microdistortion and enlarges the space of possible weak solutions, i.e. $[\Hone]^3 \times \mathit{H}(\text{Curl})$. Furthermore, the relaxed micromorphic theory aspires to capture the entire spectrum of mechanical behaviour between the macro and micro scale of the material. This is achieved via homogenization of the material parameters and the introduction of the characteristic length $\Lc$ \cite{Neff2019,Madeo2018}, which determines the influence of the dislocation density in the free energy functional.
\AS{Specific analytical solutions to the full isotropic relaxed micromorphic model are presented in \cite{Rizzi2021,rizzi2020analytical,rizzi2020analytical_2}.}

For non-trivial boundary value problems, solutions of continuum theories are approximated via the finite element method. While the standard Lagrange elements are well suited for solutions in $\Hone$, solutions in $\mathit{H}(\text{curl})$ may require a different class of elements, depending on the problem at hand. The lowest class of finite elements in $\mathit{H}(\text{curl})$, sometimes called edge elements, have been derived by N\'{e}d\'{e}lec $\cite{Nedelec1980, Ned2}$. Extensions to higher order element formulations can be found in \cite{Zaglmayr2006,Joachim2005,dem1,dem2}. In this paper we consider finite element formulations employing either $\Hone \times [\Hone]^2$ or $\Hone \times \Hc)$ and investigate their validity in correctly approximating results in the relaxed micromorphic continuum. Furthermore, we test both a primal and mixed formulation of the corresponding boundary problem for increasingly large values of the characteristic length $\Lc$. To that end, we consider a planar version of the relaxed micromorphic continuum, namely of antiplane shear \cite{Voss2020}. More precisely, the matrix-Curl in 3D reduces to a scalar-curl of the microdistortion in 2D. However, the results of our investigation directly apply to the \AS{full} three-dimensional version. 

The paper is organized as follows: In the following section we introduce the planar relaxed micromorphic continuum. \cref{ch:2} is devoted to prove solvability of the primal and mixed problem and discussing properties in the limit case $\Lc\to\infty$, in both the continuous and discrete settings, respectively. In \cref{ch:4} we present appropriate base functions for $\Hc)$, the corresponding covariant Piola transformation for N\'{e}d\'{e}lec finite elements and the resulting stiffness matrices. Finally, we present several numerical examples to confirm the theoretical results.

\section{The planar relaxed micromorphic continuum}
\label{intro}
The free energy functional of the relaxed micromorphic continuum \cite{Neff2019,Neff2014} incorporates the gradient of the displacement field, the microdistortion and its Curl
\begin{gather}
\begin{align} 
I(\vb{u}, \bm{P}) =  & \, \dfrac{1}{2} \int_{\Omega} \langle \mathbb{C}_{\textrm{e}} \sym(\nabla \vb{u} - \bm{P}) , \, \sym(\nabla \vb{u} - \bm{P}) \rangle
 +  \langle \mathbb{C}_{\textrm{micro}} \sym\bm{P} , \, \sym\bm{P} \rangle \notag \\ 
& + \langle \mathbb{C}_{\textrm{c}} \skw(\nabla \vb{u} - \bm{P}) , \, \skw (\nabla \vb{u} - \bm{P}) \rangle
 + \dfrac{\mu_\text{macro} \, \Lc^2}{2} \, \| \text{Curl}\bm{P} \|^2 - \langle \vb{f} , \, \vb{u} \rangle  - \langle \bm{M} , \, \bm{P} \rangle \, \dd X \, ,  
\end{align}
\label{eq:1} \\[2ex]
\nabla \vb{u} = \begin{bmatrix}
u_{1,1} & u_{1,2} & u_{1,3} \\
u_{2,1} & u_{2,2} & u_{2,3} \\
u_{3,1} & u_{3,2} & u_{3,3} 
\end{bmatrix}
\, , \quad
\text{Curl}\bm{P} = \begin{bmatrix}
	(\text{curl} \begin{bmatrix}
	P_{11} & P_{12} & P_{13} 
	\end{bmatrix})^T \\[1ex]
	(\text{curl} \begin{bmatrix}
	P_{21} & P_{22} & P_{23} 
	\end{bmatrix})^T \\[1ex]
	(\text{curl} \begin{bmatrix}
	P_{31} & P_{32} & P_{33} 
	\end{bmatrix})^T
	\end{bmatrix} \, , \quad 
\text{curl}\vb{v} = \nabla \times \vb{v} \, , 
\end{gather}
with $\vb{u} : \Omega \subset \mathbb{R}^3 \to \mathbb{R}^3$ and $\bm{P}: \Omega \subset \mathbb{R}^3 \to \mathbb{R}^{3 \times 3}$ representing the displacement and the non-symmetric microdistortion, respectively. Here, $\mathbb{C}_{\textrm{e}}$ and $\mathbb{C}_{\textrm{micro}}$ are standard elasticity tensors and $\mathbb{C}_{\textrm{c}}$ is a positive semi-definite coupling tensor for rotations. The macroscopic shear modulus is denoted by $\mu_\text{macro}$ and the parameter $\Lc \geq 0$ represents the characteristic length scale motivated by the microstructure.

From now on, we consider the planar reduction of this continuum \AS{to antiplane shear, still} capturing the main mathematical aspects of the three-dimensional version, namely the additional microdistortion and the curl
\begin{gather}
I (u,\bm{\zeta}) = \int_{\Omega}  \mue \|\nabla u - \bm{\zeta}\|^2 + \mumi \|\bm{\zeta}\|^2  + \muma \dfrac{\Lc^2}{2}  \|\curl\bm{\zeta}\|^2 - \langle u , \, f \rangle - \langle \bm{\zeta} , \, \bm{\omega} \rangle\, \dd X \, , \quad \Omega \subset \mathbb{R}^2 \, ,\label{eq:2} 
\end{gather}
where we employ the two-dimensional definitions of the curl and gradient operators
\begin{align}
\curl \bm{\zeta} = \zeta_{2,1} - \zeta_{1,2} \, , \quad \bm{\zeta} \in \mathbb{R}^2 \, , \qquad \Dc (u) = \begin{bmatrix}
	u_{,2} \\ -u_{,1}
\end{bmatrix} \, ,  u \in \mathbb{R} \, , \qquad \nabla u = \begin{bmatrix}
u_{,1} \\ u_{,2}
\end{bmatrix} \, , \quad u \in \mathbb{R} \, .
\end{align}

In \cref{eq:2} we reduced the displacement to a scalar field $u:\Omega \subset \mathbb{R}^2 \to \mathbb{R}$ and the microdistortion $\bm{P}$ to a vector field $\bm{\zeta} : \Omega \subset \mathbb{R}^2 \to \mathbb{R}^2$. \AS{The displacement field $u$ is now perpendicular to the plane of the domain}. The elasticity tensors $\mathbb{C}_\text{e}$ and $\mathbb{C}_\text{micro}$ are replaced by the scalars $\mue, \, \mumi > 0$ and $\mathbb{C}_\text{c}$ no longer appears.
\AS{\begin{remark}
		The simplification of the model to antiplane shear serves to facilitate the mathematical analysis of the model and allows for a thorough investigation of the numerical behaviour of finite element solutions in the relaxed micromorphic theory. Whether this \textbf{reduced} model can be applied to real-world metamaterials is unclear at this time. For applications of the \textbf{full} three-dimensional theory see \cite{Madeo2018, Madeo2016}.
\end{remark}} 
In order to find functions minimizing the potential energy $I$ we calculate the variations with respect to $u$ and $\bm{\zeta}$
\begin{subequations}
	\begin{align}
	&\int_{\Omega}  2 \mue \langle (\nabla u - \bm{\zeta}) , \, \nabla \delta u \rangle \,  \dd X = \int_{\Omega}   \langle \delta u , \, f \rangle \, \dd X \, , \label{eq:a} \\
	&\int_{\Omega}  2\mue \, \langle (\nabla u - \bm{\zeta}) , \, (-\delta \bm{\zeta}) \rangle + 2\mumi \,  \langle \bm{\zeta} , \, \delta \bm{\zeta} \rangle + \muma \, \Lc ^2 \, \langle \curl\bm{\zeta} , \, \curl\delta \bm{\zeta} \rangle \, \dd X = \int_{\Omega}   \langle \delta \bm{\zeta} , \, \bm{\omega} \rangle \, \dd X \, .
	\label{eq:3} 
	\end{align}
\end{subequations}
Partial integration of \cref{eq:a} and \cref{eq:3} yields the strong form including boundary conditions (see \cref{ap:B} for more details)
\begin{subequations}
\label{eq:strong}
\begin{align}
-2\mue\di(\nabla u - \bm{\zeta}) & = f &&\text{ in } \Omega \, ,  \label{eq:strong1} \\
-2\mue (\nabla u - \bm{\zeta}) + 2\mumi \bm{\zeta} + \muma \Lc^2 \, \Dc(  \curl\bm{\zeta} )  & =  \bm{\omega} &&\text{ in } \Omega  \, , \label{eq:strong2}\\
u & = \widetilde{u} && \text{ on } \Gamma_D^u\, ,\\
\langle \bm{\zeta} , \, \bm{\tau} \rangle & = \langle \widetilde{\bm{\zeta}} , \,  \bm{\tau} \rangle &&  \text{ on } \Gamma_D^{\zeta}\,, \label{eq:consisten}\\
\langle \nabla u , \, \bm{\nu} \rangle & = \langle \bm{\zeta} , \, \bm{\nu} \rangle &&  \text{ on } \Gamma_N^u\,, \label{eq:neumann}\\
\curl\bm{\zeta} & =0 && \text{ on }\Gamma_N^\zeta \,, \label{eq:cur}
\end{align}  
\end{subequations}
where $\bm{\tau}$ and $\bm{\nu}$ denote the outer tangent and normal vector on the boundary, see \cref{fig:domian}, and with $\widetilde{u}$ and $\widetilde{\bm{\zeta}}$ the displacement and microdistortion fields on $\Gamma_D^u$ and $\Gamma_D^\zeta$ are prescribed.
From a mathematical point of view, it is possible to prescribe the tangential components of the microdistortion $\bm{\zeta}$ on the boundary $\Gamma_D^\zeta$. This is used to test our numerical formulation in \cref{ch:7}.
However, from the point of view of physics it is impossible to control the microdistortion of the continuum with no direct relation to the displacement $u$
and as such, the \textbf{consistent coupling condition} $\langle \bm{\zeta} , \, \bm{\tau} \rangle = \langle \nabla \widetilde{u} , \, \bm{\tau} \rangle$ arises on the Dirichlet boundary, being common to both $u$ and $\bm{\zeta}$, enforcing the condition $\Gamma_D^{\zeta}\subset\Gamma_D^u$. Furthermore, Dirichlet boundary data for the microdistortion $\bm{\zeta}$ are not required for the existence of a unique solution here, as coercivity in the appropriate spaces is still determined.
\begin{figure}
	\centering
	\definecolor{ududff}{rgb}{0.30196078431372547,0.30196078431372547,1.}
	\definecolor{ffqqff}{rgb}{1.,0.,1.}
	\definecolor{qqqqff}{rgb}{0.,0.,1.}
	\definecolor{qqzzcc}{rgb}{0.,0.6,0.8}
	\begin{tikzpicture}[scale = 1][line cap=round,line join=round,>=triangle 45,x=1.0cm,y=1.0cm]
	\clip(-2.2,5.1) rectangle (11,9);
	\draw [shift={(7.752719419156214,6.940151305514022)},line width=2.pt]  plot[domain=-1.1063259343647651:1.395759665287226,variable=\t]({1.*1.0428913505061734*cos(\t r)+0.*1.0428913505061734*sin(\t r)},{0.*1.0428913505061734*cos(\t r)+1.*1.0428913505061734*sin(\t r)});
	\draw [shift={(4.49857469685139,14.245386765916496)},line width=2.pt,dash pattern=on 2pt off 2pt,color=qqzzcc]  plot[domain=4.38530463822533:5.136692472017892,variable=\t]({1.*9.039185235700698*cos(\t r)+0.*9.039185235700698*sin(\t r)},{0.*9.039185235700698*cos(\t r)+1.*9.039185235700698*sin(\t r)});
	\draw [shift={(2.5378976991138265,7.679147499509712)},line width=2.pt,color=qqzzcc]  plot[domain=3.3948617433976085:4.2703956526624305,variable=\t]({1.*2.2056806476805906*cos(\t r)+0.*2.2056806476805906*sin(\t r)},{0.*2.2056806476805906*cos(\t r)+1.*2.2056806476805906*sin(\t r)});
	\draw [shift={(1.3889089532431642,7.412102085635455)},line width=2.pt,color=qqqqff]  plot[domain=1.0825549335295008:3.423473403538445,variable=\t]({1.*1.0268530941355785*cos(\t r)+0.*1.0268530941355785*sin(\t r)},{0.*1.0268530941355785*cos(\t r)+1.*1.0268530941355785*sin(\t r)});
	\draw [shift={(3.9595594229361315,11.434728967844318)},line width=2.pt,dash pattern=on 4pt off 4pt,color=qqqqff]  plot[domain=4.121766213863017:4.931963289100204,variable=\t]({1.*3.7512335218833015*cos(\t r)+0.*3.7512335218833015*sin(\t r)},{0.*3.7512335218833015*cos(\t r)+1.*3.7512335218833015*sin(\t r)});
	\draw [shift={(4.49857469685139,14.245386765916496)},line width=2.pt,dash pattern=on 4pt off 4pt]  plot[domain=4.38530463822533:5.136692472017892,variable=\t]({1.*9.039185235700698*cos(\t r)+0.*9.039185235700698*sin(\t r)},{0.*9.039185235700698*cos(\t r)+1.*9.039185235700698*sin(\t r)});
	\draw [shift={(6.673267623030113,2.6856616638565862)},line width=2.pt,color=ffqqff]  plot[domain=1.3364120818142247:1.927613934165801,variable=\t]({1.*5.429913001008422*cos(\t r)+0.*5.429913001008422*sin(\t r)},{0.*5.429913001008422*cos(\t r)+1.*5.429913001008422*sin(\t r)});
	\draw (8.89,7.118323494687127) node[anchor=north west] {$\Gamma_D^u$};
	\draw [color=qqzzcc](0.1,6.45) node[anchor=north west] {$\Gamma_D^\zeta$};
	\draw [color=ududff](0.4,9) node[anchor=north west] {$\Gamma_N^\zeta$};
	\draw [color=ffqqff](7.0453915781188545,8.661023219205033) node[anchor=north west] {$\Gamma_N^u$};
	\draw [->,line width=2.pt] (5.953519343975546,8.067661070519344) -- (5.834656954071715,8.894905719780734);
	\draw [->,line width=2.pt] (5.953519343975546,8.067661070519344) -- (5.140880177239844,7.950897264825448);
	\draw (4.32205430932704,6.69329397874852) node[anchor=north west] {$\Omega$};
	\draw (5.4,7.9) node[anchor=north west] {$\boldsymbol{\tau}$};
	\draw (6,8.6) node[anchor=north west] {$\boldsymbol{\nu}$};
	\draw [shift={(3.9595594229361315,11.434728967844318)},line width=2.pt,dash pattern=on 2pt off 2pt,color=ffqqff]  plot[domain=4.121766213863017:4.931963289100204,variable=\t]({1.*3.7512335218833015*cos(\t r)+0.*3.7512335218833015*sin(\t r)},{0.*3.7512335218833015*cos(\t r)+1.*3.7512335218833015*sin(\t r)});
	\end{tikzpicture}
	\caption{Outer tangent $\bm{\tau}$ and normal vector $\bm{\nu}$ on the boundary of the domain $\Omega$.}
	\label{fig:domian}
\end{figure}
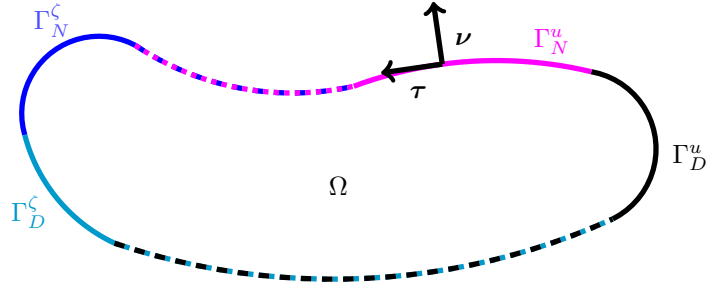

\section{Solvability and limit problems} \label{ch:2}
\subsection{Continuous case} 
In this section we prove the existence and uniqueness of the weak form of the planar relaxed micromorphic continuum. Further, the corresponding mixed formulation is presented, whose coercivity constant is independent of $\Lc$. Finally, we study necessary and sufficient conditions such that $\nabla u=\bm{\zeta}$ is guaranteed in the limit $\Lc\to\infty$. For simplicity, we assume homogeneous Dirichlet conditions on the entire boundary throughout this section, i.e., $u=0$ and $\langle \bm{\zeta} , \, \bm{\tau} \rangle =0$ on $\Gamma^u_D=\Gamma^{\zeta}_D=\partial \Omega$, and mention that the proof can be readily adapted for inhomogeneous and mixed boundary conditions as long as the Dirichlet boundary for the displacements is non-trivial, $|\Gamma^u_D|>0$, \cite{Neff_existence}.

We define the following Hilbert spaces and their respective norms
\begin{subequations}
	\begin{align}
	\Hone(\Omega) &= \{ u \in \Le(\Omega) \, | \,  \nabla u \in \Le(\Omega) ^2 \} \, , \quad \|u \|_{\Hone}^2 = \|u\|_{\Le}^2 + \| \nabla u \|_{\Le}^2 \,,\\
	\Honez(\Omega) &= \{ u \in \Hone(\Omega) \, | \, u=0 \text{ on }\partial \Omega \} \, ,\\
	\Hc,\Omega) &= \{ \bm{\zeta} \in \Le(\Omega)^2 \, | \curl\bm{\zeta} \in \Le(\Omega)   \} \, , \quad \| \bm{\zeta} \|_{\Hc)}^2 = \| \bm{\zeta} \|_{\Le}^2 + \| \curl\bm{\zeta} \|_{\Le}^2 \, ,\\
	\Hcz,\Omega) &= \{ \bm{\zeta} \in \Hc , \, \Omega) \, | \, \langle \bm{\zeta} , \, \bm{\tau} \rangle =0  \text{ on }\partial \Omega \} \, ,
	\end{align}
\end{subequations}
which are based on the Lebesgue norm and space
\begin{align}
\| u \|_{\Le}^2 = \int_{\Omega} \| u \|^2\, \dd X \,,\quad \Le(\Omega)=\{u:\Omega\to\mathbb{R}\,|\, \|u\|_{\Le}<\infty\}\, , \quad \Lez(\Omega)= \left \{u\in\Le(\Omega)\,|\, \int_{\Omega}u\,\dd X = 0 \right \} \, .
\end{align} 

Further, we use the product space $\X=\Honez(\Omega)\times \Hcz,\Omega)$ with the norm
\begin{align}
&\| \{u, \bm{\zeta}\} \|_{\X} = \|u \|_{\Hone} + \| \bm{\zeta} \|_{\Hc )} \, ,
\end{align}
to define the following minimization problem\footnote{Note carefully that $u$ and $\bm{\zeta}$ are two independent variables and lead to a minimization problem despite the resemblance to mixed formulations, i.e. saddle-point problems.}: Find $\{u,\bm{\zeta}\}\in X$ such that for all $\{\delta u,\delta \bm{\zeta}\}\in X$
\begin{align}
\underbrace{\int_{\Omega} 2\mue \langle (\nabla u - \bm{\zeta}) , \, (\nabla \delta u -\delta \bm{\zeta}) \rangle + 2\mumi \,  \langle \bm{\zeta} , \, \delta \bm{\zeta} \rangle + \muma \, \Lc ^2 \, \langle \curl\bm{\zeta} , \, \curl\delta \bm{\zeta} \rangle \, \dd X}_{\displaystyle =a(\{u,\bm{\zeta}\},\{\delta u,\delta\bm{\zeta}\})} = 
\int_{\Omega}   \langle \delta u , \, f \rangle + \langle \delta \bm{\zeta} , \, \bm{\omega} \rangle \, \dd X \, ,
\label[Problem]{eq:weak}
\end{align}
In order to show the existence of unique solutions we consider the Lax--Milgram theorem.  
\begin{theorem}
	\label{thm:primal_hcurl_solve}
	If $\mue, \mumi, \muma, \Lc > 0$, then \cref{eq:weak} has a unique solution $\{u,\bm{\zeta}\}\in X$ and there holds the stability estimate
	\begin{align*}
	\|\{u,\bm{\zeta}\}\|_X \leq \frac{1}{\beta}\Big(\|f\|_{\Le}+\|\bm{\omega}\|_{\Le}\Big)\,,
	\end{align*}
	with $\beta=\beta(\mue,\mumi,\muma,\Lc) > 0$ .
\end{theorem}
\begin{proof}
Using Cauchy--Schwarz and triangle  inequality yields the continuity of $a(\cdot,\cdot)$
\begin{align}
	| a(\{u, \bm{\zeta}\}, \{ \delta u, \delta \bm{\zeta} \}) | & \leq
	  2 \mue \| \nabla u - \bm{\zeta} \|_{\Le}  \| \nabla \delta u - \delta \bm{\zeta} \|_{\Le} + 2 \mumi \| \bm{\zeta} \|_{\Le} \| \delta \bm{\zeta} \|_{\Le}
	  + \muma \Lc^2 \| \curl\bm{\zeta} \|_{\Le} \, \| \curl\delta \bm{\zeta} \|_{\Le} \notag \\[2ex]
	& \leq
	c_1 \bigg ( \Big ( \| \nabla u \|_{\Le} + \| \bm{\zeta} \|_{\Le} \Big ) \Big (\| \nabla \delta u \|_{\Le} + \| \delta \bm{\zeta} \|_{\Le} \Big )  
	+ \| \bm{\zeta} \|_{\Le} \| \delta \bm{\zeta} \|_{\Le} + \| \curl \bm{\zeta} \|_{\Le} \| \curl \delta \bm{\zeta} \|_{\Le} \bigg ) \notag \\[2ex]
	& \leq 3 \, c_1 \| \{u, \bm{\zeta}\} \|_{\X} \| \{\delta u,  \delta \bm{\zeta}\} \|_{\X}\, ,
	\label{eq:conti}
\end{align}
for all $\{u, \bm{\zeta}\}, \{\delta u, \delta\bm{\zeta}\}\in X$ with the constant $c_1 = \max\left\{2\mue, 2\mumi,\muma \Lc^2\right\}$.

By employing Young's\footnote{Young: $- v\, w \geq - \left ( \dfrac{\varepsilon \, v^2}{2} + \dfrac{w^2}{2 \varepsilon} \right ) \, , \quad \forall \varepsilon > 0 \, , v,w\in \mathbb{R}$} and Poincar\'{e}-Friedrich's\footnote{Poincar\'{e}-Friedrich: $\exists c_F > 0 : \quad \|v \|_{\Le} \leq c_F \| \nabla v \|_{\Le} \, , \quad \forall v \in \mathit{H}^1_0(\Omega)$} inequalities we show the bilinear form to be coercive
\begin{align}
	a(\{u, \bm{\zeta}\}, \{ u, \bm{\zeta} \})& = 2 \mue \Big( \| \nabla u \|_{\Le}^2 + \| \bm{\zeta} \|_{\Le}^2 -2 \langle \nabla u, \bm{\zeta} \rangle_{\Le} \Big ) + 2\mumi \| \bm{\zeta} \|_{\Le}^2 + \muma \, \Lc^2 \| \curl\bm{\zeta} \|_{\Le}^2 \notag \\
	& \geq 2\mue \Big (\| \nabla u \|_{\Le}^2 + \| \bm{\zeta} \|_{\Le}^2 - \varepsilon \| \nabla u \|_{\Le}^2 - \dfrac{1}{\varepsilon} \| \bm{\zeta} \|_{\Le}^2 \Big ) + 2\mumi \| \bm{\zeta} \|_{\Le}^2 +  \muma \, \Lc^2 \| \curl\bm{\zeta} \|_{\Le}^2 \notag \\
	&\geq c_3\Big (\|\nabla u\|^2_{\Le}+\|\bm{\zeta}\|^2_{\Le}+\|\curl\bm{\zeta} \|^2_{\Le}\Big)\notag\\
	& \geq \frac{c_3}{2}\min\left\{1,\frac{1}{1+c_F^2}\right\} \| \{u , \bm{\zeta} \} \|_{\X}^2 \, ,
\end{align} 
when the constant $\varepsilon$ is chosen as $1 > \varepsilon >\dfrac{\mue}{\mue + \mumi}$, which is possible for $\mue,\mumi>0$. Consequently, the coercivity constant reads
\begin{align}
	\beta = \frac{c_3}{2}\min\left\{1,\frac{1}{1+c_F^2}\right\},\quad c_3=\min \left \{2\mue (1 - \varepsilon), 2\mue\left(1 - \dfrac{1}{\varepsilon}\right) + 2 \mumi, \, \muma \Lc^2 \right \} \, .
\end{align}
This finishes the proof. 
\end{proof}

\begin{remark}\label{rem:discussion_h1_hcurl}
Note, that the proof fails when taking instead $\X=\Honez(\Omega)\times [\Honez(\Omega)]^2$ as $a(\cdot,\cdot)$ is then no longer coercive in this space because one cannot find a constant $c>0$ such that $\|\bm{\zeta}\|^2_{\Le}+\|\curl \bm{\zeta}\|^2_{\Le}\geq c\,\|\bm{\zeta}\|^2_{\Hone}$, for all $\bm{\zeta}\in[\Honez(\Omega)]^2$. As $[\Hone(\Omega)]^2$ is dense in $\Hc,\Omega)$, we might expect convergence for $\bm{\zeta}\in [\Hone(\Omega)]^2$, however, at the cost of sub-optimal convergence rates in the discretized setting. We present numerical examples, where the exact solution is in $\Hc,\Omega)$ but not in $[\Hone(\Omega)]^2$ observing only slow convergence. If the exact solution is smooth, i.e. $\bm{\zeta}$ is also in $[\Hone(\Omega)]^2$, optimal convergence is observed.
\end{remark}

An important aspect of the relaxed micromorphic continuum is its relation to the classical continuum theory (linear elasticity). This relation is governed by the material constants, where the characteristic length $\Lc$ plays a significant role. We are therefore interested in robust computations with respect to $\Lc$. 

The following result characterizes the conditions when a trivial solution with respect to $\Lc$ is expected.
\begin{theorem}
	\label{thm:omega_gradfield}
	Assume that the requirements of \cref{thm:primal_hcurl_solve} are fulfilled. Further, let $\bm{\omega} = \nabla r$ be a gradient field, then, the microdistortion $\bm{\zeta}$ is compatible, i.e. $\bm{\zeta} = \nabla \chi$ and the solution $\{u,\bm{\zeta}\}\in X$ is independent of the parameter $\Lc$.
\end{theorem}
\begin{proof}
	We make the ansatz $\bm{\zeta}=\nabla\chi$, $\chi\in\Hone(\Omega)$ and insert it in \cref{eq:weak} choosing $\delta u=0$
	\begin{align*}
	\int_{\Omega}2\mue\langle \nabla\chi-\nabla u, \, \delta\bm{\zeta}\rangle + 2\mumi\langle\nabla\chi, \, \delta\bm{\zeta}\rangle\,\dd X  = \int_{\Omega}\langle\nabla r, \delta\bm{\zeta}\rangle\,\dd X\quad \text{ for all }\delta\bm{\zeta}.
	\end{align*}
	We can express
	\begin{align}
	\nabla\chi =\dfrac{1}{2(\mue+\mumi)}\left(\nabla r+2\mue\nabla u\right)\label{eq:zeta_gradfield}
	\end{align}
	and inserting into \cref{eq:weak} choosing $\delta \bm{\zeta}=0$ gives the following Laplace problem for $u$
	\begin{align*}
	\int_{\Omega}\dfrac{2\mue \, \mumi}{\mue+\mumi}\langle\nabla u,\nabla\delta u\rangle\,\dd X = \int_{\Omega}f\delta u+\dfrac{\mue}{\mue+\mumi}\langle\nabla r,\nabla \delta u\rangle\,\dd X \quad \text{ for all } \delta u \, ,
	\end{align*}
	which is uniquely solvable. Since by Lax--Milgram the solution is unique, $\bm{\zeta} = \nabla \chi$ and the resulting $u$ are the only possible solutions. According to \cref{eq:zeta_gradfield} the solution of \cref{eq:weak} is given independently of $\Lc$.
\end{proof}

Considering the limit case  $\Lc = 0$, the continuity of the bilinear form $a(\cdot, \cdot)$ follows automatically from \cref{eq:conti}. However, for coercivity to hold, the space for $\bm{\zeta}$ must be changed to $[\Le(\Omega)]^2$, i.e., the regularity of $\bm{\zeta}$ is lost.

\begin{theorem}
	\label{thm:zeta}
	If $\mue,\mumi>0$ and $\Lc=0$ \cref{eq:weak} has a unique solution $\{u,\bm{\zeta}\} \in \Honez(\Omega)\times[\Le(\Omega)]^2$. Further, if the right-hand side $\bm{\omega} = \nabla r$ is a gradient field with $r\in\Hone(\Omega)$, the microdistortion $\bm{\zeta}$ results in a gradient field $\bm{\zeta} = \nabla \chi$ with $\chi\in\Hone(\Omega)$. Especially, there holds the regularity result $\bm{\zeta}\in\Hc,\Omega)$.
\end{theorem}
\begin{proof}
	The proof of existence and uniqueness follows exactly the same lines as the proof of \cref{thm:primal_hcurl_solve}. If $\bm{\omega} = \nabla r$ we can conclude as in the proof of \cref{thm:omega_gradfield} that $\bm{\zeta}$ is a gradient field.
\end{proof}
\begin{remark}
	Using \cref{thm:omega_gradfield} and assuming $\bm{\omega}=0$, we can reformulate \cref{eq:strong2} to retrieve $\bm{\zeta}$ from the known field $u$ 
	\begin{align}
		\bm{\zeta} = \nabla \chi = \dfrac{\mue}{\mumi + \mue} \nabla u \, .
	\end{align}
	Furthermore, we can condensate \cref{eq:strong1} into the Poisson equation 
	\begin{gather}
	-\di \left (\dfrac{2 \mue \, \mumi}{\mue + \mumi} \nabla u \right ) = \underbrace{\left (\dfrac{ -2 \mue \, \mumi}{\mue + \mumi} \right )}_{= -2 \,\muma} \Delta u = - 2 \, \muma \, \Delta u = f \, , \label{eq:lap}
	\end{gather}
	where the homogenization of the material constants follows as in \cite{Neff2019}.
	We notice, that \cref{thm:omega_gradfield} and \cref{thm:zeta} imply the field $u$ is always independent of the microdistortion $\bm{\zeta}$ in this setting. \AS{In the condensed state, the relation of the model with antiplane shear for membranes is apparent.}
	\label{rem:membrane}
\end{remark}

\begin{remark}
We note that the previous result does not hold in the full three-dimensional relaxed micromorphic continuum, i.e. the absence of external moments does not automatically imply $\bm{P} = \nabla \bm{\chi}$ for $\bm{\chi} \in [\Hone(\Omega)]^3$.
\end{remark}

Having considered the limit of the characteristic length $\Lc \to 0$, we reformulate \cref{eq:weak} as an equivalent mixed formulation in order to examine its limit for $\Lc \to \infty$. We start by introducing the new variable 
\begin{align}
	m = \muma\,\Lc^2 \curl \bm{\zeta} \,\in \Lez(\Omega) ,\label{eq:def_p}
\end{align}
and constructing a new bilinear form by multiplying it with a test function
\begin{align}
	\int_{\Omega} \langle \curl \bm{\zeta} , \delta m \rangle - \dfrac{1}{\muma\,\Lc^2}\langle m , \delta m \rangle  \,\dd X = 0 \,  \quad \text{ for all } \delta m \in \Lez(\Omega) \, .
\end{align} 
The restriction to $m \in \Lez(\Omega)$ follows from the Stoke's theorem
	\begin{align}
	\int_{\Omega} \curl \bm{\zeta} \, \dd X = \oint_{\partial \Omega} \langle \bm{\zeta} ,\, \bm{\tau} \rangle \, \dd s = 0 \,  \quad \text{ for all } \bm{\zeta} \in \Hcz,\Omega) \, .
	\label{eq:stokes}
	\end{align}
We introduce the (bi-)linear forms
\begin{subequations}
\begin{align}
	a(\{u,\bm{\zeta}\},\{\delta u,\delta \bm{\zeta}\}) &= \int_{\Omega} 2\mue \langle (\nabla u - \bm{\zeta}) , (\nabla \delta u -\delta \bm{\zeta}) \rangle + 2\mumi \,  \langle \bm{\zeta} ,  \delta \bm{\zeta} \rangle \,\dd X \, , \\
	b(\{u,\bm{\zeta}\}, \delta m) &= \int_{\Omega} \langle \curl\bm{\zeta} , \delta m \rangle\, \dd X \, , \\
	c(m,\delta m) &= \int_{\Omega} \langle m , \delta m \rangle\, \dd X \, , \\
	d(\{\delta u, \delta \bm{\zeta}\}) &=  \int_{\Omega}   \langle \delta u , \, f \rangle + \langle \delta \bm{\zeta} , \, \bm{\omega} \rangle \, \dd X \, ,
\end{align}
\end{subequations}
and the resulting mixed formulation reads: find $(\{u,\bm{\zeta}\}, m)\in X\times \Lez(\Omega)$ such that
\begin{subequations}
\label{eq:mixed_problem}
\begin{align}
	 a(\{u,\bm{\zeta}\},\{\delta u,\delta \bm{\zeta}\}) +  b(\{\delta u, \delta \bm{\zeta}\},  m) &= d(\{\delta u, \delta \bm{\zeta}\}) \,  
	&&\text{ for all } \{\delta u,\delta \bm{\zeta}\} \in \X \, ,
	\label{eq:mixed_problem_a} \\
	 b(\{ u, \bm{\zeta}\}, \delta m)- \dfrac{1}{\muma\,\Lc^2}c(m,\delta m)  &= 0 \,  &&\text{ for all } \delta m \in \Lez(\Omega) \, , \label{eq:mixed_problem_b} 
\end{align}
\end{subequations}
where the Lagrange multiplier $m$ has the physical meaning of a moment stress tensor.

The limit case $\lim \Lc \to \infty$ of \cref{eq:mixed_problem} is well-defined, resulting in the problem: Find $(\{u_{\infty},\bm{\zeta}_{\infty}\},m_{\infty})\in X\times \Lez(\Omega)$ such that
\begin{subequations}
	\label{eq:mixed_problem_limit}
	\begin{align}
	a(\{u_\infty,\bm{\zeta}_\infty\},\{\delta u,\delta \bm{\zeta}\}) +  b(\{\delta u, \delta \bm{\zeta}\},  m_\infty) &= d(\{\delta u, \delta \bm{\zeta}\}) \,  
	&&\text{ for all } \{\delta u,\delta \bm{\zeta}\} \in \X \, ,\label{eq:mixed_problem_limit_a}
	\\
	b(\{ u_\infty, \bm{\zeta}_\infty\}, \delta m)  &= 0 \,  &&\text{ for all } \delta m \in \Lez(\Omega) \, . \label{eq:mixed_problem_limit_b}
	\end{align}
\end{subequations}
Consequently, at the limit $\lim \Lc \to \infty$ we have $\curl \bm{\zeta} = 0$.

We now show existence and uniqueness of both mixed problems and that in the limit case $\Lc\to\infty$ the solution of \cref{eq:mixed_problem} converges to the solution of \cref{eq:mixed_problem_limit} with quadratic convergence rate in $\Lc$.
\begin{theorem}
	For $\mue, \mumi , \muma, \Lc > 0$ \cref{eq:mixed_problem} has a unique solution $(\{u,\bm{\zeta}\},m)\in X\times \Lez(\Omega)$ satisfying for $(\muma\Lc^2)^{-1}\leq 1$ the stability estimate
	\begin{align}
	\|\{ u,\bm{\zeta} \}\|_{\X}+\|m\|_{\Le} \leq c_1 \Big(\|f\|_{\Le}+\|\bm{\omega}\|_{\Le}\Big) \, ,\label{eq:mixed_method_stab_est}
	\end{align}
	where $c_1$ is independent of $\Lc$. Further, let $(\{u_{\infty},\bm{\zeta}_{\infty}\},m_{\infty})\in X\times \Lez(\Omega)$ be the unique solution of \cref{eq:mixed_problem_limit}. Then, we have the estimate
	\begin{align}
	\| \{u_\infty - u,\bm{\zeta}_\infty\ - \bm{\zeta} \} \|_{\X} + \| m_\infty - m \|_{\Le} \leq \dfrac{c_2}{\Lc^2} \Big(\|f\|_{\Le}+\|\bm{\omega}\|_{\Le}\Big) \, ,\label{eq:mixed_method_conv_Lc}
	\end{align}
	where $c_2$ does not depend on $\Lc$.
\end{theorem}
\begin{proof}
Existence and uniqueness follows from the extended Brezzi theorem \cite[Thm. 4.11]{Bra2013}. The continuity of $a(\cdot,\cdot)$, $b(\cdot,\cdot)$, $c(\cdot,\cdot)$ and non-negativity of $a(\cdot,\cdot)$ and $c(\cdot,\cdot)$ are obvious. Therefore, we have to prove that $a(\cdot,\cdot)$ is coercive on the kernel of $b(\cdot,\cdot)$
\begin{align}
	\ker (b) = \Big \{ \{u, \bm{\zeta} \} \in \X \, | \, b(\{u, \bm{\zeta} \}, \delta m) = 0 \,\text{ for all } \delta m \in \Lez(\Omega) \Big \} = \Big \{ \{u, \bm{\zeta} \} \in \X \, | \, \curl \bm{\zeta} = 0 \Big \} \, .
\end{align}
However, we already know from \cref{thm:primal_hcurl_solve} that $a(\{u,\bm{\zeta}\},\{\delta u,\delta\bm{\zeta}\}) + \int_{\Omega}\langle\curl\bm{\zeta},\curl\delta\bm{\zeta}\rangle\,\dd X$ is coercive. This leaves us with the Ladyzhenskaya--Babu{\v{s}}ka--Brezzi (LBB) condition to be satisfied
\begin{align}
	\exists \, \beta_2 > 0 : \quad \sup_{\{u, \bm{\zeta}\} \in \X} \dfrac{b(\{u, \bm{\zeta}\},m)}{ \| \{u, \bm{\zeta}\} \|_{\X}} \geq \beta_2 \, \|m\|_{\Le} \,  \quad \text{ for all } m \in \Lez(\Omega) \, .
\end{align}
We choose $u=0$ and $\bm{\zeta}$ such that $\curl \bm{\zeta}=m$ with $\|\bm{\zeta}\|_{\Le}\leq c\|m\|_{\Le}$ leading to
\begin{align}
\dfrac{b(\{u, \bm{\zeta}\},m)}{ \| \{u, \bm{\zeta}\} \|_{\X}} &= \dfrac{\int_\Omega \langle m , \, \curl\bm{\zeta} \rangle \,\dd X}{ \| \bm{\zeta}\|_{\Le}+\| \curl\bm{\zeta}\|_{\Le}}\geq c \, \dfrac{\|m\|^2_{\Le}}{\|m\|_{\Le}} = c\,\|m\|_{\Le} \, ,
\end{align}
where the construction of $\bm{\zeta}$ is according to \cite{LS15}\footnote{The construction is derived directly from the 2D Stokes LBB condition with $H(\di)$-conforming elements and applies here since the $\curl$ operator is a rotated divergence operator in two dimensions.}. 
Thus, there exists a unique solution independent of $\Lc$ satisfying the stability estimate \cref{eq:mixed_method_stab_est}.\\

With the (classical) Brezzi-Theorem also the existence and uniqueness of \cref{eq:mixed_problem_limit} follows immediately and estimate \cref{eq:mixed_method_conv_Lc} due to the continuous dependence of the solution with respect to the parameter $\Lc$, \cite[Cor. 4.15]{Bra2013}.
\end{proof}

\begin{remark}
As mentioned in \cite{LS15} the space for $m$ must be chosen as $\Lez(\Omega)$, where its mean is zero, if Dirichlet data are prescribed on the whole boundary $\Gamma_D^{\zeta}=\partial \Omega$. This follows from \cref{eq:stokes} 
\begin{align}
\int_\Omega m \,\dd X =\muma\Lc^2\int_\Omega \curl \bm{\zeta}\,\dd X = \muma\Lc^2\int_{\partial \Omega} \langle \bm{\zeta} , \, \bm{\tau} \rangle \,\dd s = 0 \,  \qquad \text{ for all } \bm{\zeta}\in \Hcz,\Omega).
\end{align}
If also Neumann data is prescribed for $\bm{\zeta}$, the appropriate function space for $m$ is $\Le(\Omega)$.
\end{remark}
\begin{remark}
	In the full micromorphic continuum, where the gradient takes the place of the curl of the microdistortion
	\begin{align}
	\underbrace{\int_{\Omega} 2\mue \langle (\nabla u - \bm{\zeta}) , \, (\nabla \delta u -\delta \bm{\zeta}) \rangle + 2\mumi \,  \langle \bm{\zeta} , \, \delta \bm{\zeta} \rangle + \muma \, \Lc ^2 \, \langle \nabla\bm{\zeta} , \, \nabla\delta \bm{\zeta} \rangle \, \dd X}_{\displaystyle =a_\mathrm{grad}(\{u,\bm{\zeta}\},\{\delta u,\delta\bm{\zeta}\})} = 
	\int_{\Omega}   \langle \delta u , \, f \rangle + \langle \delta \bm{\zeta} , \, \bm{\omega} \rangle \, \dd X \, ,
	\label[Problem]{eq:weak_full}
	\end{align}
	existence and uniqueness follow similarly with the space $\mathit{X} = \Hone(\Omega) \times [\Hone(\Omega)]^2$. However, the limit case $\Lc \to \infty$ yields $\nabla \bm{\zeta} = 0$ and consequently $\bm{\zeta} = const.$, for which non-trivial boundary conditions cannot be considered, compare also \cref{subsec:num_ex_cons} for a numerical example.
	\label{rem:full}
\end{remark}

To conclude this section we investigate the necessary and sufficient conditions such that in the limit $\Lc\to\infty$ the solution satisfies $\nabla u = \bm{\zeta}$. This state represents a zoom into the microstructure in the three-dimensional theory with microscopic stiffness given by $\mumi$ \cite{Neff2019}.
In \cref{thm:omega_gradfield} we found sufficient conditions to obtain a gradient field for the microdistortion, which, however, does not have to be $\nabla u$. The following theorem states that only for a zero right-hand side $f$, but arbitrary $\vb{\omega}$, the desired behaviour is achieved.
\begin{theorem}
	\label{thm:limit_gradu_zeta}
	Let $\Omega$ be simply connected and $\Gamma^u_D=\Gamma_D^{\zeta}=\partial\Omega$. Then there holds for the solution $\{u,\bm{\zeta}\}\in X$ of \cref{eq:weak}
	\begin{align}
	\|\bm{\zeta}-\nabla u\|_{\Hc)}\leq \frac{c}{\Lc^2}\,,
	\end{align}
	if and only if $f=0$, where $c$ does not depend on $\Lc$.
\end{theorem}
\begin{proof}
	From the limit solution $\{u_{\infty},\bm{\zeta}_{\infty}\}\in X$ of \cref{eq:mixed_problem_limit} we have that $\bm{\zeta}_{\infty}\in\Hcz,\Omega)$ and $\curl \bm{\zeta}_{\infty}=0$. This implies the existence of $\Psi\in \Honez(\Omega)$ such that $\bm{\zeta}=\nabla \Psi \in \ker(\curl)$. Inserting this into \cref{eq:mixed_problem_limit_a}, where $\delta\bm{\zeta}=0$ is chosen, yields
	\begin{align*}
	\int_{\Omega} 2\mue \langle \nabla u - \nabla \Psi ,\nabla \delta u  \rangle  \, \dd X = \int_{\Omega} \langle \delta u , \, f \rangle\ \, \dd X \, \quad \text{ for all }\delta u\in \Honez(\Omega). 
	\end{align*}
	Thus, $u=\Psi\in\Honez(\Omega)$ is the unique solution if and only if $f=0$ and correspondingly $\{u_{\infty},\bm{\zeta}_{\infty}\}=\{\Psi,\nabla\Psi\}$. The claim follows with the triangle inequality, \cref{eq:mixed_method_conv_Lc} and the equivalence of the mixed and primal problem
	
	\vspace{4mm}
	\hspace{12mm} $\|\bm{\zeta}-\nabla u\|_{\Hc)}\leq \|\bm{\zeta}-\bm{\zeta}_{\infty}\|_{\Hc)}+\|\underbrace{\bm{\zeta}_{\infty}-\nabla u_\infty}_{=0}\|_{\Hc)}+\|\nabla u_{\infty}-\nabla u\|_{\Le}\leq \dfrac{c}{\Lc^2} \, . $
\end{proof}
\begin{remark}
	We can weaken the assumptions of \cref{thm:limit_gradu_zeta} to $\Gamma_D^u=\Gamma_D^\zeta\neq\emptyset$. Further, also non-homogeneous Dirichlet data can be considered, provided the consistent coupling condition $\langle \bm{\zeta} , \, \bm{\tau} \rangle= \langle \nabla \widetilde{u} , \, \bm{\tau}\rangle$ on $\Gamma_D^{\zeta}$ holds. 
	
	From the proof of \cref{thm:limit_gradu_zeta} we obtain from the existence of a potential such that $\bm{\zeta}=\nabla\Psi$. Thus, $\bm{\zeta}$ is expected to be in $\Hc,\Omega)$ as in general $\nabla\Psi\notin [\Hone(\Omega)]^2$ for $\Psi\in\Hone(\Omega)$.
\end{remark}

\subsection{Discrete case}\label{ch:3}
Motivated by the de' Rham complex (see \cref{fig:deRham}) 
\begin{figure}
	\centering
	\begin{tikzpicture}[scale = 0.6][line cap=round,line join=round,>=triangle 45,x=1.0cm,y=1.0cm]
	\clip(3.4,3.5) rectangle (16,10);
	\draw (4.011829363359278,9) node[anchor=north west] {$\mathit{H}^1$};
	\draw (9.019960885184046,9) node[anchor=north west] {$\mathit{H} (\text{curl})$};
	\draw (9.019960885184046,4.5) node[anchor=north west] {$\mathit{U}^h$};
	\draw (4.033698933323579,4.5) node[anchor=north west] {$\mathit{V}^h$};
	\draw (14.531092516187721,9) node[anchor=north west] {$\mathit{L}^2$};
	\draw (14.531092516187721,4.5) node[anchor=north west] {$\mathit{Q}^h$};
	\draw [->,line width=1.5pt] (5.5,8.5) -- (8.5,8.5);
	\draw [->,line width=1.5pt] (4.5,7.5) -- (4.5,5.);
	\draw [->,line width=1.5pt] (5.5,4.) -- (8.5,4.);
	\draw [->,line width=1.5pt] (9.5,7.5) -- (9.5,5.);
	\draw [->,line width=1.5pt] (11.5,8.5) -- (14.,8.5);
	\draw [->,line width=1.5pt] (15.,7.5) -- (15.,5.);
	\draw [->,line width=1.5pt] (10.5,4.) -- (14.,4.);
	\draw (3.2306988241446713,6.8) node[anchor=north west] {$\Pi_g$};
	\draw (8.216960776005138,6.8) node[anchor=north west] {$\Pi_c$};
	\draw (13.728092407008813,6.8) node[anchor=north west] {$\Pi_0$};
	\draw (6.023829800074905,9.54081254265689) node[anchor=north west] {$\nabla$};
	\draw (12.016091970293186,9.54081254265689) node[anchor=north west] {$\text{curl }$};
	\draw (6.023829800074905,5.035681130011) node[anchor=north west] {$\nabla$};
	\draw (12.016091970293186,5.035681130011) node[anchor=north west] {$\text{curl }$};
	\end{tikzpicture}
	\caption{The de' Rham complex in two dimensions depicting Hilbert spaces and approximation spaces connected by differential and interpolation operators. The kernel of one differential operator is exactly the range of the previous differential operator on its space and the differential and projection operators commute.}
	\label{fig:deRham}
\end{figure}
we formulate a finite element combining base functions from both $\Hone(\Omega)$ and $\Hc,\Omega)$ (and $\Le(\Omega)$ for the mixed formulation) setting
\begin{gather}
u^h, \, \delta u^h \in \V^h \subset \Hone(\Omega) \, , \quad \bm{\zeta}^h , \, \delta \bm{\zeta}^h \in \U^h \subset \Hc, \Omega) \, , \quad m^h ,\delta m^h \in \Q^h \subset \Le(\Omega) \, .
\end{gather}
Throughout this work we will use meshes consisting of quadrilaterals. On each element we denote the set of quadrilateral polynomials by $\Q^{n,m}=\text{span}\{ x^ky^j\,|\, 0\leq k\leq n,\, 0\leq j\leq m\}$, compare also \cref{eq:Q}, and further the set of N\'{e}d\'{e}lec ansatz functions by
\begin{align}
\Pe^k = \begin{bmatrix} \Q^{k-1,k} \\ \Q^{k,k-1}\end{bmatrix}\,.\label{eq:Pe}
\end{align}

We start with the Lax--Milgram setting by defining $\X^h=V^h\times U^h$. We note that solvability of the discretized problem follows directly from the continuous one as $\X^h\subset \X$. Using Cea's lemma for the quasi-best approximation
\begin{align}
	\|\{ u,\bm{\zeta} \} - \{ u^h, \bm{\zeta}^h \} \|_{\X} \leq \dfrac{\alpha}{\beta}\inf\limits_{\{\delta u^h,\delta\bm{\zeta}^h\}\in \X^h} \| \{u, \bm{\zeta}\} - \{\delta u^h, \delta \bm{\zeta}^h\} \|_{\X} \, ,
\end{align}
we can generate convergence estimates a priori. 
\begin{lemma}
	Assume a smooth exact solution $\{u,\bm{\zeta}\}\in X$. Further, if on each element $\Q^{k,k}\subset V^h$ and $\Pe^k\subset U^h$, then the discrete solution $\{u^h, \bm{\zeta}^h\}\in X^h$ converges with the optimal convergence rate
	\begin{align}
	\|\{ u,\bm{\zeta} \} - \{ u^h, \bm{\zeta}^h \} \|_{\X} \leq c(\Lc^2, \mue, \mumi, \muma) \, h^k \, .\label{eq:conv_rate_primal}
	\end{align}
\end{lemma}
\begin{proof}
By inserting the interpolation operators associated through the commuting diagram we find 
\begin{align}
	\|\{ u,\bm{\zeta} \} - \{ u^h, \bm{\zeta}^h \} \|_{\X}^2 & \leq c \inf\limits_{\{\delta u^h,\delta\bm{\zeta}^h\}\in \X^h} \| \{u, \bm{\zeta}\} - \{\delta u^h, \delta \bm{\zeta}^h\} \|_{\X}^2 \notag \\
	& \leq c \, \Big ( \|u - \Pi_g u \|_{\Hone}^2 + \|\bm{\zeta} - \Pi_c \bm{\zeta} \|_{\Le}^2 + \| \curl\bm{\zeta} - \curl\Pi_c \bm{\zeta} \|_{\Le}^2 \Big ) \notag \\
	& = c \, \Big ( \|u - \Pi_g u \|_{\Hone}^2 + \|\bm{\zeta} - \Pi_c \bm{\zeta} \|_{\Le}^2 + \| (\text{id}-\Pi_0) \curl\bm{\zeta}\|_{\Le}^2 \Big ) \notag \\
	& \leq c \, h^{2k} \Big (|u|_{\mathit{H}^{k+1}}^2 + |\bm{\zeta}|_{\mathit{H}^k}^2 + |\curl\bm{\zeta}|_{\mathit{H}^k}^2 \Big ) \, ,
\end{align}
where $|\cdot|_{\mathit{H}^k}$ denotes the standard Sobolev semi-norm.
\end{proof}

Note that the constant $c$ in \cref{eq:conv_rate_primal} depends on $\Lc$. One may prove robust estimates in this setting. We, however, test for robustness with respect to $\Lc$ in the context of mixed methods and use the equivalence of both.

In general the solvability of the discretized mixed problem does not follow from the continuous one. However, thanks to the commuting property of the de' Rham complex, the discrete kernel coercivity and the LBB condition follow immediately. Thus, we obtain the quasi-best approximation error
\begin{align}
\|\{ u,\bm{\zeta} \} - \{ u^h, \bm{\zeta}^h \} \|_{\X} + \|m-m^h\|_{\Le} \leq c\inf\limits_{(\{\delta u^h,\delta\bm{\zeta}^h\},\delta \omega^h)\in \X^h\times Q^h} \Big(\| \{u, \bm{\zeta}\} - \{\delta u^h, \delta \bm{\zeta}^h\} \|_{\X} + \|m-\delta m^h\|_{\Le}\Big),
\end{align}
where $c$ is independent of $\Lc$.
\begin{lemma}
	Assume that the exact solution $(\{u,\bm{\zeta}\},m)\in \X\times\Le(\Omega)$ of \cref{eq:mixed_problem} is smooth and that on each element $\Q^{k,k}\subset V^h$, $\Pe^k\subset U^h$, and $\Q^{k-1,k-1}\subset Q^h$. Then the discrete solution $(\{u^h,\bm{\zeta}^h\},m^h)\in X_h\times Q_h$ satisfies the optimal convergence rate independent of $\Lc$ 
	\begin{align}
	\|\{ u,\bm{\zeta} \} - \{ u^h, \bm{\zeta}^h \} \|_{\X} + \|m-m^h\|_{\Le} \leq c \, h^k \, .\label{eq:mixed_method_conv_rate}
	\end{align} 
	Additionally, with $\{u_{\infty},\bm{\zeta}_{\infty}\}$ the (smooth) solution of the limit problem we obtain
	\begin{align}
	\|\{ u_{\infty},\bm{\zeta}_{\infty} \} - \{ u^h, \bm{\zeta}^h \} \|_{\X} + \|m_{\infty}-m^h\|_{\Le} \leq \dfrac{c_1}{\Lc^2} + c_2 \, h^k\,.\label{eq:conv_lc_h}
	\end{align}
\end{lemma}
\begin{proof}
	Using the interpolation operators $\Pi_g$, $\Pi_c$, and $\Pi_0$ gives estimate \cref{eq:mixed_method_conv_rate}. Inequality \cref{eq:conv_lc_h} follows immediately by adding and subtracting the solution of the corresponding continuous solution $(\{u,\bm{\zeta}\},m)\in \X\times\Le(\Omega)$ for a fixed $\Lc$, using triangle inequality, \cref{eq:mixed_method_conv_Lc} and \cref{eq:mixed_method_conv_rate}.
\end{proof}

Inequality \cref{eq:conv_lc_h} states that, as long as the discretization error is not reached, we have quadratic convergence to the limit case $\lim \Lc\to\infty$. Due to the equivalence of the primal formulation \cref{eq:weak} and the mixed \cref{eq:mixed_problem} we can deduce that the solution of \cref{eq:weak} is also robust with respect to $\Lc$. As we will see in the numerical examples, the mixed formulation is better suited for extremely large values of $\Lc$ due to rounding errors.

\section{Finite element formulations} \label{ch:4}
\subsection{Appropriate base functions}
In the following we demonstrate the construction of the hybrid element in the linear case. The finite elements for the mixed formulation are employed directly using the open source finite element library NETGEN/NGSolve\footnote{www.ngsolve.org} \cite{Sch1997,Sch2014}. 

For the mapping of $x$ and $y$, see \cref{fig:amp}, we make use of linear quadrilateral Lagrange nodal base functions
\begin{gather}
\begin{align}
&N_1(\xi,\eta) = \dfrac{1}{4}(\xi - 1)(\eta-1) \, ,  \qquad
N_2(\xi,\eta) = \dfrac{1}{4}(\xi + 1)(1-\eta) \, , \nonumber \\
&N_3(\xi,\eta) = \dfrac{1}{4}(\xi + 1)(\eta+1) \, ,  \qquad
N_4(\xi,\eta) = \dfrac{1}{4}(1-\xi)(\eta+1) \, ,
\end{align} \\
x  = \bigcup^{n}_{e=1} \underbrace{\begin{bmatrix}
	N_1 & N_2 & N_3 & N_4
	\end{bmatrix} }_{\displaystyle \bm{H}(\xi,\eta)} \underbrace{\begin{bmatrix}
	x_{1} \\
	x_{2} \\
	x_{3} \\
	x_{4}
	\end{bmatrix}}_{\displaystyle \bar{\vb{x}}_e} \, , \qquad
y = \bigcup^{n}_{e=1} \bm{H}\vb{y}_e \, , \qquad  \vb{x} = \begin{bmatrix}
x & y
\end{bmatrix}^T \, ,
\label{eq:H}
\end{gather}
where $n$ is the number of finite elements in the mesh. As shown in \cref{fig:amp}, the elements are mapped via
\begin{gather}
\vb{x} : \Xi \mapsto \Omega \, , \qquad \Xi = [-1,1] \times [-1,1]  \, , \qquad \Omega = \bigcup^{n}_{e=1}  \Omega_e \subset \mathbb{R}^2 \, .
\end{gather}
\begin{figure}
	\centering
	\definecolor{xfqqff}{rgb}{0.4980392156862745,0.,1.}
	\definecolor{qqwwzz}{rgb}{0.,0.4,0.6}
	\begin{tikzpicture}[scale = 0.5][line cap=round,line join=round,>=triangle 45,x=1.0cm,y=1.0cm]
	\clip(-2,-2.5) rectangle (17,7);
	\fill[line width=0.4pt,color=qqwwzz,fill=qqwwzz,fill opacity=0.10000000149011612] (0.,0.) -- (4.,0.) -- (4.,4.) -- (0.,4.) -- cycle;
	\fill[line width=0.4pt,color=qqwwzz,fill=qqwwzz,fill opacity=0.10000000149011612] (12.,0.) -- (16.,2.) -- (14.,6.) -- (11.,4.) -- cycle;
	\draw [line width=0.4pt,color=qqwwzz] (0.,0.)-- (4.,0.);
	\draw [line width=0.4pt,color=qqwwzz] (4.,0.)-- (4.,4.);
	\draw [line width=0.4pt,color=qqwwzz] (4.,4.)-- (0.,4.);
	\draw [line width=0.4pt,color=qqwwzz] (0.,4.)-- (0.,0.);
	\draw [->,line width=1.2pt] (2.,4.) -- (2.,6.);
	\draw [->,line width=1.2pt] (4.,2.) -- (6.,2.);
	\draw [line width=0.4pt] (0.,-1.)-- (4.,-1.);
	\draw [line width=0.4pt] (-1.,0.)-- (-1.,4.);
	\draw [->,line width=1.2pt] (10.,-2.) -- (10.,0.);
	\draw [->,line width=1.2pt] (10.,-2.) -- (12.,-2.);
	\draw [line width=0.4pt,color=qqwwzz] (12.,0.)-- (16.,2.);
	\draw [line width=0.4pt,color=qqwwzz] (16.,2.)-- (14.,6.);
	\draw [line width=0.4pt,color=qqwwzz] (14.,6.)-- (11.,4.);
	\draw [line width=0.4pt,color=qqwwzz] (11.,4.)-- (12.,0.);
	\draw [->,line width=1.2pt,color=xfqqff] (3.,4.) -- (3.,5.);
	\draw [->,line width=1.2pt,color=xfqqff] (4.,0.5) -- (4.,1.5);
	\draw [->,line width=1.2pt,color=xfqqff] (13.012220419365061,0.5061102096825305) -- (14.038157043516751,1.0190785217583755);
	\draw [->,line width=1.2pt,color=xfqqff] (14.62875816842749,4.74248366314502) -- (15.591685024265983,5.223947091064266);
	\draw [color=qqwwzz](0.3,1.4060766687904682) node[anchor=north west] {$\Xi$};
	\draw [color=qqwwzz](12.775319131797966,3.5816191348674904) node[anchor=north west] {$\Omega_e$};
	\draw [color=xfqqff](4.3,1.2) node[anchor=north west] {$\boldsymbol{\varsigma}$};
	\draw [color=xfqqff](2.5,5.9) node[anchor=north west] {$\boldsymbol{\varrho}$};
	\draw (6.184705190446985,2.493847901828979) node[anchor=north west] {$\xi$};
	\draw (1.5,7) node[anchor=north west] {$\eta$};
	\draw (3.55,-1.2) node[anchor=north west] {$1$};
	\draw (-2.5,0.5) node[anchor=north west] {$-1$};
	\draw (-1.9415857857818923,4.413444195426352) node[anchor=north west] {$1$};
	\draw [color=xfqqff](13.591147556576848,0.6542347871314972) node[anchor=north west] {$\boldsymbol{\tau}$};
	\draw [color=xfqqff](15.58,5.74) node[anchor=north west] {$\boldsymbol{\nu}$};
	\draw (12.183443607938775,-1.6) node[anchor=north west] {$x$};
	\draw (9.5,0.9581708669510812) node[anchor=north west] {$y$};
	\draw [color=qqwwzz](2.4,0) node[anchor=north west] {$\Sigma_1$};
	\draw [color=qqwwzz](4,3.5) node[anchor=north west] {$\Sigma_2$};
	\draw [color=qqwwzz](0.5,5) node[anchor=north west] {$\Sigma_3$};
	\draw [color=qqwwzz](-1.1,1.7) node[anchor=north west] {$\Sigma_4$};
	\draw [color=qqwwzz](10.3,2.8) node[anchor=north west] {$\Gamma_1$};
	\draw [color=qqwwzz](14.6,1.4) node[anchor=north west] {$\Gamma_2$};
	\draw [color=qqwwzz](15.4,3.8) node[anchor=north west] {$\Gamma_3$};
	\draw [color=qqwwzz](11.7,6) node[anchor=north west] {$\Gamma_4$};
	\draw [shift={(6.,0.5)},line width=2.pt]  plot[domain=0.8621700546672264:1.7894652726688387,variable=\t]({1.*4.609772228646444*cos(\t r)+0.*4.609772228646444*sin(\t r)},{0.*4.609772228646444*cos(\t r)+1.*4.609772228646444*sin(\t r)});
	\draw [->,line width=2.pt] (8.451054785975783,4.4041427274300835) -- (9.,4.);
	\draw (6.3,6.3) node[anchor=north west] {$\mathbf{x}(\xi,\eta)$};
	\draw (-1,-1.2) node[anchor=north west] {$-1$};
	\draw [line width=1.2pt] (2.,2.)-- (4.,2.);
	\draw [line width=1.2pt] (2.,2.)-- (2.,4.);
	\begin{scriptsize}
	\draw [fill=qqwwzz] (0.,0.) circle (2.5pt);
	\draw [fill=qqwwzz] (4.,0.) circle (2.5pt);
	\draw [fill=qqwwzz] (4.,4.) circle (2.5pt);
	\draw [fill=qqwwzz] (0.,4.) circle (2.5pt);
	\draw [color=black] (0.,-1.)-- ++(-4.5pt,0 pt) -- ++(9.0pt,0 pt) ++(-4.5pt,-4.5pt) -- ++(0 pt,9.0pt);
	\draw [color=black] (4.,-1.)-- ++(-4.5pt,0 pt) -- ++(9.0pt,0 pt) ++(-4.5pt,-4.5pt) -- ++(0 pt,9.0pt);
	\draw [color=black] (-1.,0.)-- ++(-4.5pt,0 pt) -- ++(9.0pt,0 pt) ++(-4.5pt,-4.5pt) -- ++(0 pt,9.0pt);
	\draw [color=black] (-1.,4.)-- ++(-4.5pt,0 pt) -- ++(9.0pt,0 pt) ++(-4.5pt,-4.5pt) -- ++(0 pt,9.0pt);
	\draw [fill=black] (10.,-2.) circle (2.5pt);
	\draw [fill=qqwwzz] (12.,0.) circle (2.5pt);
	\draw [fill=qqwwzz] (16.,2.) circle (2.5pt);
	\draw [fill=qqwwzz] (14.,6.) circle (2.5pt);
	\draw [fill=qqwwzz] (11.,4.) circle (2.5pt);
	\draw [fill=black] (2.,2.) circle (2.5pt);
	\end{scriptsize}
	\end{tikzpicture}
	\caption{Element mapping from the parametric space into the physical space.}
	\label{fig:amp}
\end{figure}
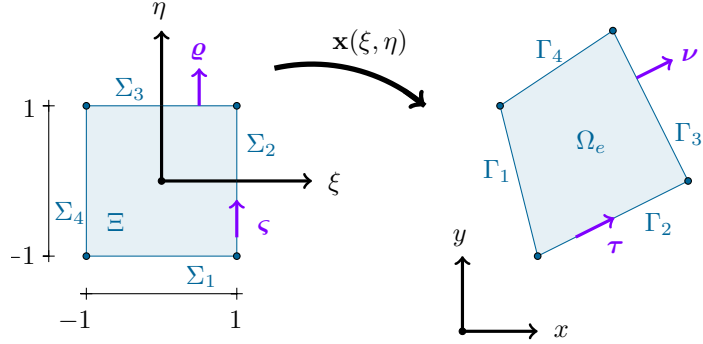
We approximate $u$ according to the isoparametric concept
\begin{gather}
u_e^h = \bm{H}\bar{\vb{u}}_e \, , \qquad u^h = \bigcup^{n}_{e=1} u_e^h  \, .
\label{eq:u}
\end{gather}

However, for $\bm{\zeta}$ we make use of linear N\'{e}d\'{e}lec base functions of the first type for quadrilaterals \cite{Anjam2015, Meunier2010, Nedelec1980, Zaglmayr2006}.
These functions are built around approximations of the curl operator. The corresponding spaces are those of quadrilateral polynomials 
\begin{gather}
p(\xi,\eta) = \Bigg ( \sum_{k=0}^{n} c_k \xi^k \Bigg ) \Bigg ( \sum_{j=0}^{m} d_j \eta^j \Bigg ) \, \in \Q^{n,m} \, .
\label{eq:Q}
\end{gather}
The weak form of the curl in the 2D space is formulated via Greens' formula\footnote{$  \curl (q \, \bm{\zeta}) = \di (\bm{R} \, (q\, \bm{\zeta})) = q \, \curl \bm{\zeta} - \langle \bm{\zeta} , \, \Dc q \rangle  \, , \quad \bm{R} = \begin{bmatrix}
		0 & 1 \\ -1 & 0
		\end{bmatrix} \, . $}
\begin{gather}
\begin{align}
\int_{\Omega}  q  \, \curl\bm{\zeta}  \,\dd X = \oint_{\partial \Omega}  \langle q \, \bm{\zeta}  , \, \bm{\tau} \rangle \,\dd s  
 + \int_{\Omega} \langle \bm{\zeta} , \,  \Dc q \rangle \,\dd X \,  \quad \text{ for all } q \in \C^1 (\Omega, \, \mathbb{R}) \, .
\end{align}
\end{gather}
Therefore, the curl in $\Omega$ is fully determined by its interface and inner rotation field. Consequently, we can decompose the two terms, such that the elements' dofs determine the interpolated field completely. This can be confirmed by setting all dofs to zero, checking for a vanishing field. The corresponding dofs and degrees of the polynomial spaces have been defined by N\'{e}d\'{e}lec \cite{Nedelec1980}.
The element's boundary has been decomposed as ${\partial \Xi =   \Sigma_1\cup\Sigma_2\cup\Sigma_3\cup\Sigma_4 }$. The dofs read
\begin{align}
&4k \text{ edge dofs: } \; \; \quad \qquad f_{ij}(\bm{\vartheta}) = \int_{\Sigma_j}   q_i \, \langle \, \bm{\vartheta} , \, \bm{\varsigma}_j \rangle \,\dd \Sigma \, , &&\bm{\vartheta} \in \Pe^k(\Xi) \,  \quad \text{ for all } q_i \in \mathbb{P}^{k-1}(\Sigma_j) \, , \\[2ex]
&2k(k-1) \text{ cell dofs: } \quad f_{i}(\bm{\vartheta}) = \int_{\Xi} \langle \bm{\vartheta} , \, \vb{q}_i \rangle \,\dd \Xi \, , &&\bm{\vartheta} \in \Pe^k(\Xi) \,  \quad \text{ for all } \vb{q}_i = \begin{bmatrix}
q_1 \\ q_2
\end{bmatrix} \, , \begin{matrix}
q_1 \in \Q^{k-2,k-1}(\Xi) \\ q_2 \in \Q^{k-1,k-2}(\Xi)
\end{matrix} \, , \notag
\end{align}
where $\Pe^k$ and $\Q$ are according to \cref{eq:Pe} and \cref{eq:Q}, and $\mathbb{P}^k$ is the space of polynomials of order $k$.
Since we employ linear N\'{e}d\'{e}lec base functions with $k = 1$, no inner dofs occur.  
The ansatz for the base function reads
\begin{gather}
\bm{\vartheta}_m(\xi, \eta) = \AS{\begin{bmatrix}
		d_0 + d_1 \eta \\
		c_0 + c_1 \xi 
\end{bmatrix}}  \, , \qquad \bm{\vartheta}_m(\xi,\eta) \in \Pe^1(\Xi) \, , \qquad  m = \Big \{1, \,2, \,\dots,\dim(\Pe^1) = 4 \Big \}  \, .
\label{eq:theta}
\end{gather}
Applying the dofs along all edges with the \AS{variable} basis $q_i = 1$
\begin{gather}
f_{ij}(\bm{\vartheta}_m) = \int_{\Sigma_j} q_i \, \langle \bm{\vartheta}_m , \, \bm{\varsigma}_j \rangle \,\dd \Sigma = \delta_{ij} \, ,
\label{eq:dofs}
\end{gather}
we find our base functions
\begin{gather}
\bm{\vartheta}_1 = \dfrac{1}{2} \begin{bmatrix}
1 - \eta \\ 0
\end{bmatrix} \, , \qquad \bm{\vartheta}_2 = \dfrac{1}{2} \begin{bmatrix}
0 \\ 1 + \xi
\end{bmatrix} \, ,  \qquad
\bm{\vartheta}_3 = \dfrac{1}{2} \begin{bmatrix}
- 1 - \eta \\ 0
\end{bmatrix} \, , \qquad \bm{\vartheta}_4 = \dfrac{1}{2} \begin{bmatrix}
0 \\  \xi - 1
\end{bmatrix} \, .
\label{eq:nbase}
\end{gather}
The factor $\nicefrac{1}{2}$ is chosen instead of the resulting $\nicefrac{1}{4}$ as to simplify prescription on the Dirichlet boundary.
The functions are depicted in \cref{fig:basefunctions}.
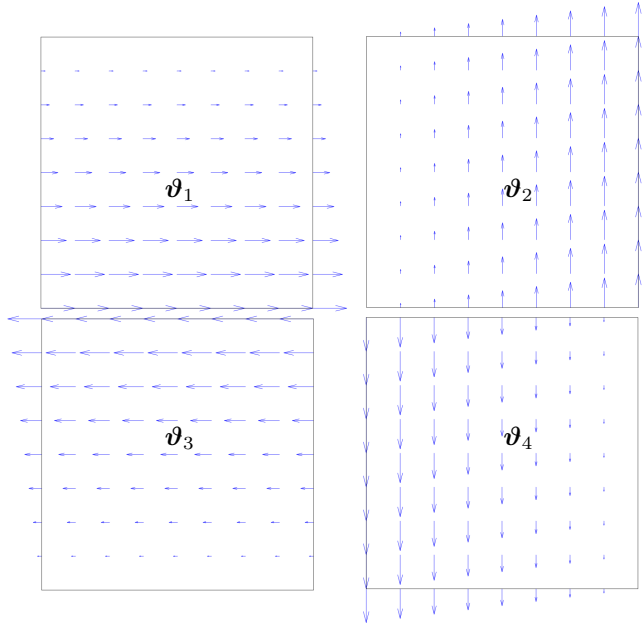
\begin{figure}
		\centering
		\begin{tikzpicture}[line cap=round,line join=round,>=triangle 45,x=1.0cm,y=1.0cm]
		\node[anchor=south west,inner sep=0] at (0,0) {\includegraphics[width=0.5\linewidth]{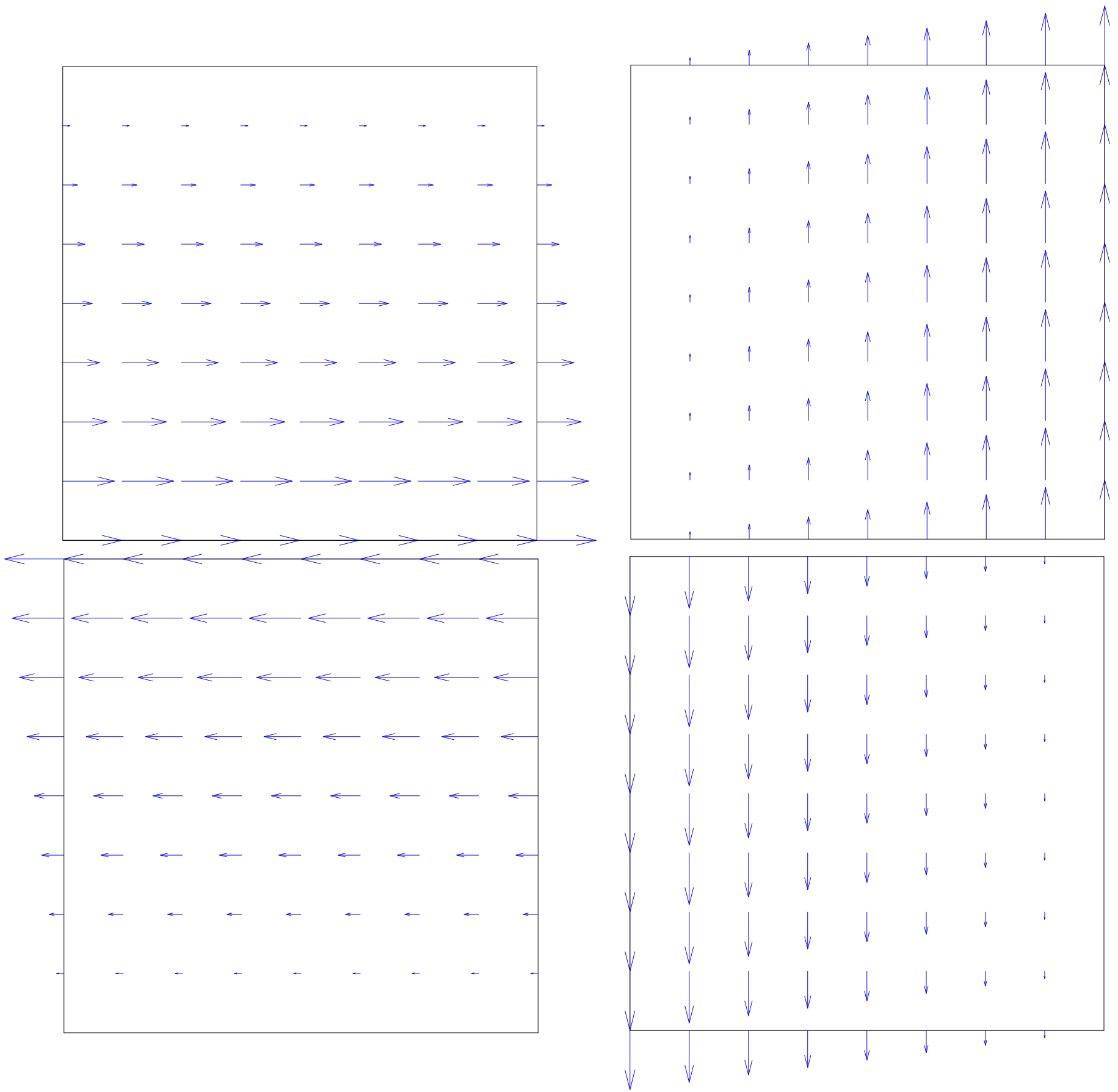}};
		\draw (2,6.05) node[anchor=north west] {$\boldsymbol{\vartheta}_1$};
		\draw (6.5,6.05) node[anchor=north west] {$\boldsymbol{\vartheta}_2$};
		\draw (2,2.75) node[anchor=north west] {$\boldsymbol{\vartheta}_3$};
		\draw (6.5,2.75) node[anchor=north west] {$\boldsymbol{\vartheta}_4$};
		\end{tikzpicture}
		\caption{N\'{e}d\'{e}lec base functions from \cref{eq:nbase} in the parametric space.}
		\label{fig:basefunctions}
\end{figure}

For the mixed formulation involving $m\in\Le(\Omega)$ the corresponding finite element space is given by piece-wise constants, $N_0(\xi,\eta)=1$. To enforce zero mean value, i.e. $m\in\Lez(\Omega)$, a Lagrange multiplier $\lambda\in\mathbb{R}$ has to be used, leading to one additional equation in the final system.

Using higher polynomial orders, we can achieve faster convergence rates and better approximations. The finite element software NGSolve offers the use of hierarchical high order base functions for $\Hone$, $\Hc)$, and $\Le$ spaces \cite{Zaglmayr2006}. We employ NGSolve in our investigation of the mixed formulation with higher order base functions.

\subsection{Covariant Piola transformation}\label{ch:5}
In the previous section we formulated our base functions for the curl in the parametric space. In order to preserve the properties of the base function $\bm{\vartheta}_j$ acting on the curve's tangents $\bm{\varsigma}$ (see \cref{fig:amp}), namely
\begin{gather}
\int_{\Sigma_i} \langle \bm{\vartheta}_j , \, \bm{\varsigma} \rangle \, \dd \Sigma 
= \int_{\Gamma_i} \langle \bm{\theta}_j , \, \bm{\tau} \rangle \, \dd s
= \delta_{ij} \, ,
\end{gather}
where $\bm{\theta}_j$ is the base function in the physical space and $\partial \Omega_e = \Gamma_1 \cup \Gamma_2 \cup \Gamma_3 \cup \Gamma_4$,
the so called covariant Piola transformation is required \cite{Mon03}. The transformation is achieved by considering the push forward of the boundaries' normal vectors
\begin{gather} 
\langle \vb{v} , \, \bm{\nu} \rangle = \det\bm{J}\, \langle \vb{v} , \, \bm{J}^{-T} \bm{\varrho} \rangle \, ,
\label{eq:push}
\end{gather}
where $\bm{J}$ is the Jacobi matrix of the element mappings.
In two dimensions the normal vectors on the element boundary $\bm{\varrho}$ and $\bm{\nu}$ are the $90\degree$ rotation of the tangent vectors given by
\begin{gather}
\bm{\varrho} = \bm{R} \bm{\varsigma} \, , \qquad \bm{\nu} = \bm{R} \bm{\tau}  \, , \qquad \bm{R} = \begin{bmatrix}
0 &  1 \\ -1 & 0
\end{bmatrix} \, .
\label{eq:rot}
\end{gather} 
Using \cref{eq:rot} in \cref{eq:push} results in 
\begin{gather}
 \langle \vb{v} , \, \bm{R} \bm{\tau} \rangle =  \det\bm{J}\, \langle \vb{v} , \, \bm{J}^{-T}  \bm{R} \bm{\varsigma} \rangle  \, ,
\end{gather}
finally yielding the definition of a transformation preserving integration along the tangent
\begin{gather}
\vb{v}_\textrm{0} =  \det\bm{J} \, \bm{R}^T \bm{J}^{-1} \bm{R} \, \vb{v}  \, , \quad \qquad
\vb{v} = \underbrace{\dfrac{1}{\det\bm{J}} \, \bm{R}^T \bm{J}  \bm{R}}_{\displaystyle \bm{J}^{-T}} \vb{v}_\textrm{0} \, .
\label{eq:piola}
\end{gather}
The transformation in \cref{eq:piola} alone cannot guarantee the aligned orientation of base functions on the edges of neighbouring elements \cite{Zaglmayr2006}. In order to achieve conformity we introduce a topological correction function $\psi_j$ based on the global orientation of edges given by node collections as demonstrated in \cref{fig:ori}.
\begin{figure}
	\centering
	\begin{gather}
	\overbrace{\begin{bmatrix}1&2\\2&3\\3&6\\4&1\\4&5\\5&6\\5&2\\3&6\end{bmatrix}}^{\substack{\textrm{Global edge} \\ \textrm{orientation array}}} \quad \begin{matrix}
	\overbrace{\begin{bmatrix}1&2\\2&5\\5&4\\4&1\end{bmatrix}}^{A_1 \textrm{ edge array } \vb{E}_1} \\[10ex]
	\overbrace{\begin{bmatrix}5&2\\2&3\\3&6\\6&5\end{bmatrix}}^{A_2 \textrm{ edge array } \vb{E}_2}
	\end{matrix}
	\qquad
	\begin{matrix}
	\overbrace{\begin{bmatrix}\psi_1=1\\\psi_2=-1\\\psi_3=-1\\\psi_4=1\end{bmatrix}}^{\psi_j(\vb{E}_1)} \\[10ex]
	\overbrace{\begin{bmatrix}\psi_1=1\\\psi_2=1\\\psi_3=1\\\psi_4=-1\end{bmatrix}}^{\psi_j(\vb{E}_2)}
	\end{matrix} \nonumber
	\end{gather}
	\definecolor{qqqqff}{rgb}{0.,0.,1.}
	\definecolor{xfqqff}{rgb}{0.4980392156862745,0.,1.}
	\definecolor{qqwwzz}{rgb}{0.,0.4,0.6}
	\begin{tikzpicture}[scale = 0.95][line cap=round,line join=round,>=triangle 45,x=1.0cm,y=1.0cm]
	\clip(-0.31965533234704546,-5.547856706731499) rectangle (8,10.7674166714165);
	\fill[line width=0.4pt,color=qqwwzz,fill=qqwwzz,fill opacity=0.10000000149011612] (1.,1.) -- (4.,2.) -- (3.,4.) -- (1.,3.) -- cycle;
	\fill[line width=0.4pt,color=qqwwzz,fill=qqwwzz,fill opacity=0.10000000149011612] (3.,4.) -- (6.,4.) -- (6.,1.) -- (4.,2.) -- cycle;
	\fill[line width=0.4pt,color=qqwwzz,fill=qqwwzz,fill opacity=0.10000000149011612] (1.,6.) -- (4.,6.) -- (4.,9.) -- (1.,9.) -- cycle;
	\fill[line width=0.4pt,color=qqwwzz,fill=qqwwzz,fill opacity=0.10000000149011612] (0.9999999999820135,-4.004896197649135) -- (3.9999999999820135,-3.004896197649132) -- (2.9999999999820135,-1.0048961976491353) -- (0.9999999999820135,-2.004896197649133) -- cycle;
	\fill[line width=0.4pt,color=qqwwzz,fill=qqwwzz,fill opacity=0.10000000149011612] (2.9999999999820135,-1.0048961976491353) -- (5.9999999999820135,-1.0048961976491353) -- (5.9999999999820135,-4.004896197649135) -- (3.9999999999820135,-3.004896197649132) -- cycle;
	\draw [line width=0.4pt,color=qqwwzz] (1.,1.)-- (4.,2.);
	\draw [line width=0.4pt,color=qqwwzz] (4.,2.)-- (3.,4.);
	\draw [line width=0.4pt,color=qqwwzz] (3.,4.)-- (1.,3.);
	\draw [line width=0.4pt,color=qqwwzz] (1.,3.)-- (1.,1.);
	\draw [line width=0.4pt,color=qqwwzz] (3.,4.)-- (6.,4.);
	\draw [line width=0.4pt,color=qqwwzz] (6.,4.)-- (6.,1.);
	\draw [line width=0.4pt,color=qqwwzz] (6.,1.)-- (4.,2.);
	\draw [line width=0.4pt,color=qqwwzz] (4.,2.)-- (3.,4.);
	\draw (0.9985576220934834,0.15) node[anchor=north west] {$x$};
	\draw (-0.25,1.5) node[anchor=north west] {$y$};
	\draw [color=xfqqff](1.7054254382427527,3.431275011921382) node[anchor=north west] {$\boldsymbol{J}^{-T}\boldsymbol{\vartheta}_2$};
	\draw [color=qqwwzz](1.4379619402403265,2.208584735338862) node[anchor=north west] {$\Omega_1$};
	\draw [color=qqwwzz](5.010510092129876,2.647989053485705) node[anchor=north west] {$\Omega_2$};
	\draw (1.074975764379891,1.043208065471148) node[anchor=north west] {$1$};
	\draw (3.864237957833764,1.8838076306216303) node[anchor=north west] {$2$};
	\draw (6.1185731552827844,1.1196262077575554) node[anchor=north west] {$3$};
	\draw (0.5591533039466405,3.4694840830645854) node[anchor=north west] {$4$};
	\draw (2.5842340745364387,4.45) node[anchor=north west] {$5$};
	\draw (6.156782226425988,4.157247363642253) node[anchor=north west] {$6$};
	\draw [->,line width=1.2pt,color=xfqqff] (3.,3.8) -- (3.4,2.);
	\draw [->,line width=1.2pt,color=qqqqff] (3.,3.8) -- (3.8,2.2);
	\draw [->,line width=1.2pt,color=qqqqff] (3.8,2.2) -- (3.4,2.);
	\draw [->,line width=1.2pt,color=qqqqff] (3.9618886883568507,2.2809443441784256) -- (3.161888688356851,3.8809443441784253);
	\draw [->,line width=1.2pt,color=xfqqff] (3.9618886883568507,2.2809443441784256) -- (4.2,4.4);
	\draw [->,line width=1.2pt,color=qqqqff] (3.161888688356851,3.8809443441784253) -- (4.2,4.4);
	\draw [color=xfqqff](4.322746811552209,3.7942611877818173) node[anchor=north west] {$\boldsymbol{J}^{-T}\boldsymbol{\vartheta}_3$};
	\draw [color=qqwwzz](1.495275546955132,6.9) node[anchor=north west] {$\Xi$};
	\draw [color=xfqqff](1.8582617228155676,5.95) node[anchor=north west] {$\boldsymbol{\vartheta}_1$};
	\draw [color=xfqqff](4.0,7.15) node[anchor=north west] {$\boldsymbol{\vartheta}_2$};
	\draw [color=xfqqff](2.85,9.5) node[anchor=north west] {$\boldsymbol{\vartheta}_3$};
	\draw [color=xfqqff](0.4,8.3) node[anchor=north west] {$\boldsymbol{\vartheta}_4$};
	\draw (5.05,7.7) node[anchor=north west] {$\xi$};
	\draw (2.25,10.5) node[anchor=north west] {$\eta$};
	\draw [line width=0.4pt,color=qqwwzz] (1.,6.)-- (4.,6.);
	\draw [line width=0.4pt,color=qqwwzz] (4.,6.)-- (4.,9.);
	\draw [line width=0.4pt,color=qqwwzz] (4.,9.)-- (1.,9.);
	\draw [line width=0.4pt,color=qqwwzz] (1.,9.)-- (1.,6.);
	\draw [->,line width=1.2pt,color=xfqqff] (1.5,6.) -- (3.5,6.);
	\draw [->,line width=1.2pt,color=xfqqff] (4.,6.5) -- (4.,8.5);
	\draw [->,line width=1.2pt,color=xfqqff] (3.5,9.) -- (1.5,9.);
	\draw [->,line width=1.2pt,color=xfqqff] (1.,8.5) -- (1.,6.5);
	\draw [->,line width=1.2pt] (2.5,7.5) -- (5.,7.5); 
	\draw [->,line width=1.2pt] (2.5,7.5) -- (2.5,10.);
	\draw [shift={(4.25,5.5)},line width=2.pt]  plot[domain=-0.9272952180016123:0.9272952180016122,variable=\t]({1.*1.25*cos(\t r)+0.*1.25*sin(\t r)},{0.*1.25*cos(\t r)+1.*1.25*sin(\t r)});
	\draw [->,line width=2.pt] (5.1544257262252495,4.637147691813988) -- (5.,4.5);
	\draw (5.698273372707543,5.991282778516033) node[anchor=north west] {$\boldsymbol{J}^{-T}$};
	\draw [line width=0.4pt,color=qqwwzz] (0.9999999999820135,-4.004896197649135)-- (3.9999999999820135,-3.004896197649132);
	\draw [line width=0.4pt,color=qqwwzz] (3.9999999999820135,-3.004896197649132)-- (2.9999999999820135,-1.0048961976491353);
	\draw [line width=0.4pt,color=qqwwzz] (2.9999999999820135,-1.0048961976491353)-- (0.9999999999820135,-2.004896197649133);
	\draw [line width=0.4pt,color=qqwwzz] (0.9999999999820135,-2.004896197649133)-- (0.9999999999820135,-4.004896197649135);
	\draw [line width=0.4pt,color=qqwwzz] (2.9999999999820135,-1.0048961976491353)-- (5.9999999999820135,-1.0048961976491353);
	\draw [line width=0.4pt,color=qqwwzz] (5.9999999999820135,-1.0048961976491353)-- (5.9999999999820135,-4.004896197649135);
	\draw [line width=0.4pt,color=qqwwzz] (5.9999999999820135,-4.004896197649135)-- (3.9999999999820135,-3.004896197649132);
	\draw [line width=0.4pt,color=qqwwzz] (3.9999999999820135,-3.004896197649132)-- (2.9999999999820135,-1.0048961976491353);
	\draw [color=qqwwzz](1.4379619402403265,-2.7968035844208283) node[anchor=north west] {$\Omega_1$};
	\draw [color=qqwwzz](5.010510092129876,-2.3573992662739855) node[anchor=north west] {$\Omega_2$};
	\draw (1.074975764379891,-3.9621802542885427) node[anchor=north west] {$1$};
	\draw (3.864237957833764,-3.12158068913806) node[anchor=north west] {$2$};
	\draw (6.1185731552827844,-3.8857621120021353) node[anchor=north west] {$3$};
	\draw (0.5591533039466405,-1.5359042366951048) node[anchor=north west] {$4$};
	\draw (2.5842340745364387,-0.55538) node[anchor=north west] {$5$};
	\draw (6.156782226425988,-0.8481409561174375) node[anchor=north west] {$6$};
	\draw [->,line width=1.2pt,color=qqqqff] (3.9618886883388638,-2.7239518534707057) -- (3.1618886883388653,-1.1239518534707098);
	\draw [->,line width=1.2pt,color=xfqqff] (3.9618886883388638,-2.7239518534707057) -- (4.29340993667295,-0.5581912293036662);
	\draw [->,line width=1.2pt,color=qqqqff] (3.1618886883388653,-1.1239518534707098) -- (4.29340993667295,-0.558191229303666);
	\draw [->,line width=0.8pt] (0.,0.) -- (1.,0.);
	\draw [->,line width=0.8pt] (0.,0.) -- (0.,1.);
	\draw [shift={(5.906118359711166,-0.019104535571601747)},line width=2.pt]  plot[domain=-0.9272952180016123:0.9272952180016122,variable=\t]({1.*1.25*cos(\t r)+0.*1.25*sin(\t r)},{0.*1.25*cos(\t r)+1.*1.25*sin(\t r)});
	\draw [->,line width=2.pt] (6.8105440859364155,-0.8819568437576141) -- (6.656118359711166,-1.0191045355716017);
	\draw (7.360367967436906,0.39365385603668424) node[anchor=north west] {$\psi_j$};
	\draw [->,line width=1.2pt,color=qqqqff] (2.9999999999820135,-1.2048961976491357) -- (2.6,-1.4);
	\draw [->,line width=1.2pt,color=xfqqff] (3.798041520925957,-2.800979239537021) -- (2.6,-1.4);
	\draw [->,line width=1.2pt,color=qqqqff] (3.798041520925957,-2.800979239537021) -- (2.9999999999820135,-1.2048961976491357);
	\draw [->,line width=0.8pt] (0.,-5.) -- (1.,-5.);
	\draw [->,line width=0.8pt] (0.,-5.) -- (0.,-4.);
	\draw (0.9794530865218816,-4.85) node[anchor=north west] {$x$};
	\draw (-0.25,-3.5) node[anchor=north west] {$y$};
	\draw [color=xfqqff](1.208707513381104,-1.8797858769839386) node[anchor=north west] {$\psi_2\,\boldsymbol{J}^{-T}\boldsymbol{\vartheta}_2$};
	\draw [color=xfqqff](4.20,-1.3448588809790862) node[anchor=north west] {$\psi_3\,\boldsymbol{J}^{-T}\boldsymbol{\vartheta}_3$};
	\begin{scriptsize}
	\draw [fill=black] (0.,0.) circle (2.5pt);
	\draw [fill=qqwwzz] (1.,1.) circle (2.5pt);
	\draw [fill=qqwwzz] (4.,2.) circle (2.5pt);
	\draw [fill=qqwwzz] (3.,4.) circle (2.5pt);
	\draw [fill=qqwwzz] (1.,3.) circle (2.5pt);
	\draw [fill=qqwwzz] (6.,1.) circle (2.5pt);
	\draw [fill=qqwwzz] (6.,4.) circle (2.5pt);
	\draw [fill=qqwwzz] (1.,6.) circle (2.5pt);
	\draw [fill=qqwwzz] (4.,6.) circle (2.5pt);
	\draw [fill=qqwwzz] (4.,9.) circle (2.5pt);
	\draw [fill=qqwwzz] (1.,9.) circle (2.5pt);
	\draw [fill=black] (2.5,7.5) circle (2.5pt);
	\draw [fill=qqwwzz] (0.9999999999820135,-4.004896197649135) circle (2.5pt);
	\draw [fill=qqwwzz] (3.9999999999820135,-3.004896197649132) circle (2.5pt);
	\draw [fill=qqwwzz] (2.9999999999820135,-1.0048961976491353) circle (2.5pt);
	\draw [fill=qqwwzz] (0.9999999999820135,-2.004896197649133) circle (2.5pt);
	\draw [fill=qqwwzz] (5.9999999999820135,-4.004896197649135) circle (2.5pt);
	\draw [fill=qqwwzz] (5.9999999999820135,-1.0048961976491353) circle (2.5pt);
	\draw [fill=black] (0.,-5.) circle (2.5pt);
	\end{scriptsize}
	\end{tikzpicture}
	\caption{Covariant Piola transformation and topological correction function $\psi_j$ mapping of N\'{e}d\'{e}lec base functions from the parametric space into the physical space.}
	\label{fig:ori}
\end{figure}
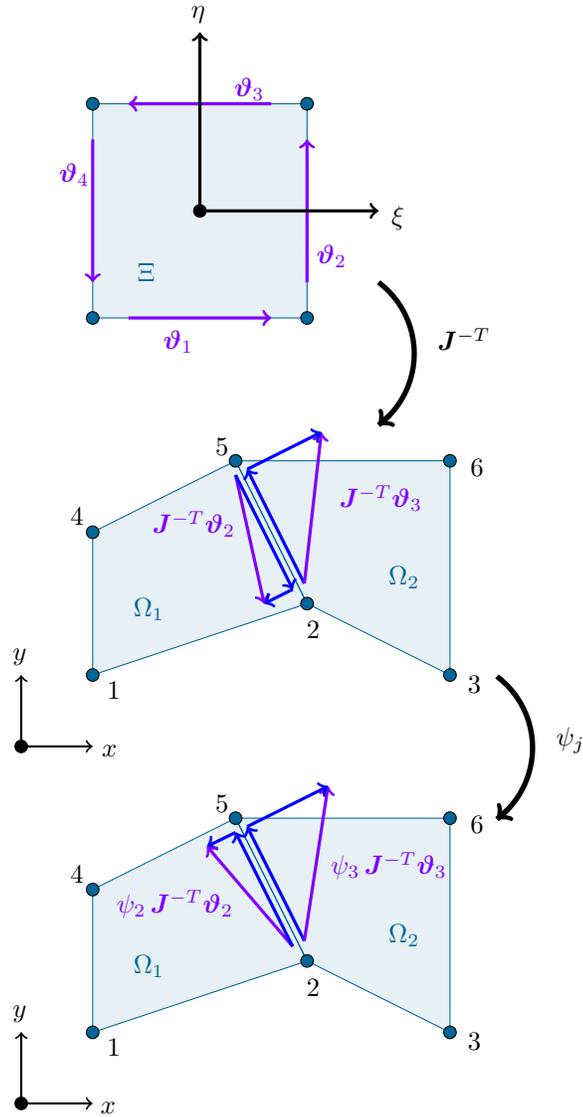
The drawings in \cref{fig:ori} show the different roles of the mapping functions:\begin{enumerate}
	\item The covariant Piola transformation scales the projection onto the edge tangent.
	\item The topological correction function sets a consistent orientation.
\end{enumerate}
Thus, the final form of our edge base functions reads
\begin{gather}
\bm{\theta}_j= \psi_j \bm{J}^{-T} \bm{\vartheta}_j \, , \qquad
\psi_j = \left\{ \begin{matrix}
1  \\ -1  
\end{matrix}   \right . \begin{matrix}
\textrm{orientation is equal} \\ \textrm{else}
\end{matrix} \, .
\label{eq:base}
\end{gather}
Using \cref{eq:base} for the approximation of the microdistortion $\bm{\zeta}$ yields   
\begin{gather}
\bm{\zeta}_e^h = 
\underbrace{\begin{bmatrix}
	\bm{\theta}_1 & \bm{\theta}_2 & \bm{\theta}_3 & \bm{\theta}_4
	\end{bmatrix}}_{\displaystyle \bm{\Theta}}
\underbrace{\begin{bmatrix}
	\zeta_{1} \\ \zeta_{2} \\ \zeta_{3} \\ \zeta_{4}
	\end{bmatrix}}_{\displaystyle \bar{\bm{\zeta}}_{e}} \, , \qquad
\bm{\zeta}^h = \bigcup^{n}_{e=1} \bm{\zeta}_e^h \, .
\label{eq:zeta}
\end{gather} 
For vectors undergoing a covariant Piola transformation, the transformation of the curl operator simplifies to
\begin{align}
	\curl_x \bm{\theta}_j = \dfrac{1}{\det\bm{J}} \psi_j \curl \bm{\vartheta}_j \, .
\end{align}

\subsection{Element stiffness matrices}\label{ch:6}
For ease of presentation we consider only the Lax--Milgram setting. The mixed formulation follows directly with simple adaptations. 

With the approximations in \cref{eq:u} for the displacement field $u$ and in \cref{eq:zeta} for the microdistortion $\bm{\zeta}$ the weak form in \cref{eq:weak} results in
\begin{gather}
\bigcup^{n}_{e=1} (\bm{K}_\textrm{e} + \bm{K}_\textrm{micro} + \bm{K}_\textrm{macro})_e \begin{bmatrix}
\bar{\vb{u}}_e \\
\bar{\bm{\zeta}}_e
\end{bmatrix} = \AS{\bigcup^{n}_{e=1} \begin{bmatrix}
	\bar{\vb{f}}_e \\
	\bar{\bm{\omega}}_e
\end{bmatrix}} \, ,
\end{gather}
where $\bm{K}_\textrm{e}$, $\bm{K}_\textrm{micro}$ and $\bm{K}_\textrm{macro}$ are the element stiffness matrices employing the base function matrices $\bm{H}$ and $\bm{\Theta}$ according to \cref{eq:zeta} and \cref{eq:H}, respectively 
\begin{subequations}
	\begin{align}
	\bm{K}_\textrm{e} &= 2\mue \int_{\Xi} \begin{bmatrix}
	(\nabla \bm{H})^T \nabla\bm{H} &  -(\nabla \bm{H})^T \bm{\Theta} 
	\\ - \bm{\Theta}^T \nabla \bm{H} & \bm{\Theta}^T \bm{\Theta}
	\end{bmatrix} \det\bm{J}\, \dd \Xi \, , \\[2ex]
	\bm{K}_\textrm{micro} &= 2\mumi \int_{\Xi} \begin{bmatrix}
	\bm{O} & \bm{O} \\
	\bm{O} & \bm{\Theta}^T \bm{\Theta}
	\end{bmatrix} \det\bm{J} \,\dd \Xi \, , \\[2ex]
	\bm{K}_\textrm{macro} &= \muma \Lc^2 \int_{\Xi} \begin{bmatrix}
	\bm{O} & \bm{O} \\
	\bm{O} & (\curl \bm{\Theta})^T \curl\bm{\Theta}
	\end{bmatrix} \det\bm{J}\, \dd \Xi \, ,
	\end{align}
\end{subequations}
with $\bm{O} \in \{0 \}^{4\times4}$.
The finite element has 8 degrees of freedom.
\AS{The right-hand side reads 
\begin{align}
	\bar{\vb{f}}_e &= \int_{\Xi} \bm{H}^T f \det \bm{J} \, \dd \Xi \, ,\\
	\bar{\bm{\omega}}_e &= \int_{\Xi} \bm{\Theta}^T \bm{\omega} \det \bm{J} \, \dd \Xi \, .
\end{align}
}
In order to compare our formulation, we also derive a nodal $\Hone$-finite element 
\begin{gather}
\bm{\zeta} = \bigcup^{n}_{e=1}\underbrace{ \begin{bmatrix}
	N_1 \bm{I} & N_2 \bm{I} & N_3 \bm{I} & N_4 \bm{I}
	\end{bmatrix} }_{\displaystyle \bm{N}} \underbrace{\begin{bmatrix}
	\zeta_{1} \\ \zeta_{2} \\ \vdots \\ \zeta_{8} \end{bmatrix} }_{\displaystyle \bar{\bm{\zeta}}_{e}} \, , \quad \bm{I} = \begin{bmatrix}
1 & 0 \\
0 & 1
\end{bmatrix} \, .
\label{eq:nod}
\end{gather}
In contrast to the hybrid element, the approach in \cref{eq:nod} requires 8 dofs per element for the microdistortion. Using \cref{eq:nod} we obtain the following stiffness matrices for the nodal element
\begin{subequations}
	\begin{align}
	\bm{K}_\textrm{e} &= 2\mue \int_{\Omega} \begin{bmatrix}
	(\nabla \bm{H})^T \nabla\bm{H} & -(\nabla \bm{H})^T \bm{N} \\
	- \bm{N}^T \nabla \bm{H} & \bm{N}^T \bm{N}
	\end{bmatrix} \det\bm{J} \,\dd \Omega \, , \\[2ex]
	\bm{K}_\textrm{micro} &= 2\mumi \int_{\Omega} \begin{bmatrix}
	\bm{O} & \bm{O}_\textrm{a}^T \\
	\bm{O}_\textrm{a} & \bm{N}^T \bm{N}
	\end{bmatrix} \det\bm{J} \,\dd \Omega  \, , \\[2ex]
	\bm{K}_\textrm{macro} &= 
	\muma \Lc^2  
	\int_{\Omega} \begin{bmatrix}
	\bm{O} & \bm{O}_\textrm{a}^T \\
	\bm{O}_\textrm{a} & (\curl\bm{N})^T \curl\bm{N}
	\end{bmatrix} \det\bm{J}\, \dd \Omega \, ,
	\end{align}
\end{subequations}
with $\bm{O}_\textrm{a} \in \{0\}^{8 \times 4}$. \AS{Consequently, $\bar{\bm{\omega}}_e$ changes to 
\begin{align}
	\bar{\bm{\omega}}_e = \int_{\Xi} \bm{N}^T \bm{\omega} \det \bm{J} \, \dd \Xi \, .
\end{align}}
In conclusion, we compare the hybrid element having 8 degrees of freedom in total with the nodal element having 12 degrees of freedom. The difference in the overall degrees of freedom results from the vectorial approach to the microdistortion in the hybrid element.

\section{Numerical examples}\label{ch:7}
In following examples we construct analytical solutions by imposing predefined displacement and microdistortion fields and calculating the resulting right-hand side. The predefined fields are the analytical solutions to the resulting right-hand side along with the derived Dirichlet boundary conditions (for a full derivation see \cref{ap:A}). 
\AS{Further, in all subsequent examples the domain and the flux field $\bm{\zeta}$ lie in the $x-y$ plane and the displacement $u$ is parallel to the $z$-axis. Correspondingly, for figures of $u$ we provide a three-dimensional perspective and figures of $\bm{\zeta}$ are aerial views of the $x-y$ plane. The examples have the mechanical interpretation of a membrane antiplane deformation.}

\subsection{Benchmark for an imposed vanishing microdistortion}\label{sec:imposed}
We impose the predefined fields
\begin{gather}
\widetilde{u}(x,y) = 4 - \dfrac{x^2}{8}  - \dfrac{y^2}{8} + x \, y \, , \qquad
\widetilde{\bm{\zeta}}(x,y) = 0 \, .
\label{eq:im}
\end{gather}
In order to constrain the numerical solution to that of our proposed fields in \cref{eq:im}, we set the following Dirichlet boundary conditions
\begin{gather}
u(x,y)\at_{\partial \Omega} = \widetilde{u}(x,y)\at_{\partial \Omega}  \, , \qquad  \langle \bm{\zeta}(x,y) , \, \bm{\tau} \rangle \at_{\partial \Omega} =  \langle \widetilde{\bm{\zeta}}(x,y) , \, \bm{\tau} \rangle \at_{\partial \Omega} \, .
\label{eq:totalDir1}
\end{gather}
In the following example we set for simplicity 
\begin{gather}
\mue = \mumi = \muma = \Lc = 1 \, ,
\end{gather}
and extract the resulting force and moment (the right-hand side)
\begin{gather}
f = 1 \, , \qquad  \bm{\omega} = \begin{bmatrix}
\dfrac{x}{2}  - 2y \\[2ex]
\dfrac{y}{2} - 2x
\end{bmatrix}  \, .
\end{gather}
Our simulations consider the domain $\Omega=[-4,4] \times [-4,4]$ with irregular meshes under h-refinement, as shown in \cref{fig:con}. Both element formulations converge towards the analytical solution, see \cref{fig:gr2}.
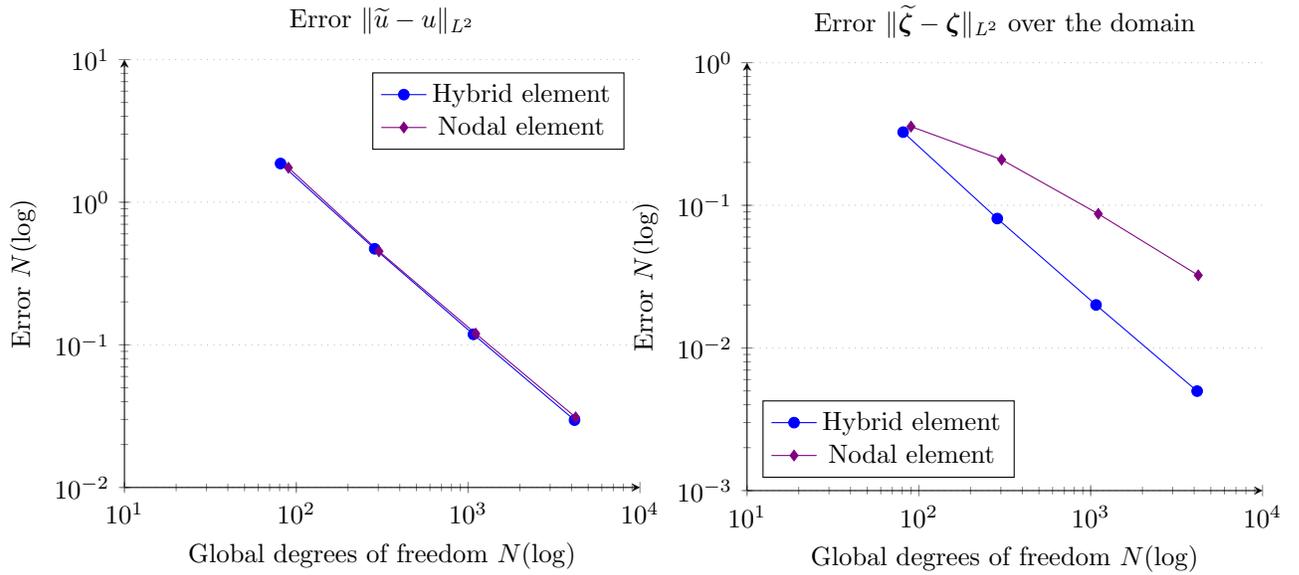
\begin{figure}
	\centering
	\begin{subfigure}{0.48\linewidth}
		\begin{tikzpicture}
		\begin{loglogaxis}[
		/pgf/number format/1000 sep={},
		title={Error $\| \widetilde{u} - u \|_{\Le}$},
		axis lines = left,
		xlabel={Global degrees of freedom $N(\log)$},
		ylabel={Error $N(\log)$},
		xmin=10, xmax=10000,
		ymin=0.01, ymax=10,
		xtick={10,100,1000,10000},
		ytick={0.01,0.1,1,10},
		legend pos=north east,
		ymajorgrids=true,
		grid style=dotted,
		]
		\addplot[
		color=blue,
		mark=*,
		]
		coordinates {
			(81,	1.866948112)
			(286,	0.47133458)
			(1074,	0.118415549)
			(4162,	0.029664259)		
		};
		\addlegendentry{Hybrid element}
		\addplot[
		color=violet,
		mark=diamond*,
		] 
		coordinates {
			(90,	1.746033228)
			(303,	0.45267357)
			(1107,	0.119845842)
			(4227,	0.031026686)
		};
		\addlegendentry{Nodal element}
		\end{loglogaxis}
		\end{tikzpicture}
	\end{subfigure}
	\begin{subfigure}{0.48\linewidth}
		\begin{tikzpicture}
		\begin{loglogaxis}[
		/pgf/number format/1000 sep={},
		title={Error $\| \widetilde{\bm{\zeta}} - \bm{\zeta} \|_{\Le}$ over the domain},
		axis lines = left,
		xlabel={Global degrees of freedom $N(\log)$},
		ylabel={Error $N(\log)$},
		xmin=10, xmax=10000,
		ymin=0.001, ymax=1,
		xtick={10,100,1000,10000},
		ytick={0.001,0.01,0.1,1},
		legend pos=south west,
		ymajorgrids=true,
		grid style=dotted,
		]
		\addplot[
		color=blue,
		mark=*,
		]
		coordinates {
			(81,0.325805523)
			(286,0.080759417)
			(1074,0.020024626)
			(4162,0.004982047)
		};
		\addlegendentry{Hybrid element}
		\addplot[
		color=violet,
		mark=diamond*,
		]
		coordinates {
			(90,0.356646941)
			(303,0.209137612)
			(1107,0.087128584)
			(4227,0.032303368)			
		};
		\addlegendentry{Nodal element}
		
		\end{loglogaxis}
		\end{tikzpicture}
	\end{subfigure}
    \caption{Convergence behaviour of element formulations under mesh refinement.}
    \label{fig:gr2}
\end{figure}
The microdistortion field $\bm{\zeta}$ displayed in \cref{fig:mic1} approaches zero with each refinement, satisfying the imposed field. We notice faster convergence in the hybrid element.
\begin{figure}
	\centering
	\begin{subfigure}{0.3\linewidth}
		\centering
		\includegraphics[width=1\linewidth]{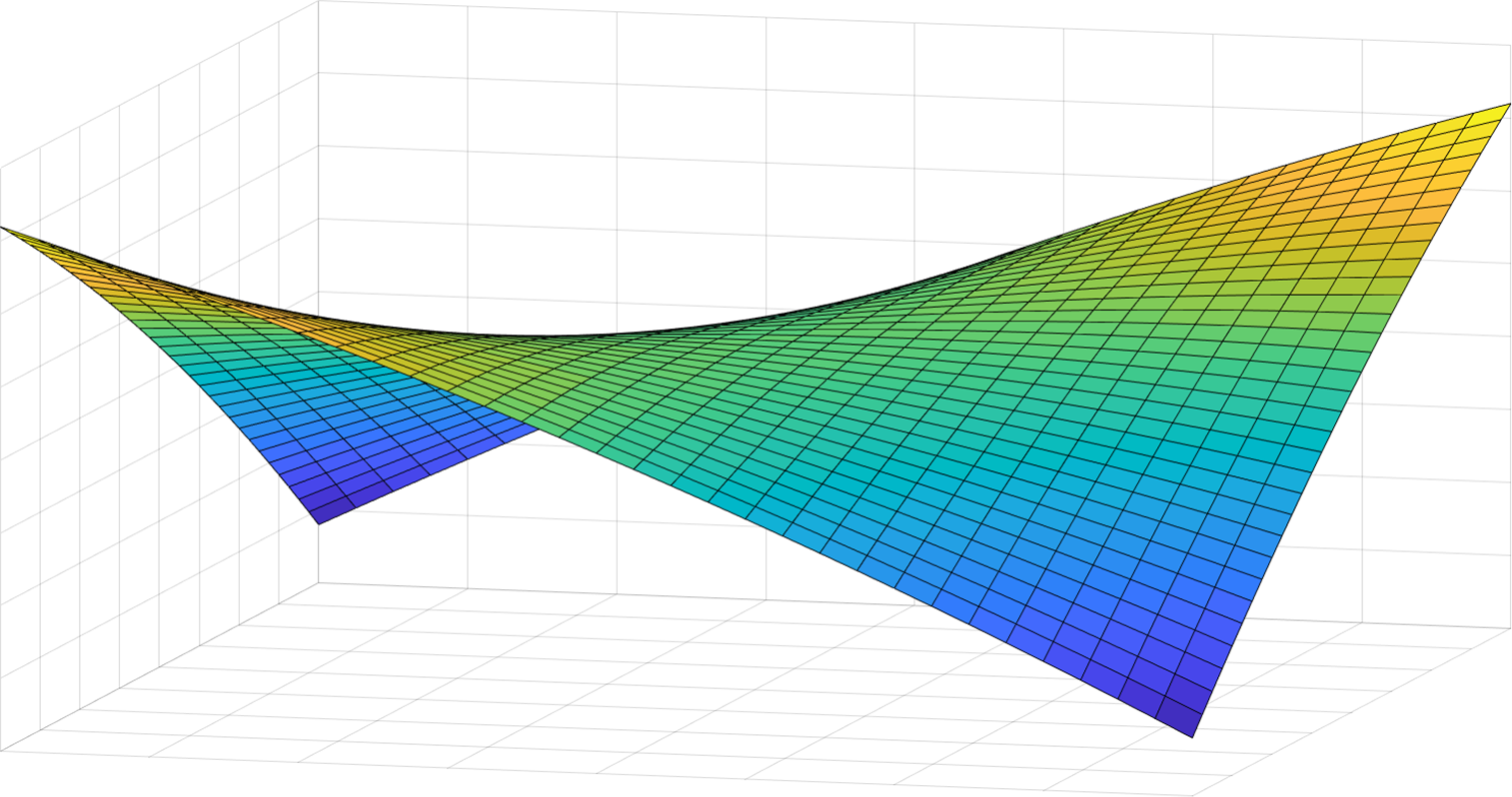}
		\caption{Analytical solution}
	\end{subfigure}
    \begin{subfigure}{0.3\linewidth}
    	\includegraphics[width=1\linewidth]{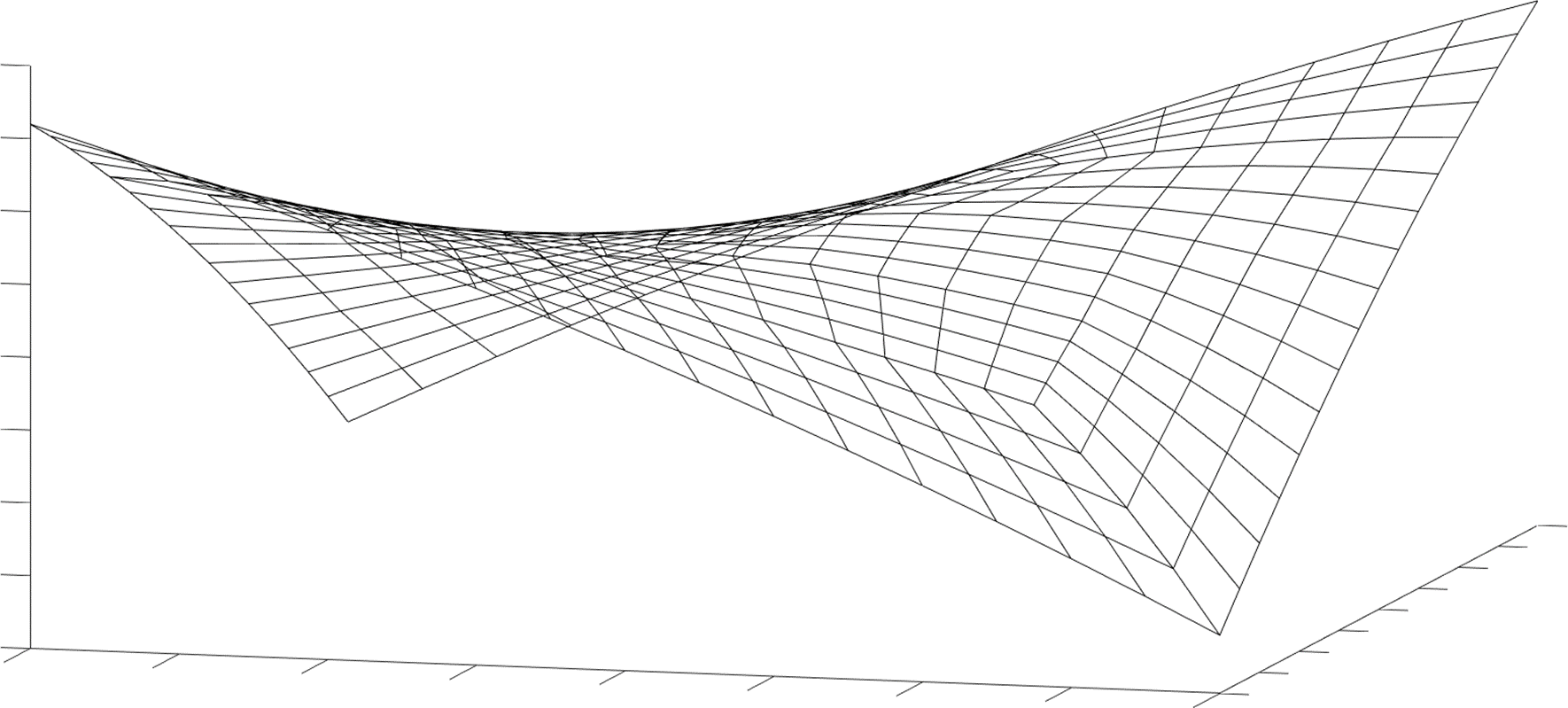}
    	\caption{336 elements}
    \end{subfigure}
	\begin{subfigure}{0.3\linewidth}
		\includegraphics[width=1\linewidth]{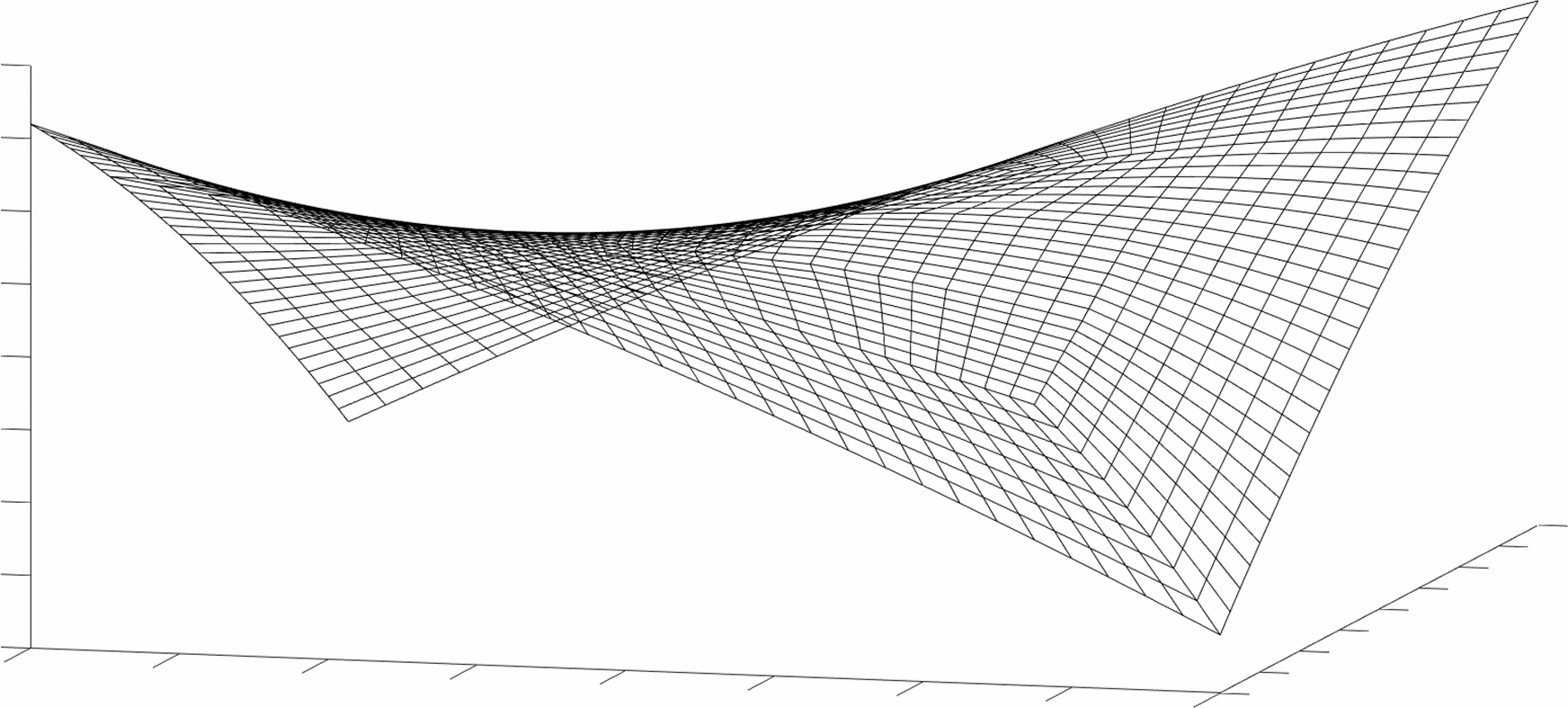}
		\caption{1344 elements}
	\end{subfigure}
	\caption{Displacement $u$ of the analytical and finite element solutions.}
	\label{fig:con}
\end{figure}

\begin{figure}
	\centering
	\begin{subfigure}{0.3\linewidth}
		\includegraphics[width=1\linewidth]{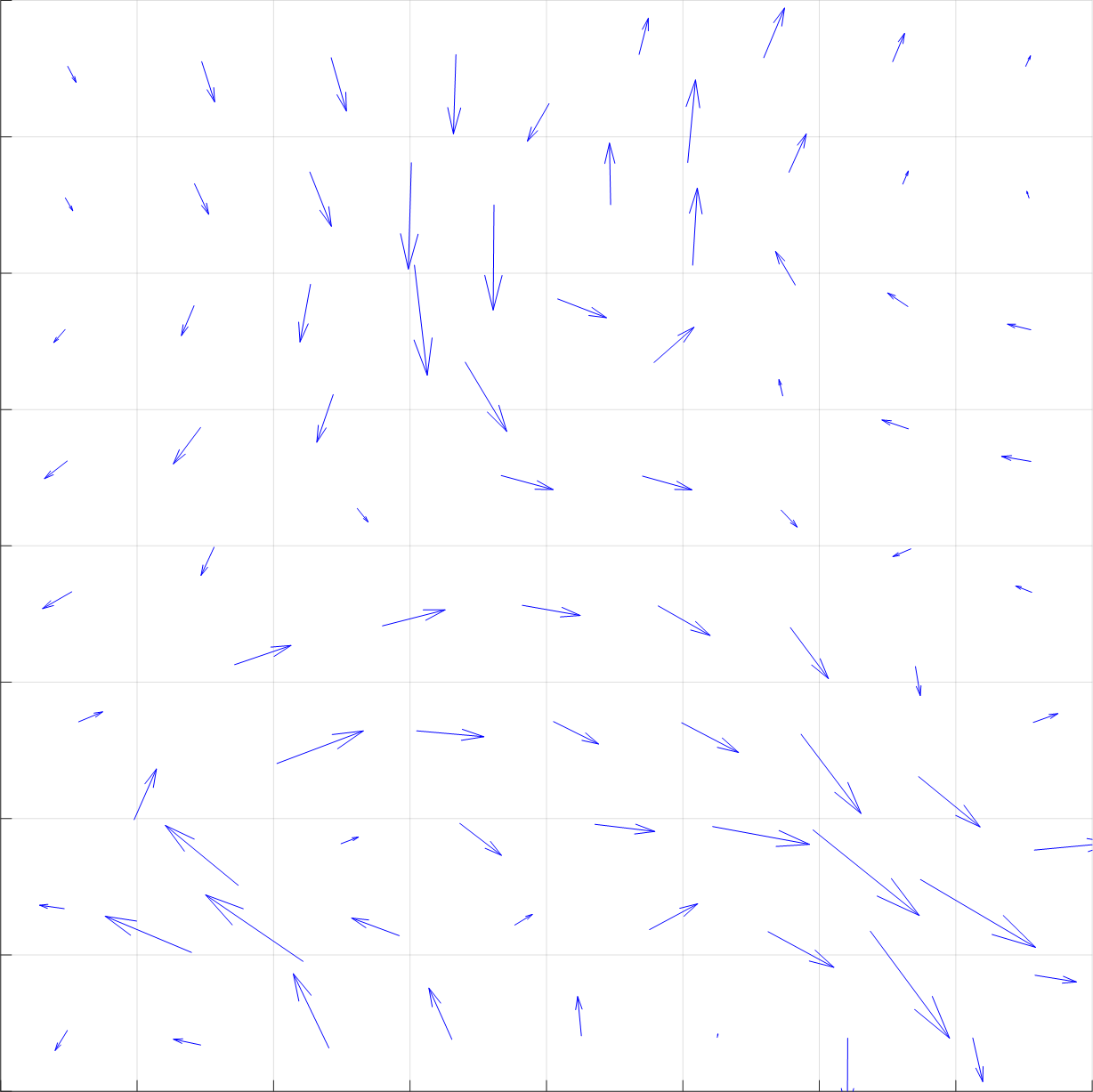}
		\caption{84 hybrid elements}
	\end{subfigure}
	\begin{subfigure}{0.3\linewidth}
		\includegraphics[width=1\linewidth]{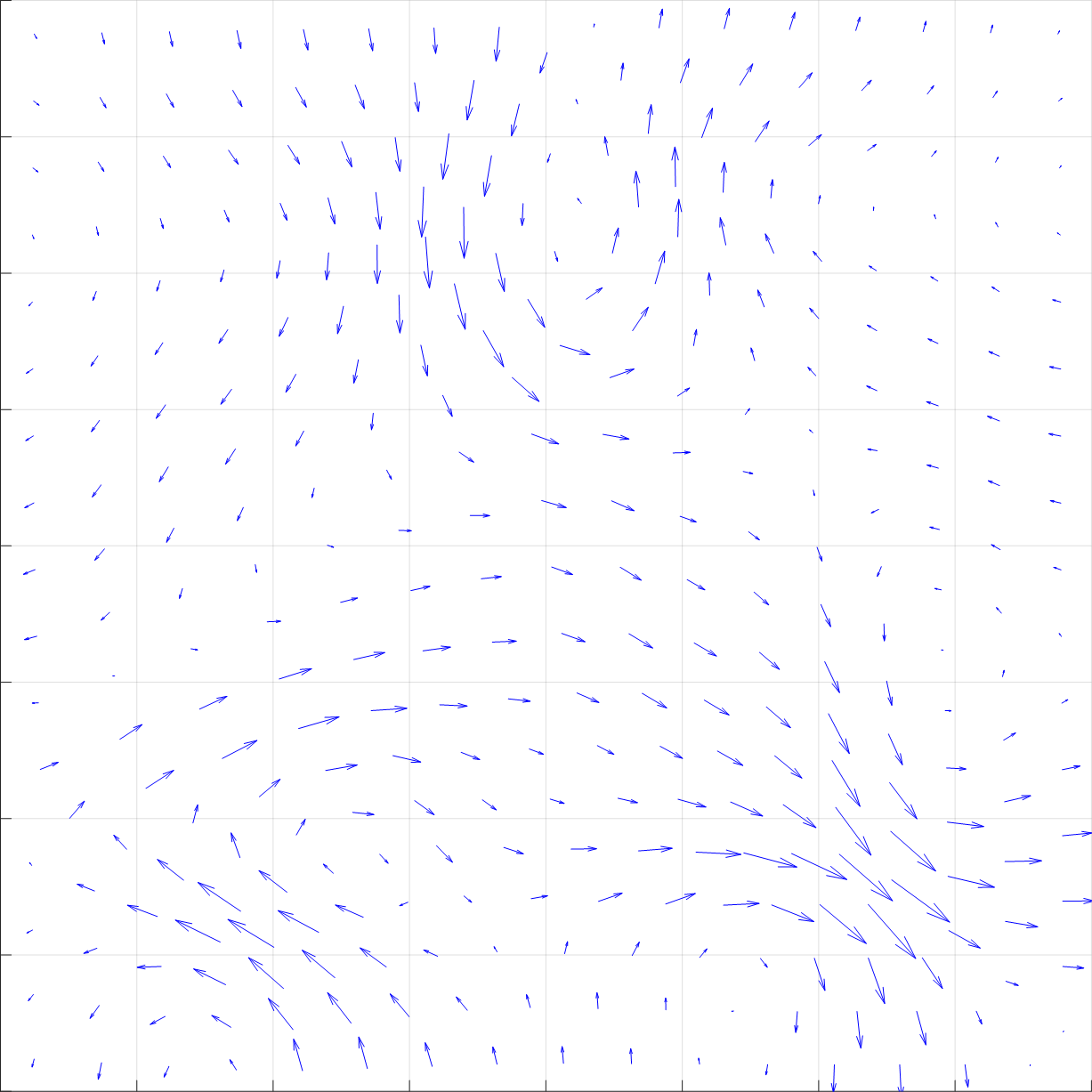}
		\caption{336 hybrid elements}
	\end{subfigure}
	\begin{subfigure}{0.3\linewidth}
		\includegraphics[width=1\linewidth]{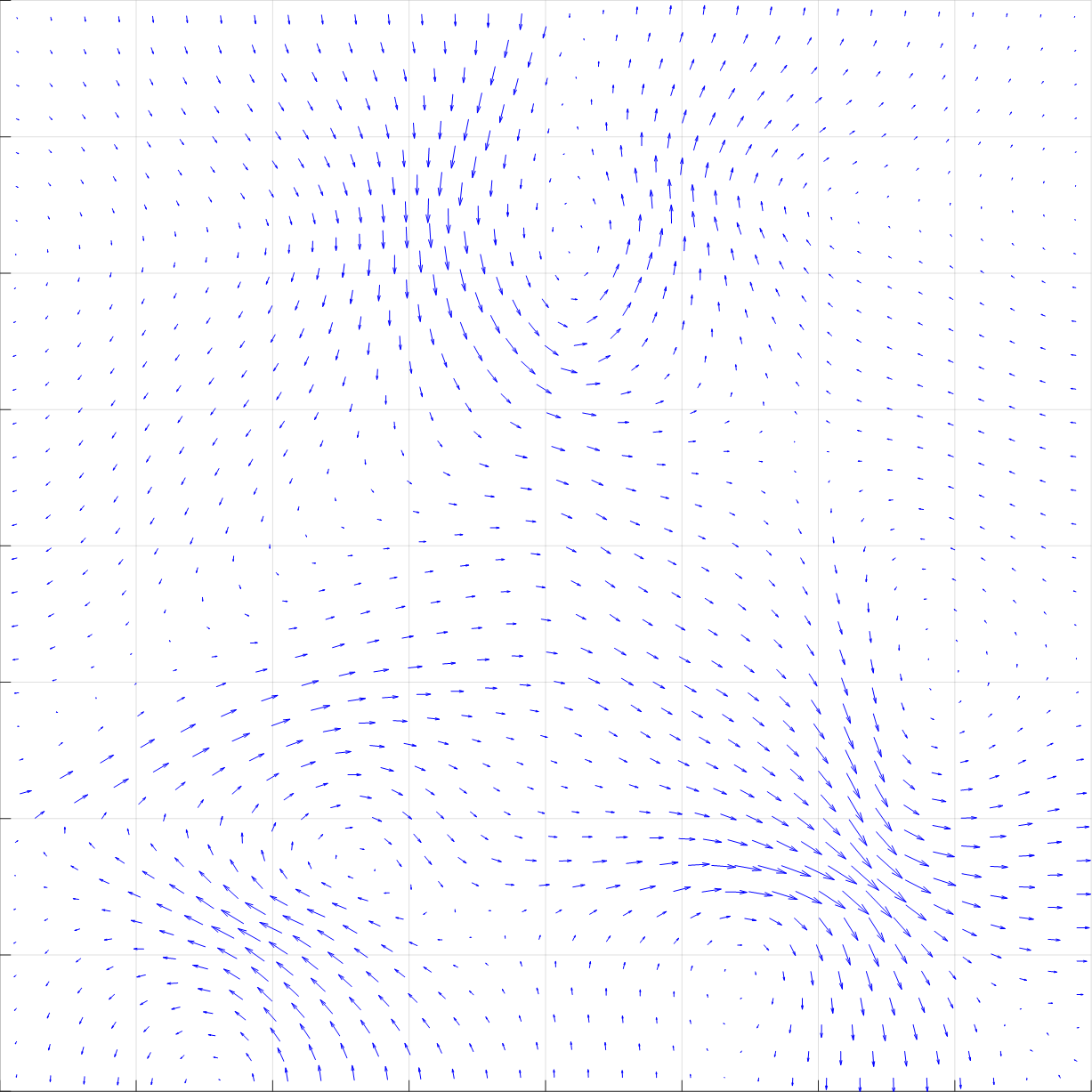}
		\caption{1344 hybrid elements}
	\end{subfigure}
	\begin{subfigure}{0.3\linewidth}
		\includegraphics[width=1\linewidth]{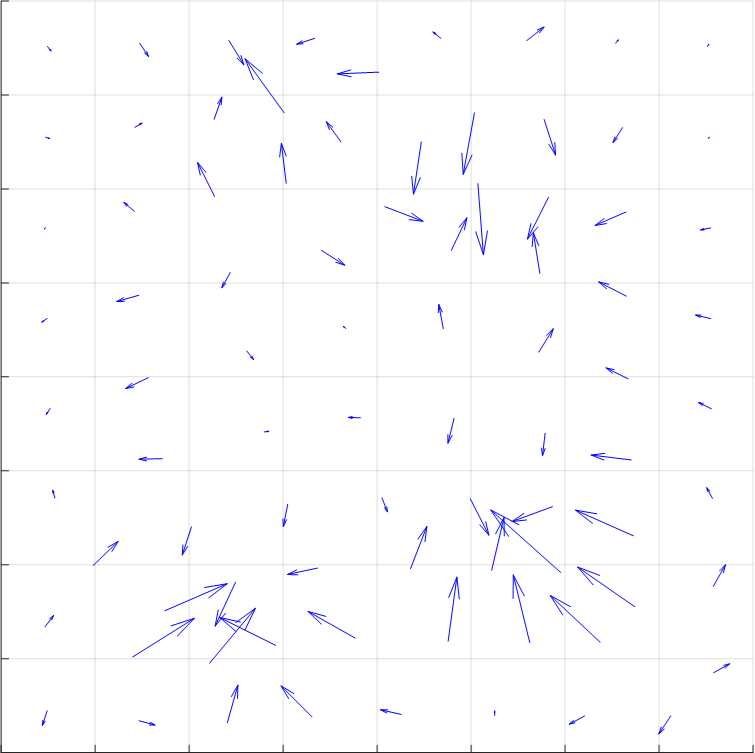}
		\caption{84 nodal elements}
	\end{subfigure}
	\begin{subfigure}{0.3\linewidth}
		\includegraphics[width=1\linewidth]{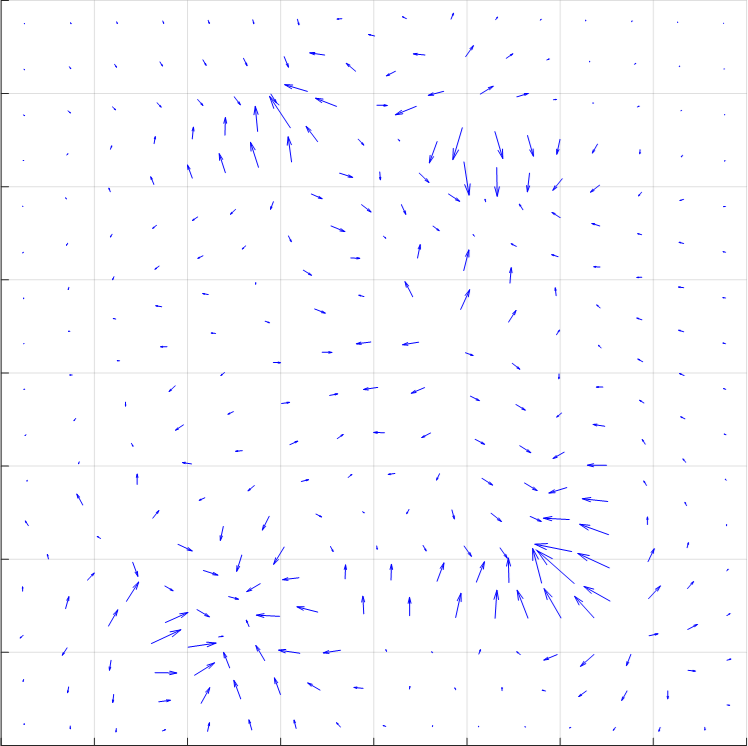}
		\caption{336 nodal elements}
	\end{subfigure}
	\begin{subfigure}{0.3\linewidth}
		\includegraphics[width=1\linewidth]{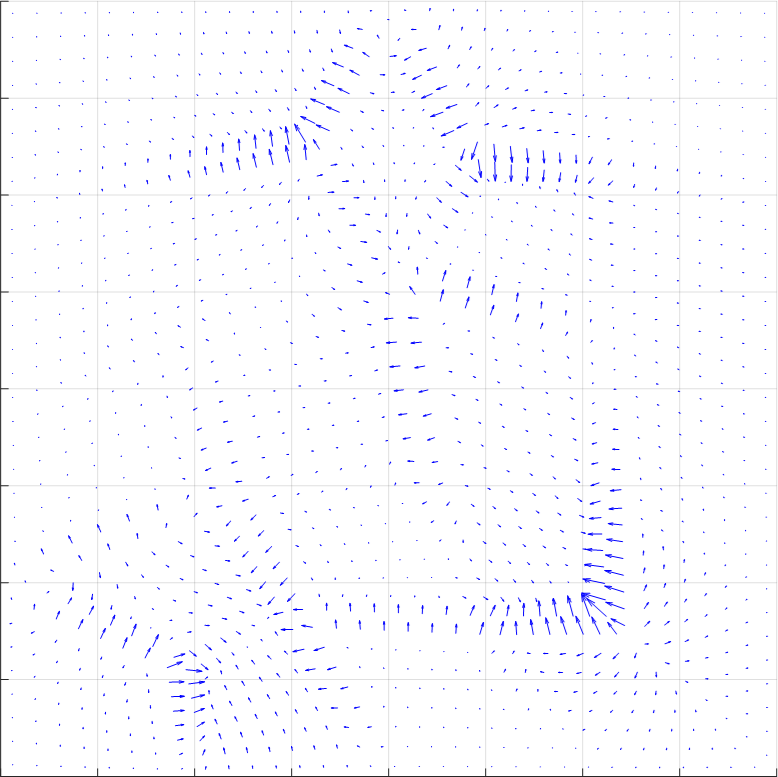}
		\caption{1344 nodal elements}
	\end{subfigure}
	\caption{Decay of the microdistortion $\bm{\zeta}$ according to \cref{eq:im} on irregular meshes undergoing refinement. The intensity of the microdistortion approaches zero with each refinement. This is seen here in a decrease of the flux vectors.}
	\label{fig:mic1}
\end{figure}

\subsection{Benchmark for a non-vanishing imposed microdistortion}
\label{subse:benchmark_non_van_microdist}
In the following step in our investigation we test our finite element formulations for a non-vanishing microdistortion field $\bm{\zeta}$, specifically a rotation field, as to determine the convergence behaviour of the nodal element with respect to the curl stiffness. 
We set $\Omega=[-4,4] \times [-4,4]$, $\mue=\muma=\mumi=\Lc=1$ and the fields 
\begin{gather}
\widetilde{u}(x,y) = xy \left (\dfrac{y^2}{16} - \dfrac{x^2}{16} \right ) - 1 \, , \qquad
\widetilde{\bm{\zeta}}(x,y) = \begin{bmatrix}
-y (\dfrac{x^2}{8}-2)(\dfrac{y^2}{8}-2) \\[2ex]
x (\dfrac{x^2}{8}-2)(\dfrac{y^2}{8}-2)
\end{bmatrix}
\label{eq:rand}
\end{gather}
with the corresponding Dirichlet boundary conditions 
\begin{gather}
u(x,y)\at_{\partial \Omega} = \widetilde{u}(x,y)\at_{\partial \Omega}  \, , \qquad \langle \bm{\zeta}(x,y) , \, \bm{\tau} \rangle \at_{\partial \Omega} =  \langle \widetilde{\bm{\zeta}}(x,y) , \, \bm{\tau} \rangle \at_{\partial \Omega} \, .
\label{eq:totalDir2} 
\end{gather}
The following force and moment are extracted, for details see \cref{ap:A},
\begin{gather}
f = -\dfrac{x \, y}{2} \left (\dfrac{y^2}{8} -\dfrac{x^2}{8}  \right ) \, , \\[2ex]  \bm{\omega} = 
\begin{bmatrix}
- (x^2 y^3)/16 + (25 x^2 y)/16 + (7y^3) /8 - 18y\\[2ex]
(x^3 y^2)/16 - (7 x^3)/8 - (25 x y^2)/16 + 18x
\end{bmatrix}  \, .
\end{gather}
\begin{figure}
	\centering
	\begin{subfigure}{0.3\linewidth}
		\centering
		\includegraphics[width=1\linewidth]{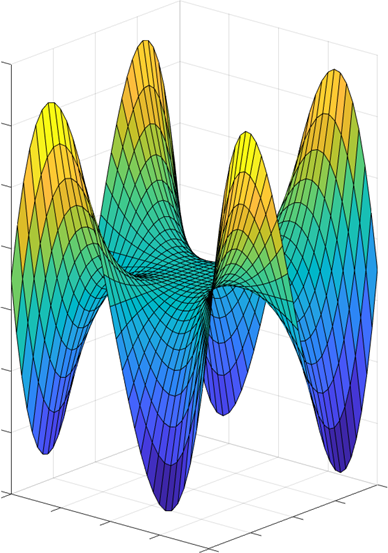}
		\caption{Analytical solution}
	\end{subfigure}
\begin{subfigure}{0.3\linewidth}
	\includegraphics[width=1\linewidth]{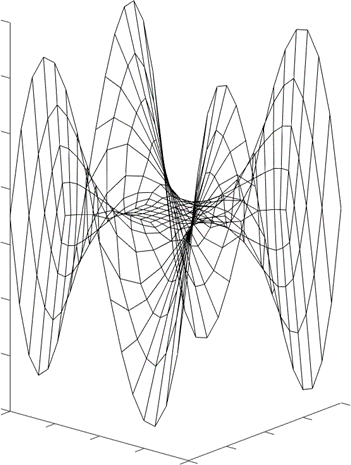}
	\caption{256 elements}
\end{subfigure}
	\begin{subfigure}{0.3\linewidth}
		\includegraphics[width=1\linewidth]{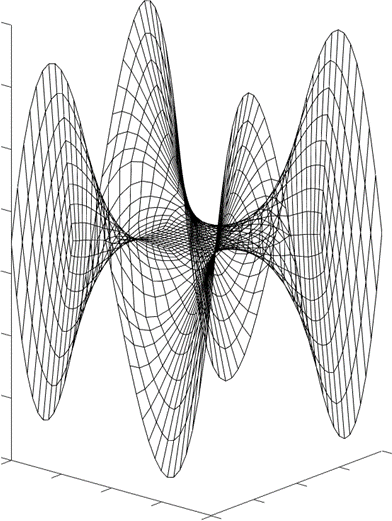}
		\caption{1024 elements}
	\end{subfigure}
	\caption{Displacement $u$ of the analytical and finite element solutions.}
	\label{fig:mesh}
\end{figure}
Consequently, the curl term is neither explicitly nor implicitly omitted.
We compare the displacement $u$ and the error $\|\widetilde{\bm{\zeta}} - \bm{\zeta}\|_{\Le}$ for both element formulations on an irregular mesh undergoing refinement, see \cref{fig:mesh,fig:gr3,fig:rot}.
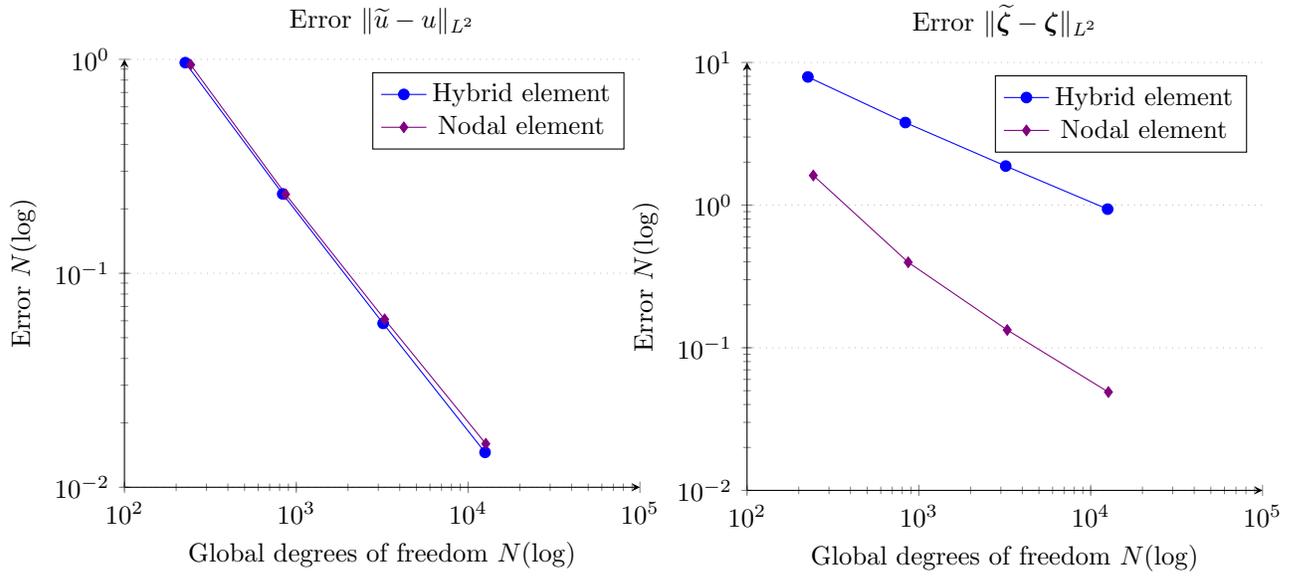
\begin{figure}
	\centering
	\begin{subfigure}{0.48\linewidth}
		\begin{tikzpicture}
		\begin{loglogaxis}[
		/pgf/number format/1000 sep={},
		title={Error $\| \widetilde{u} - u \|_{\Le}$},
		axis lines = left,
		xlabel={Global degrees of freedom $N(\log)$},
		ylabel={Error $N(\log)$},
		xmin=100, xmax=100000,
		ymin=0.01, ymax=1,
		xtick={100,1000,10000,100000},
		ytick={0.01,0.1,1},
		legend pos=north east,
		ymajorgrids=true,
		grid style=dotted,
		]
		\addplot[
		color=blue,
		mark=*,
		]
		coordinates {
			(226,0.964925632)
			(834,0.234951204)
			(3202,0.058336999)
			(12546,0.014559852)			
		};
		\addlegendentry{Hybrid element}
		\addplot[
		color=violet,
		mark=diamond*,
		] 
		coordinates {
			(243,0.943296378)
			(867,0.23395255)
			(3267,0.06094027)
			(12675,0.015970688)
		};
		\addlegendentry{Nodal element}
		\end{loglogaxis}
		\end{tikzpicture}
	\end{subfigure}
	\begin{subfigure}{0.48\linewidth}
		\begin{tikzpicture}
		\begin{loglogaxis}[
		/pgf/number format/1000 sep={},
		title={Error $\| \widetilde{\bm{\zeta}} - \bm{\zeta} \|_{\Le}$},
		axis lines = left,
		xlabel={Global degrees of freedom $N(\log)$},
		ylabel={Error $N(\log)$},
		xmin=100, xmax=100000,
		ymin=0.01, ymax=10,
		xtick={100,1000,10000,100000},
		ytick={0.01,0.1,1,10},
		legend pos=north east,
		ymajorgrids=true,
		grid style=dotted,
		]
		\addplot[
		color=blue,
		mark=*,
		]
		coordinates {
			(226,7.914976472)
			(834,3.791665503)
			(3202,1.878219253)
			(12546,0.937051989)
		};
		\addlegendentry{Hybrid element}
		\addplot[
		color=violet,
		mark=diamond*,
		] 
		coordinates {
			(243,1.613860924)
			(867,0.397180142)
			(3267,0.133204155)
			(12675,0.048883021)
		};
		\addlegendentry{Nodal element}
		\end{loglogaxis}
		\end{tikzpicture}
	\end{subfigure}
    \caption{Convergence behaviour of element formulations under mesh refinement.}
    \label{fig:gr3}
\end{figure}

As shown in \cref{fig:gr3}, both elements converge towards the analytical solution. However, we notice differences in the convergence rates, namely the nodal element converges faster in $\bm{\zeta}$.
\begin{figure}
	\centering
	\begin{subfigure}{0.3\linewidth}
		\centering
		\includegraphics[width=1\linewidth]{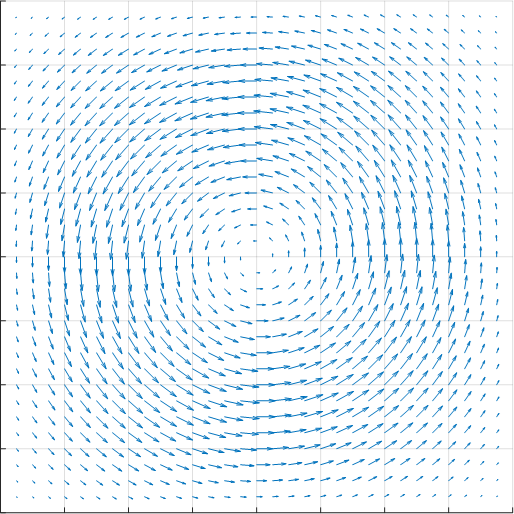}
		\caption{Analytical flux solution}
	\end{subfigure}
	\begin{subfigure}{0.3\linewidth}
		\centering
		\includegraphics[width=1\linewidth]{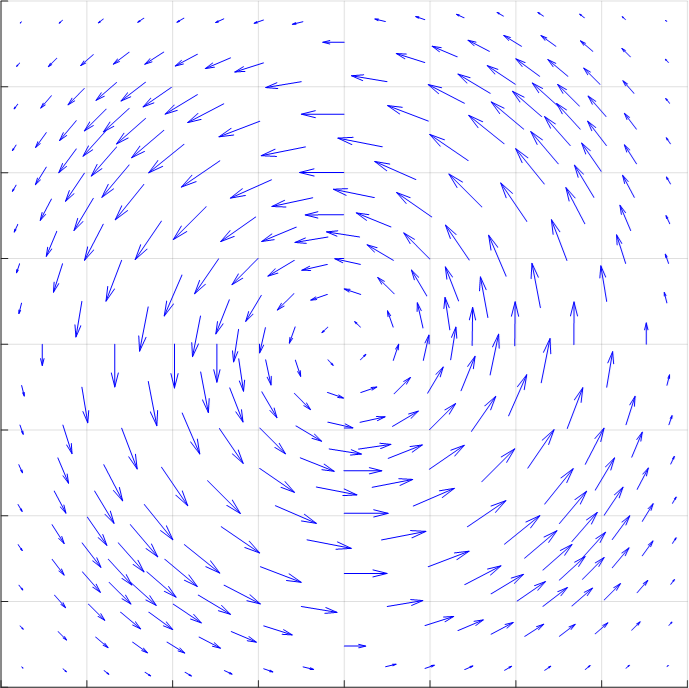}
		\caption{256 elements}
	\end{subfigure}
    \begin{subfigure}{0.3\linewidth}
    	\centering
    	\includegraphics[width=1\linewidth]{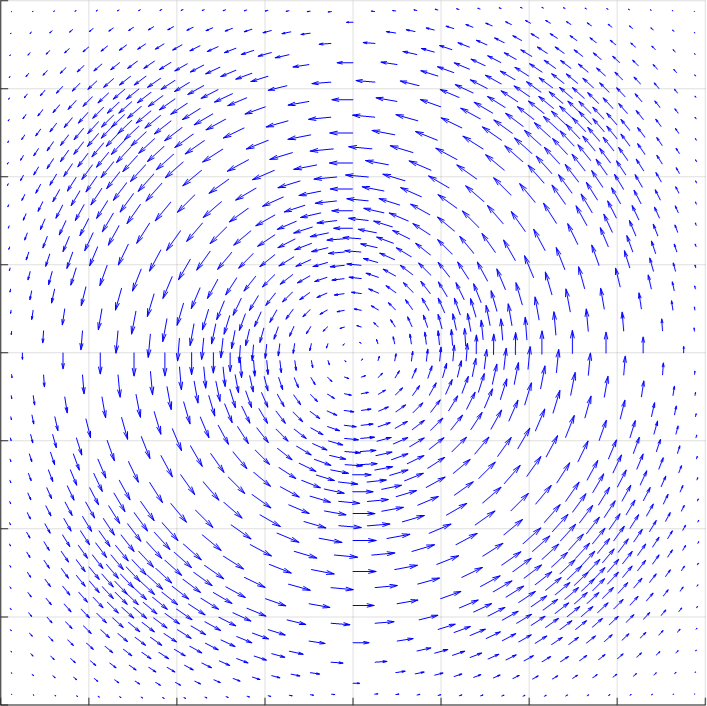}
    	\caption{1024 elements}
    \end{subfigure}
	\caption{Microdistortion $\bm{\zeta}$ of the analytical and finite element solutions on unstructured grids according to \cref{eq:rand}.}
	\label{fig:rot}
\end{figure}

\subsection{Solutions in $\Hc)$}
\label{subsec:solution_hcurl}
As $\Hc)$ is a larger space than $[\Hone]^2$, we have the relation $[\Hone]^2 \subset \Hc)$. Consequently, we can envision solutions belonging to $\Hc)$ and not $[\Hone]^2$. Such solutions fulfill the continuity of tangential components along element edges of $\Hc)$, but not the continuity of the normal component. Elements living in $[\Hone]^2$ require the continuity of both components. 

In the domain $\Omega = [-4,4] \times [-4,4]$ with $\Gamma_D^u=\partial\Omega$ and $\Gamma_D^\zeta=\emptyset$ we set $\mue=\muma=\mumi=\Lc=1$, the boundary conditions and external forces
\begin{align}
	u(-4,y) = u(4,y) = 0 \, , \qquad u(-2,y) = u(2,y) = -2 \, , \qquad u(0,y) = 2 \, , \qquad f = 0 \, , \qquad \vb{\omega} = 0 \, ,
\end{align}
for which the analytical solution reads
 \begin{align}
	\widetilde{u}(x,y) = \left \{ \begin{aligned}
	&-4-x  &\text{for}&  &-4 \leq &x \leq -2 \\
	&2+2x  &\text{for}&  &-2 < &x \leq 0 \\
	&2-2x  &\text{for}&  &0 < &x \leq 2 \\
	&x-4  &\text{for}&  &2 < &x \leq 4  
	\end{aligned}  \right . \, , \quad
	\widetilde{\bm{\zeta}} = \dfrac{\nabla \widetilde{u}}{2} = \left \{ \begin{aligned}
	&\begin{bmatrix}
	-0.5 & 0
	\end{bmatrix}^T & \text{for} && -4 \leq &x \leq -2 \\
	&\begin{bmatrix}
	1 & 0
	\end{bmatrix}^T & \text{for} && -2 < &x \leq 0 \\
	&\begin{bmatrix}
	-1 & 0
	\end{bmatrix}^T & \text{for} && 0 < &x \leq 2 \\
	&\begin{bmatrix}
	0.5 & 0
	\end{bmatrix}^T & \text{for} && 2 < &x \leq 4  
	\end{aligned}  \right . \, ,
	\label{eq:tent} 
\end{align}
where $\widetilde{\bm{\zeta}}$ follows from \cref{eq:strong1,eq:strong2}.
Note that the boundary data of $\widetilde{\bm{\zeta}}$ jumps and is therefore not in $\mathit{H}^{\nicefrac{1}{2}}(\Gamma_D)$. Consequently, the problem cannot be posed with $\bm{\zeta} \in [\Hone(\Omega)]^2$ if we set $\Gamma_D^\zeta=\partial\Omega$, see \cref{rem:discussion_h1_hcurl}. For $\bm{\zeta} \in \Hc,\Omega)$ the problem could be posed as $\langle \widetilde{\bm{\zeta}} , \, \bm{\tau} \rangle \in \Le(\partial\Omega) \subset \mathit{H}^{\nicefrac{-1}{2}}(\partial\Omega)$. 

We test both elements on an irregular mesh undergoing refinement \cref{fig:tent}.
\begin{figure}
	\centering
	\begin{subfigure}{0.3\linewidth}
		\centering
		\includegraphics[width=1\linewidth]{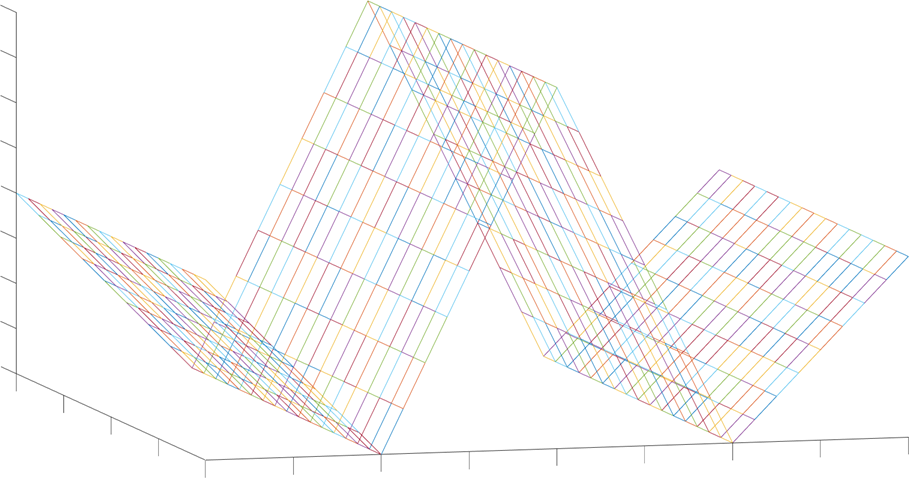}
		\caption{Analytical displacement solution}
	\end{subfigure}
	\begin{subfigure}{0.3\linewidth}
		\centering
		\includegraphics[width=1\linewidth]{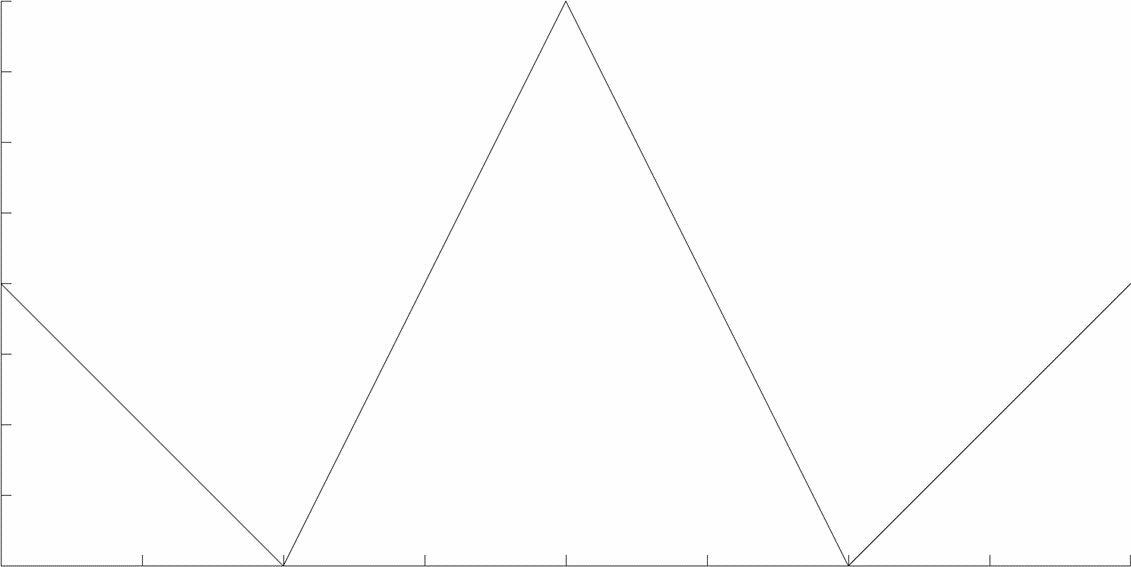}
		\caption{Solution with 46 hybrid elements}
	\end{subfigure}
	\begin{subfigure}{0.3\linewidth}
		\centering
		\includegraphics[width=1\linewidth]{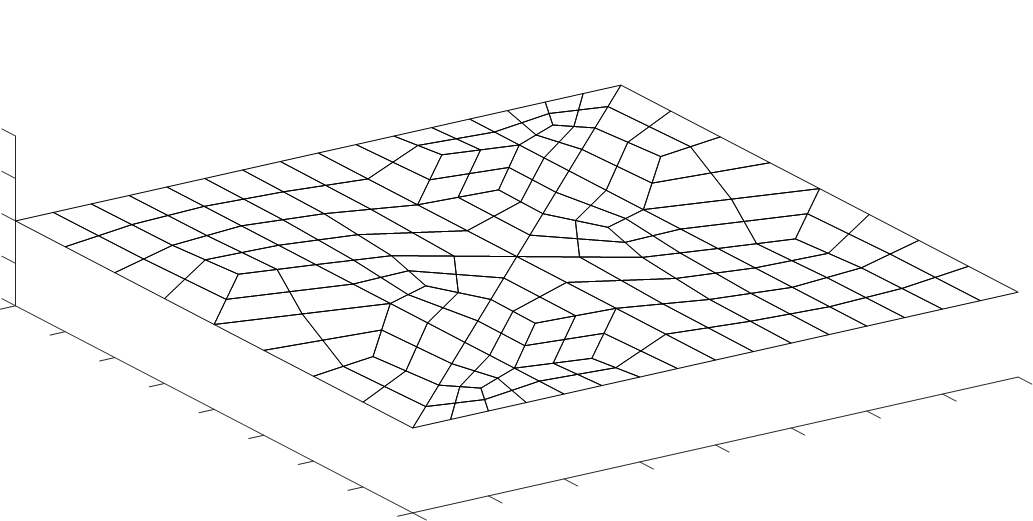}
		\caption{Example mesh with 184 elements}
	\end{subfigure}

\begin{subfigure}{0.3\linewidth}
	\centering
	\includegraphics[width=1\linewidth]{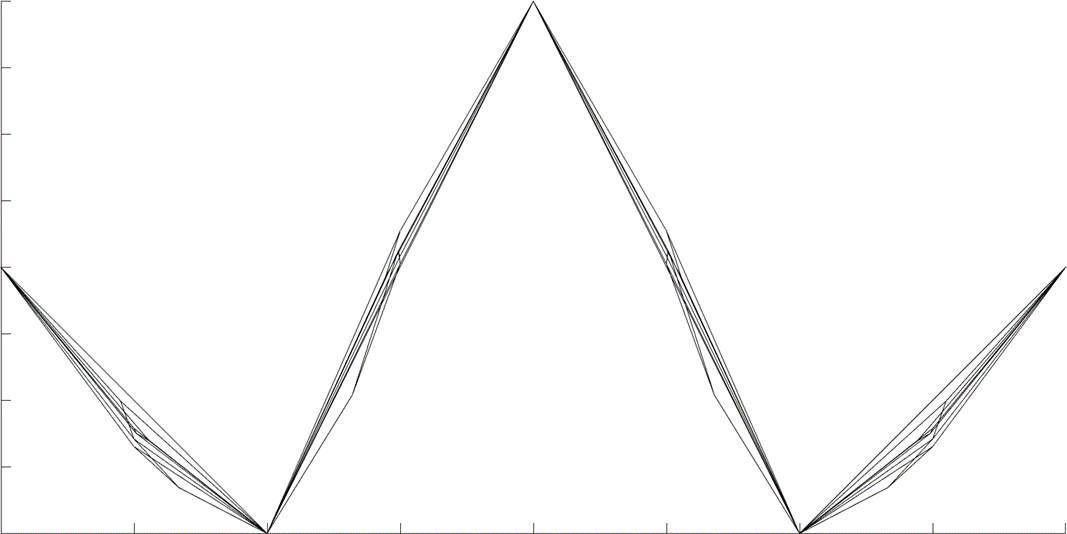}
	\caption{Solution with 46 nodal elements}
\end{subfigure}
\begin{subfigure}{0.3\linewidth}
	\centering
	\includegraphics[width=1\linewidth]{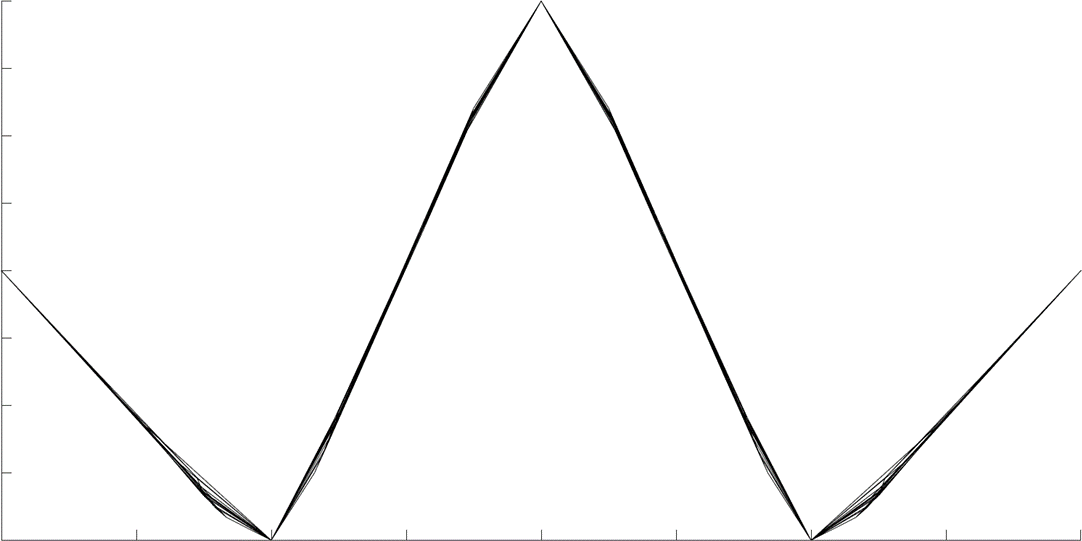}
	\caption{Solution with 184 nodal elements}
\end{subfigure}
\begin{subfigure}{0.3\linewidth}
	\centering
	\includegraphics[width=1\linewidth]{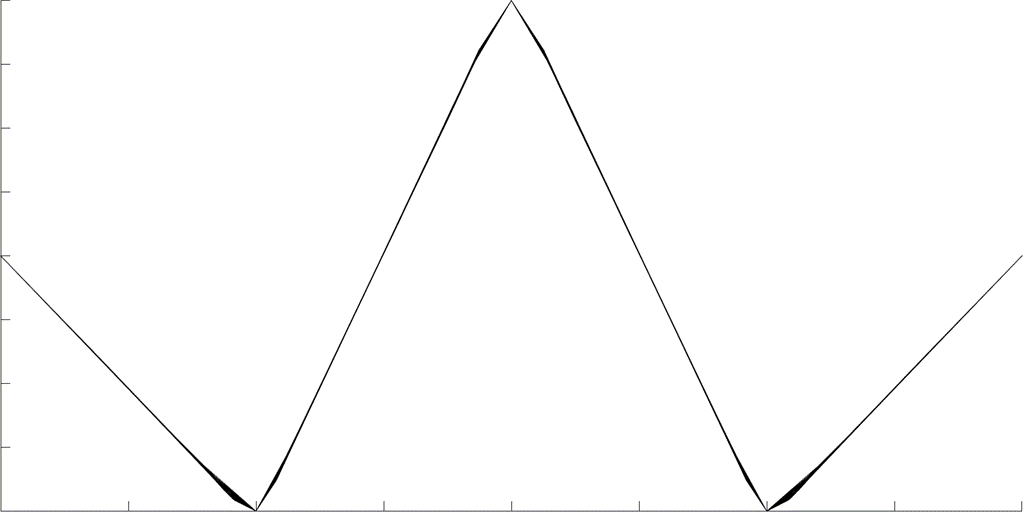}
	\caption{Solution with 736 nodal elements}
\end{subfigure}

	\caption{Analytical solution and finite element front view (\AS{$x-z$ plane}) for solutions of \cref{eq:tent}.}
	\label{fig:tent}
\end{figure}
We note the hybrid element finds the exact solution immediately with a coarse mesh, whereas the nodal element requires a much higher level of refinement in order to deliver a viable approximation. The nodal element localizes the error due to the discontinuity further with each refinement as seen in \cref{fig:tent_flux}. 
\begin{figure}
	\centering
	\begin{subfigure}{0.3\linewidth}
		\centering
		\includegraphics[width=1\linewidth]{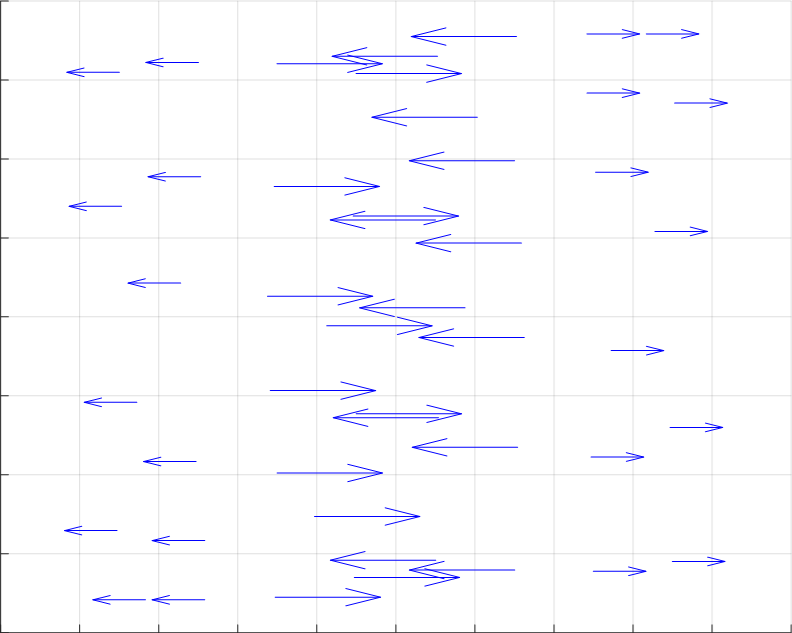}
		\caption{Microdistortion with 46 hybrid elements}
	\end{subfigure}
	\begin{subfigure}{0.3\linewidth}
		\centering
		\includegraphics[width=1\linewidth]{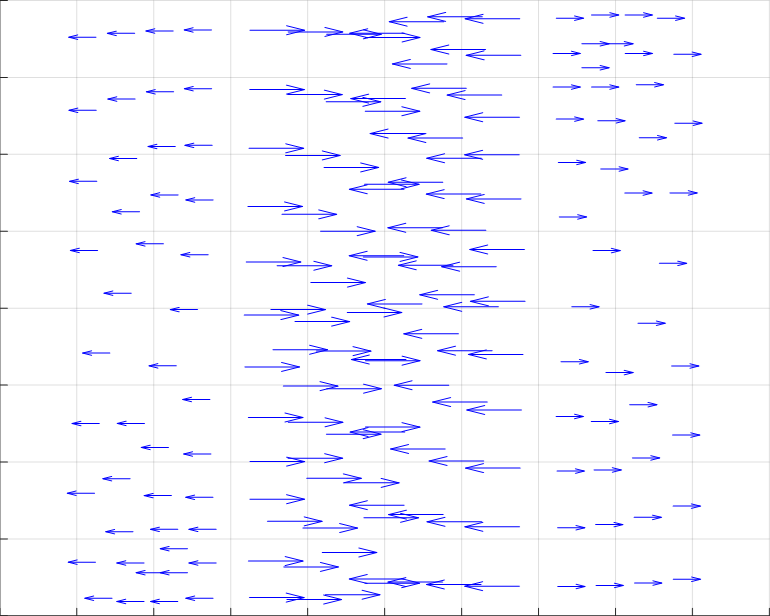}
		\caption{Microdistortion with 184 hybrid elements}
	\end{subfigure}
	\begin{subfigure}{0.3\linewidth}
		\centering
		\includegraphics[width=1\linewidth]{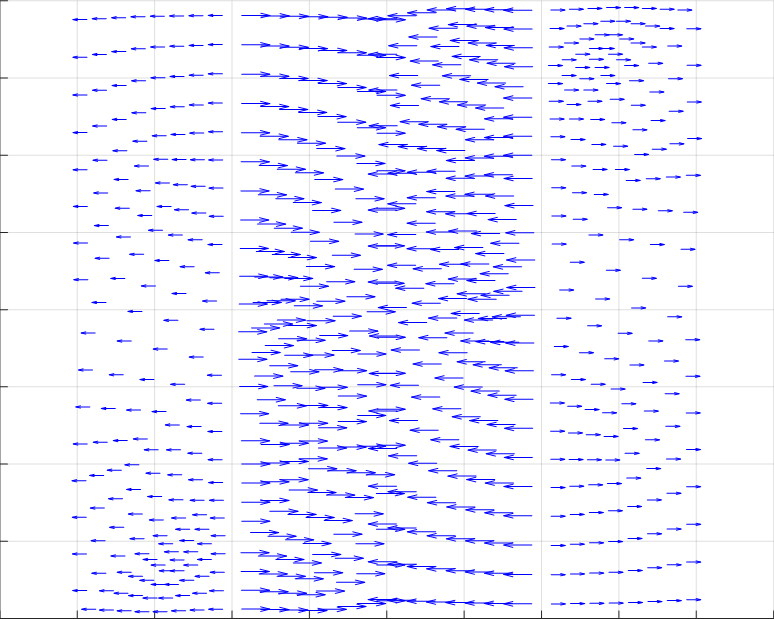}
		\caption{Microdistortion with 736 hybrid elements}
	\end{subfigure}
	
	\begin{subfigure}{0.3\linewidth}
		\centering
		\includegraphics[width=1\linewidth]{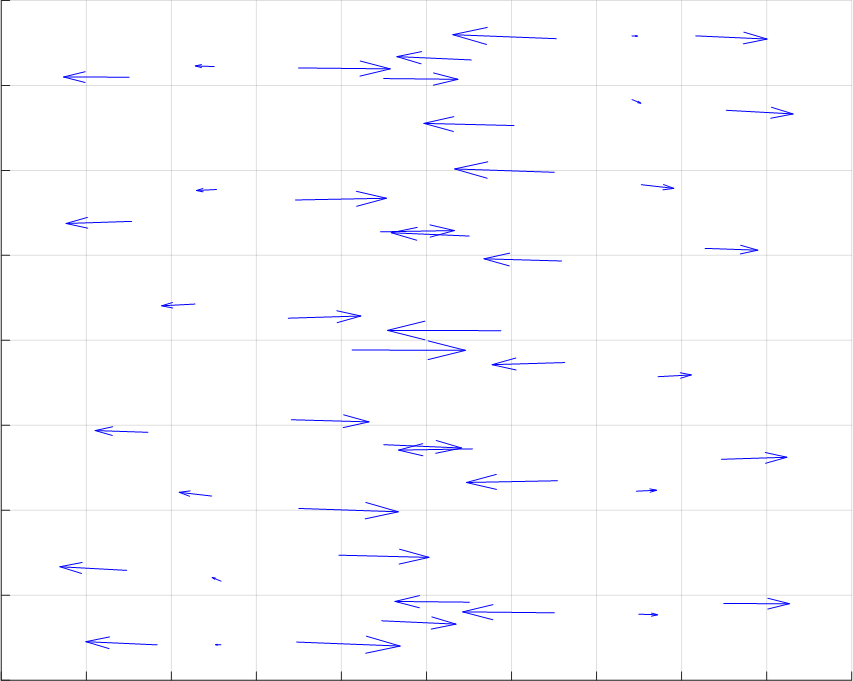}
		\caption{Microdistortion with 46 nodal elements}
	\end{subfigure}
	\begin{subfigure}{0.3\linewidth}
		\centering
		\includegraphics[width=1\linewidth]{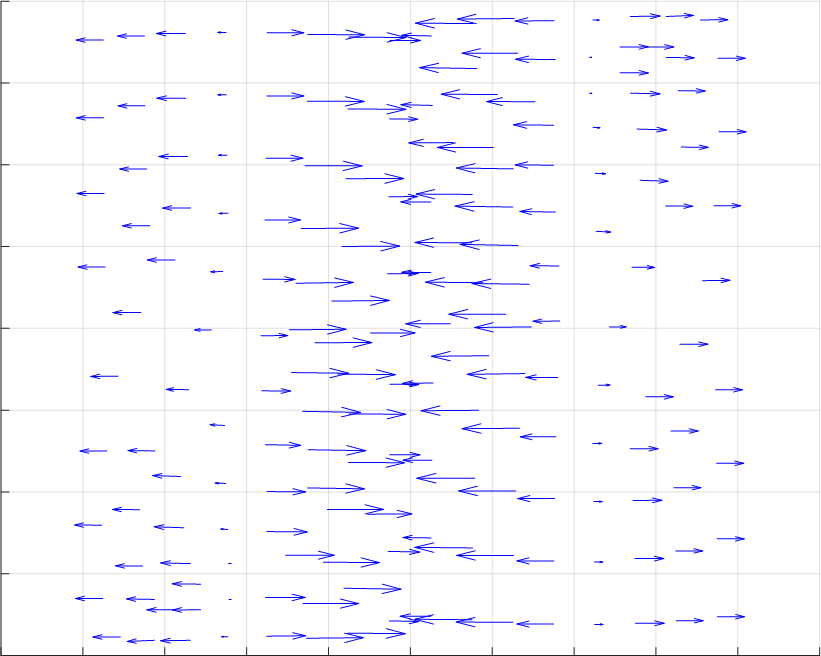}
		\caption{Microdistortion with 184 nodal elements}
	\end{subfigure}
	\begin{subfigure}{0.3\linewidth}
		\centering
		\includegraphics[width=1\linewidth]{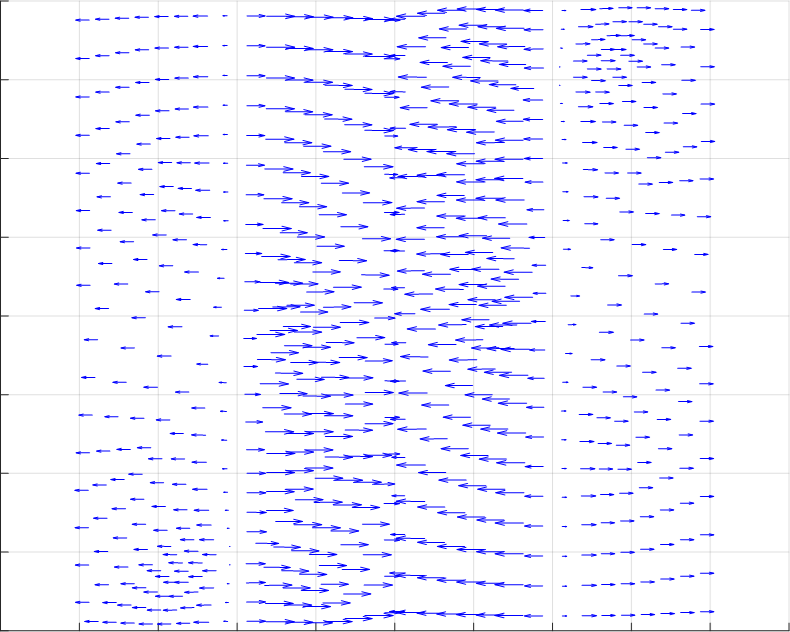}
		\caption{Microdistortion with 736 nodal elements}
	\end{subfigure}
	
	\caption{Finite element solutions of the microdistortion for \cref{eq:tent} for both formulations.}
	\label{fig:tent_flux}
\end{figure}
\begin{figure}
	\centering
	\begin{subfigure}{0.48\linewidth}
		\begin{tikzpicture}
		\begin{loglogaxis}[
		/pgf/number format/1000 sep={},
		title={Error $\| \widetilde{u} - u \|_{\Le}$ over the domain},
		axis lines = left,
		xlabel={Global degrees of freedom $N(\log)$},
		ylabel={Error $N(\log)$},
		xmin=100, xmax=10000,
		ymin=0.1, ymax=1,
		xtick={100,1000,10000},
		ytick={0.1,1},
		legend pos=north east,
		ymajorgrids=true,
		grid style=dotted,
		]
		\addplot[
		color=violet,
		mark=diamond*,
		]
		coordinates {
			(177,0.778234094)
			(627,0.617634834)
			(2355,0.367635487)
			(9123,0.20018801)
		};
		\addlegendentry{Nodal element}
		\end{loglogaxis}
		\end{tikzpicture}
	\end{subfigure}
	\begin{subfigure}{0.48\linewidth}
		\begin{tikzpicture}
		\begin{loglogaxis}[
		/pgf/number format/1000 sep={},
		title={Error $\| \widetilde{\bm{\zeta}} - \bm{\zeta} \|_{\Le}$ over the domain},
		axis lines = left,
		xlabel={Global degrees of freedom $N(\log)$},
		ylabel={Error $N(\log)$},
		xmin=100, xmax=10000,
		ymin=1, ymax=10,
		xtick={100,1000,10000},
		ytick={1,10},
		legend pos=north east,
		ymajorgrids=true,
		grid style=dotted,
		]
		\addplot[
		color=violet,
		mark=diamond*,
		]
		coordinates {
			(177,3.013678686)
			(627,2.274377726)
			(2355,1.628072962)
			(9123,1.160084183)
		};
		\addlegendentry{Nodal element}
		
		\addplot[
		dashed,
		color=black,
		mark=none,
		]
		coordinates {
			(177,0.6*6.854040906660596)
			(627,0.6*4.996007980849785)
			(2355,0.6*3.588742358850372)
			(9123,0.6*2.558029743638137)
		};
		
		\end{loglogaxis}
		\draw (3.5,3) node[anchor=north west] {$\mathcal{O}(h^{\nicefrac{1}{2}})$};
		\end{tikzpicture}
	\end{subfigure}
	\caption{Convergence behaviour of element formulations under mesh refinement.}
	\label{fig:tent_con}
\end{figure}
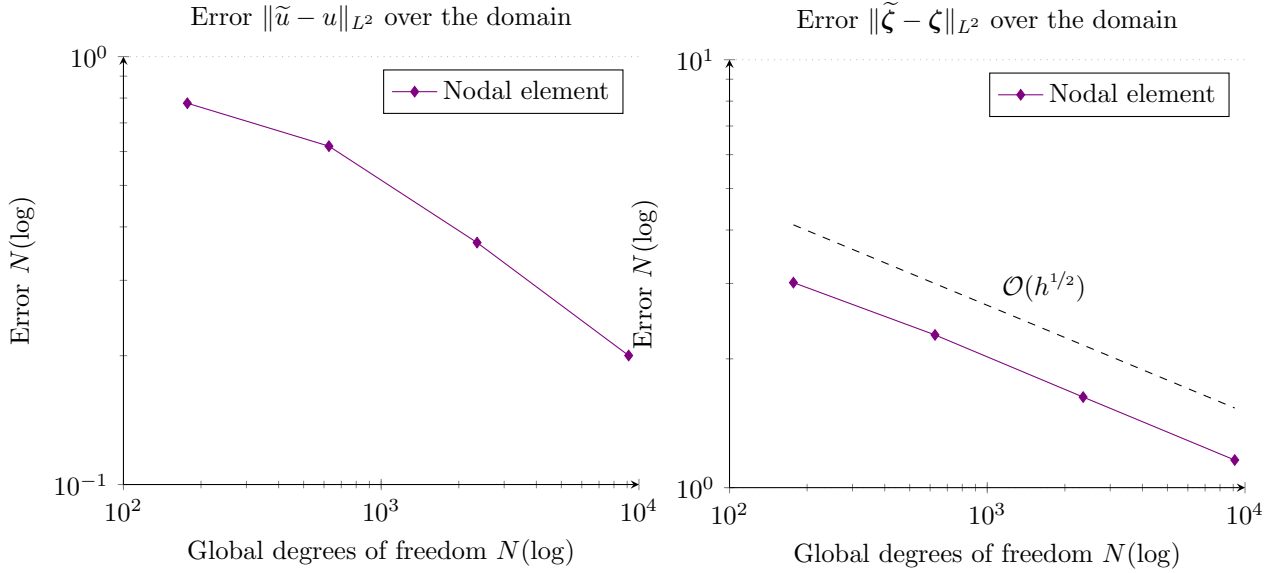
The convergence graph in \cref{fig:tent_con} depicts the slow sub-optimal convergence of the nodal element, compare \cref{eq:conv_rate_primal}. Note, the error in the hybrid element for the same meshes is always at a factor $10^{-15}$ for both $u$ and $\bm{\zeta}$. Due to the higher continuity conditions of the nodal element, it could never find the analytical solution, but would converge further towards it with each refinement. 

We present a second example allowing us to compare the convergence rates for both formulations. Let $\Omega=[0,1] \times [0,1]$, $\mue=\muma=\mumi=\Lc=1$,  and $\Gamma_D^u=\Gamma_D^\zeta=\partial\Omega$. For the given exact solution $\{\widetilde{u},\widetilde{\bm{\zeta}}\}\in \Honez(\Omega)\times\Hcz\Omega)$ 
\begin{align}
	\widetilde{u}(x,y) = \exp(1-x)y(1-y)\left \{ \begin{matrix}
	x & \text{for} & x \leq 0.5 \\
	1-x & \text{for} & x > 0.5  
	\end{matrix}  \right . \, , \quad
	\widetilde{\bm{\zeta}} = \nabla \widetilde{u} \, ,
	\label{eq:hcex}
\end{align}
the corresponding boundary conditions and external forces result in
\begin{align}
u(x,y)\at_{\partial \Omega}=0,\qquad \langle \bm{\zeta} ,\, \bm{\tau} \rangle \at_{\partial \Omega}=0,\qquad f=0,\qquad \bm{\omega}=2\widetilde{\bm{\zeta}} \, .
\end{align}
Here, the boundary conditions are compatible with $\bm{\zeta}\in [\Hone(\Omega)]^2$, but the exact solution is only in $\Hc,\Omega)$, not in $[\Hone(\Omega)]^2$. We use structured quadrilateral meshes (see \cref{fig:hcex}) resolving the interface at $x=0.5$, where the normal component of the exact solution of $\bm{\zeta}$ jumps, with linear, quadratic and cubic polynomials for the nodal elements. We observe that higher polynomial degrees do not increase the convergence rate and only sub-optimal root-convergence is achieved (see \cref{fig:tent2_con}). For linear and quadratic ansatz functions in the primal $\Hc)$ method we observe optimal convergence rates.
\begin{figure}
	\centering
	\begin{tikzpicture}
	\begin{loglogaxis}[
	/pgf/number format/1000 sep={},
	title={Error $\| \widetilde{\bm{\zeta}} - \bm{\zeta} \|_{\Hc)}$ over the domain},
	axis lines = left,
	xlabel={Global degrees of freedom $N(\log)$},
	ylabel={Error $N(\log)$},
	xmin=10, xmax=10000000,
	ymin=0.000001, ymax=5,
	xtick={100,1000,10000,100000,1000000,10000000},
	ytick={0.000001,0.00001,0.0001,0.001,0.01,0.1,1},
	legend pos=south west,
	ymajorgrids=true,
	grid style=dotted,
	]
	
	\addplot[color=violet, mark=diamond*]
	coordinates {
		(27, 1.0481213568936905)
		(75, 0.7767762484993705)
		(243, 0.5102462550613845)
		(867, 0.3317527388592938)
		(3267, 0.2204458165602305)
		(12675, 0.14992199699565636)
		(49923, 0.10368905661996256)
		(198147, 0.0724532449654316)
	};
	\addlegendentry{Nodal $k=1$}

	\addplot[color=red, mark=square]
	coordinates {
		(75, 0.38279125615224696)
		(243, 0.2860257803028133)
		(867, 0.20289234255480684)
		(3267, 0.14353689280570445)
		(12675, 0.10150822818022964)
		(49923, 0.07177935132630152)
		(198147, 0.05075605683141311)
		(789507, 0.03589002134894577)
	};
	\addlegendentry{Nodal $k=2$}

	\addplot[color=teal, mark=x]
	coordinates {
		(147, 0.24647035908987056)
		(507, 0.17299925370584796)
		(1875, 0.122189265180536)
		(7203, 0.0863804735830759)
		(28227, 0.061076598108044285)
		(111747, 0.04318703370431983)
		(444675, 0.030537730404252658)
		(1774083, 0.021593416069601412)
	};
	\addlegendentry{Nodal $k=3$}

	\addplot[color=orange, mark=o]
	coordinates {
		(21, 0.4968017363137966)
		(81, 0.23329511606525355)
		(297, 0.11477727883444597)
		(1113, 0.05715759775160357)
		(4281, 0.028550085995806752)
		(16761, 0.014271462115147935)
		(66297, 0.007135283829055237)
		(263673, 0.003567586028810376)
	};
	\addlegendentry{Hybrid $k=1$}

	\addplot[color=blue, mark=+]
	coordinates {
		(65, 0.024058829214126297)
		(241, 0.00605917357530029)
		(905, 0.0015166668251078207)
		(3481, 0.00037926014064996636)
		(13625, 9.482021896516824e-05)
		(53881, 2.3705359542877195e-05)
		(214265, 5.926358356945469e-06)
		(854521, 1.4815907260563721e-06)
	};
	\addlegendentry{Hybrid $k=2$}
	
	\addplot[dashed,color=black, mark=none]
	coordinates {
		(1000, 0.5*1.0669676460233537)
		(2000000, 0.5*0.15954887690834965)
	};
	
	\addplot[dashed,color=black, mark=none]
	coordinates {
		(100, 0.1)
		(100000, 0.003162277660168379)
	};
	
	\addplot[dashed,color=black, mark=none]
	coordinates {
		(500, 0.0059999999999999992)
		(300000, 9.999999999999997e-06)
	};
	
	\end{loglogaxis}
	\draw (3.5,5.2) node[anchor=north west] {$\mathcal{O}(h^{\nicefrac{1}{2}})$};
	\draw (2.2,3.6) node[anchor=north west] {$\mathcal{O}(h)$};
	\draw (3.5,2.5) node[anchor=north west] {$\mathcal{O}(h^2)$};
	\end{tikzpicture}

	\caption{Convergence rates of the microdistortion on both element formulations across multiple polynomial degrees undergoing mesh refinement.}
	\label{fig:tent2_con}
\end{figure}
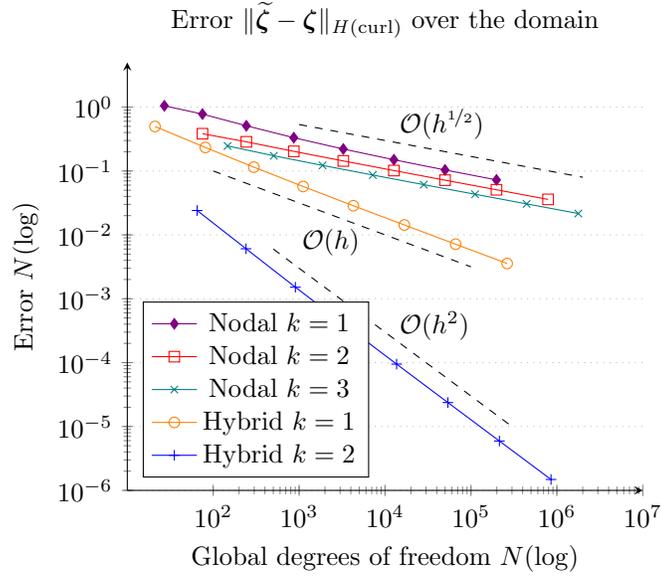

\begin{figure}
	\centering
	\begin{subfigure}{0.3\linewidth}
		\centering
		\includegraphics[width=1\linewidth]{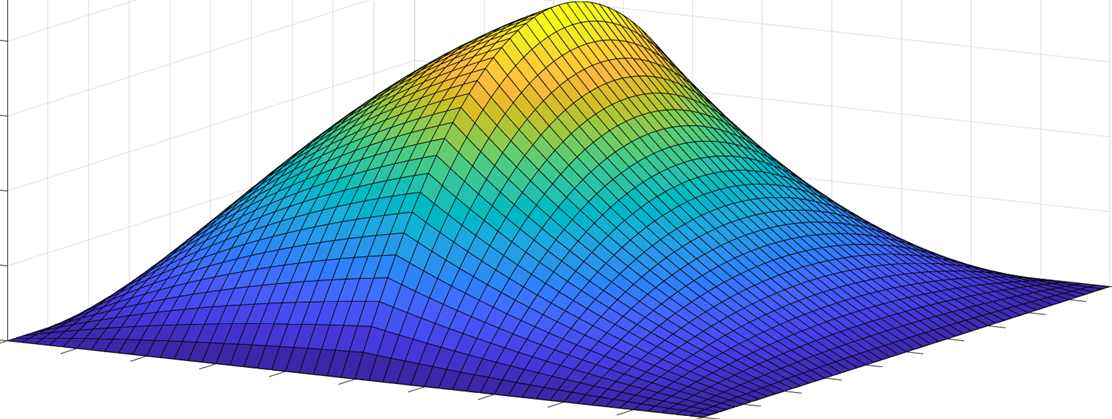}
		\caption{Displacement analytical solution}
	\end{subfigure}
	\begin{subfigure}{0.3\linewidth}
		\centering
		\includegraphics[width=1\linewidth]{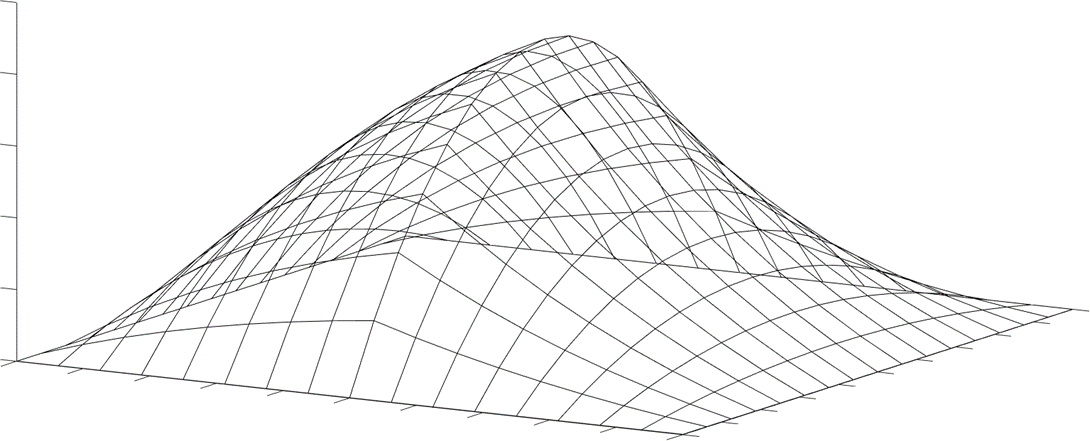}
		\caption{Displacement with 256 hybrid elements}
	\end{subfigure}
	\begin{subfigure}{0.3\linewidth}
		\centering
		\includegraphics[width=1\linewidth]{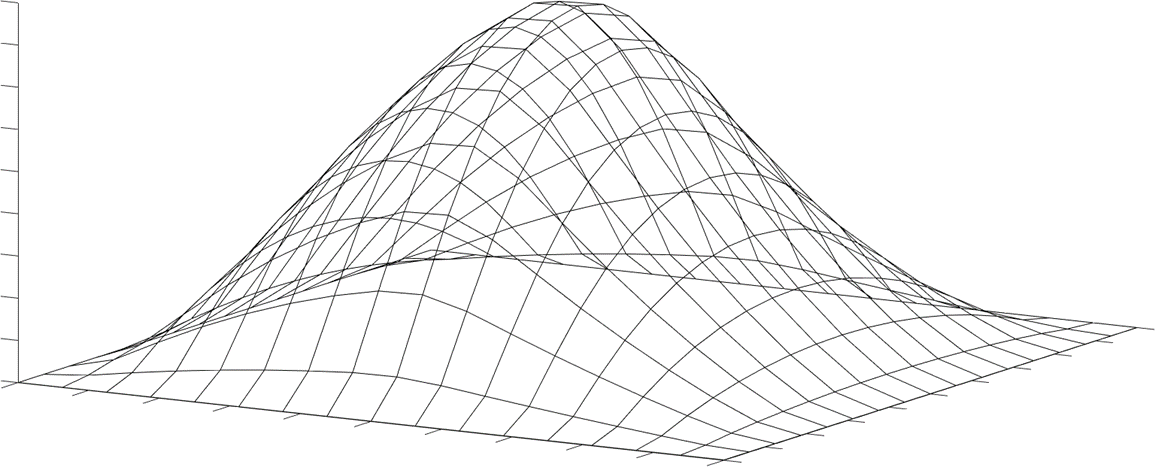}
		\caption{Displacement with 256 nodal elements}
	\end{subfigure}
	
	\begin{subfigure}{0.3\linewidth}
		\centering
		\includegraphics[width=1\linewidth]{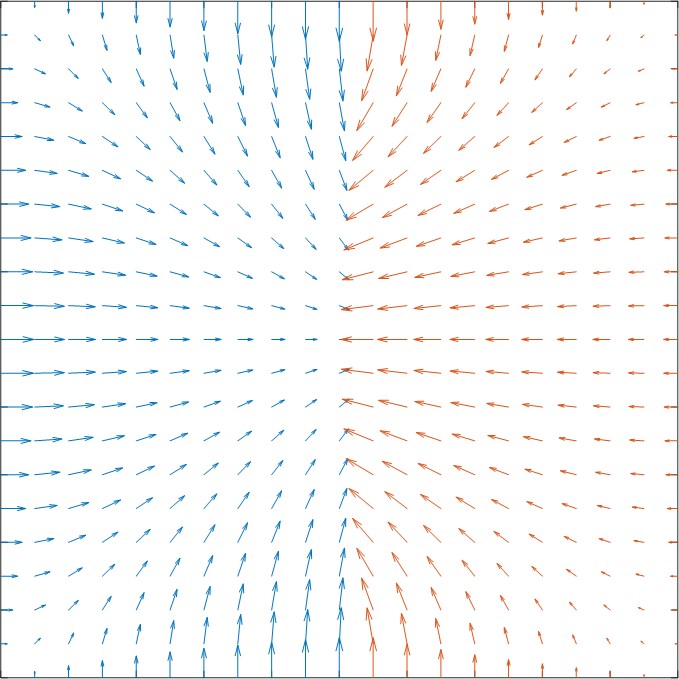}
		\caption{Microdistortion analytical solution}
	\end{subfigure}
	\begin{subfigure}{0.3\linewidth}
		\centering
		\includegraphics[width=1\linewidth]{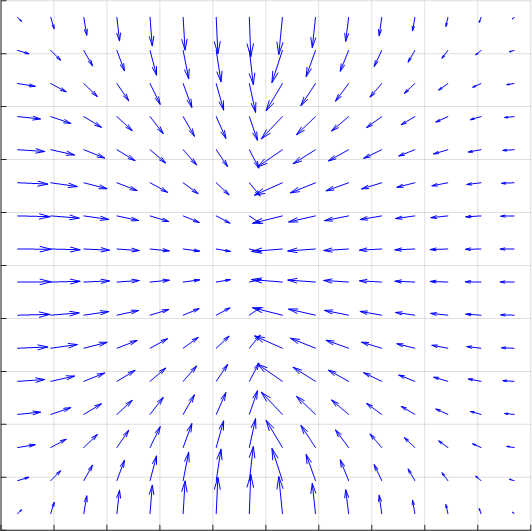}
		\caption{Microdistortion with 256 hybrid elements}
	\end{subfigure}
	\begin{subfigure}{0.3\linewidth}
		\centering
		\includegraphics[width=1\linewidth]{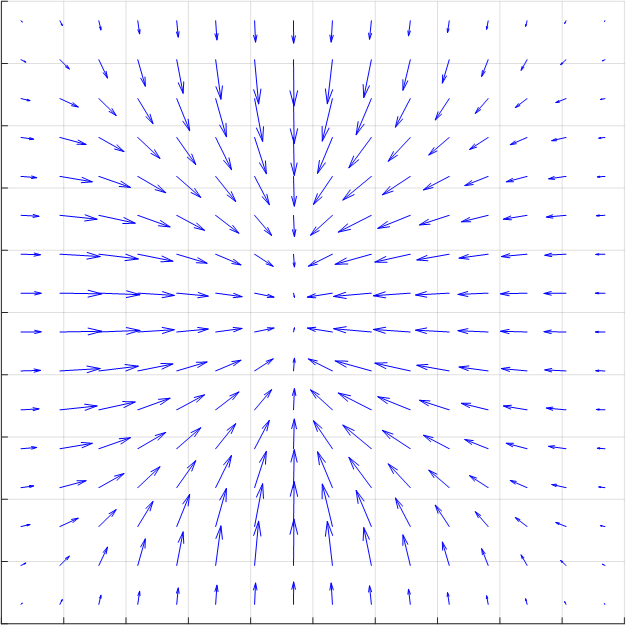}
		\caption{Microdistortion with 256 nodal elements}
	\end{subfigure}
	
	\caption{Analytical and finite element solutions of the displacement and microdistortion fields according to \cref{eq:hcex}.}
	\label{fig:hcex}
\end{figure}

\subsection{Convergence for $\Lc \to 0$}
As mentioned in \cref{ch:2}, the characteristic length $\Lc$ represents an important term in the relaxed micromorphic theory. This scalar governs the relation of the relaxed micromorphic continuum to the standard Cauchy continuum. In the previous examples we have been able to generate stable results for the case $\Lc = 1$. In this example we consider the limit $\Lc \to 0$, which can be interpreted as a highly homogenous material. In the $\Lc = 0$ setting, the relaxed micromorphic continuum retrieves the results of the classical Cauchy continuum, no external moments $\bm{\omega}$ occur and the microdistortion $\bm{\zeta}$ lives in $[\Le(\Omega)]^2$. This results in the emergence of a single Poisson equation for $u$ (see \cref{rem:membrane}), being an analogue of the standard membrane partial differential equation.
We define the domain $\Omega = [-5, \, 5] \times [-5, \, 5]$ with $\mue, \, \mumi = 1 ,\, \Lc = 0$ and the imposed displacement
\begin{align}
\widetilde{u} (x,y) = 2 - \sin(x)^2 + \cos(x)^2 - \sin(y)^2 + \cos(y)^2 \, .
\label{eq:ana}
\end{align}
We use $\widetilde{u}$ to recover the analytical solution for $\widetilde{\bm{\zeta}}$
	\begin{align}
	\widetilde{\bm{\zeta}} = \dfrac{\mue}{\mumi + \mue} \nabla \widetilde{u} =  \begin{bmatrix}
	-2 \cos(x) \sin(x)\\[2ex]
	 -2 \cos(y) \sin(y)
	\end{bmatrix} \, ,
	\label{eq:ana2}
	\end{align}
	and the resulting right-hand side
	\begin{align}
		f = 4 \, (\cos(x)^2 + \cos(y)^2 - \sin(x)^2 - \sin(y)^2) \, .
	\end{align}
	Note, since we require $\bm{\zeta} \in [\Le(\Omega)]^2$, no boundary conditions can be prescribed for $\bm{\zeta}$. The microdistortion field $\bm{\zeta}$ can always be approximated using either $\Hc)$ or $[\Hone]^2$ elements. However, the direct use of discontinuous $[\Le]^2$ elements for $\bm{\zeta}$ requires less computation and can also capture gradient fields. With \cref{thm:zeta} we have for $\bm{\omega}=0$ the regularity result that $\bm{\zeta}$ is in fact a gradient field and thus $\bm{\zeta}\in\Hc,\Omega)$, which confirms to use N{\'e}d{\'e}lec elements without risk of sub-optimal convergence rates, compare \cref{subsec:solution_hcurl}.
	The finite element solution converges towards the analytical solution as expected with optimal rate, see \cref{fig:lc0,fig:mem_con}.
	
	\begin{figure}
		\centering
		\begin{subfigure}{0.3\linewidth}
			\centering
			\includegraphics[width=1\linewidth]{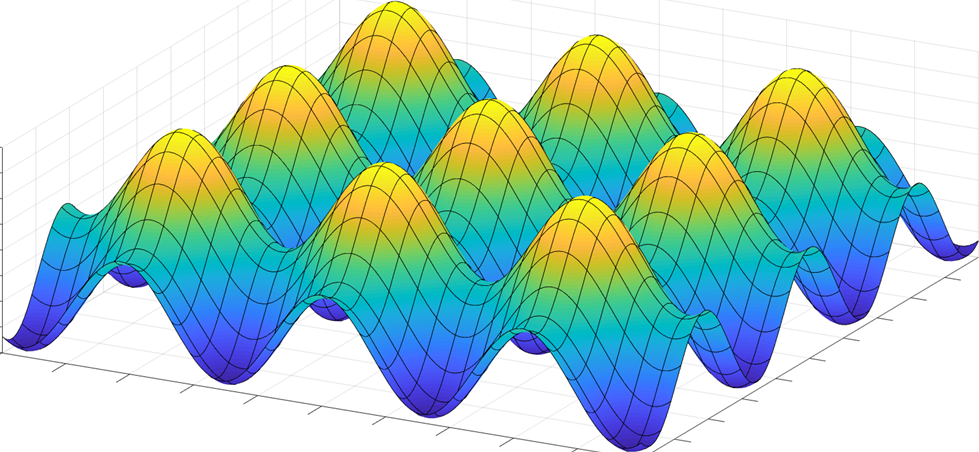}
			\caption{Displacement analytical solution}
		\end{subfigure}
		\begin{subfigure}{0.3\linewidth}
			\centering
			\includegraphics[width=1\linewidth]{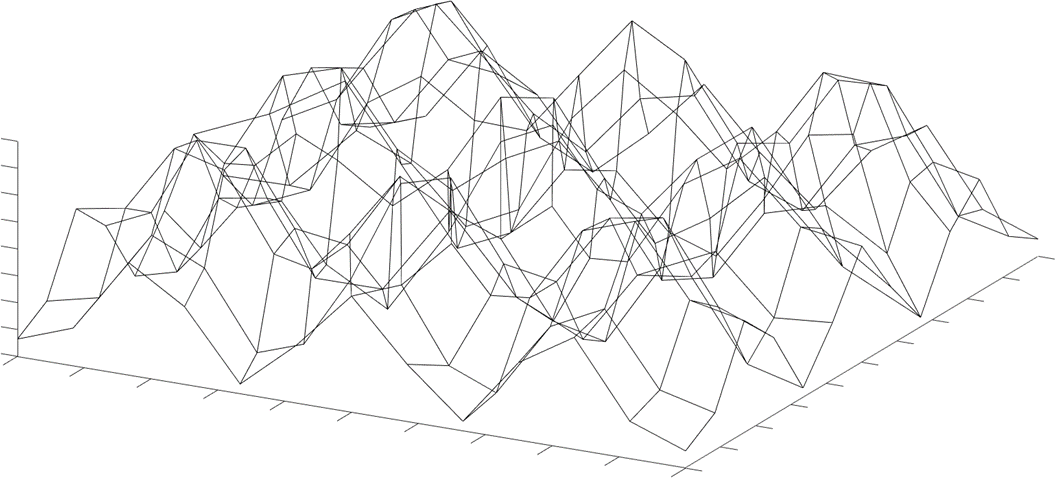}
			\caption{Displacement with 180 hybrid elements}
		\end{subfigure}
		\begin{subfigure}{0.3\linewidth}
			\centering
			\includegraphics[width=1\linewidth]{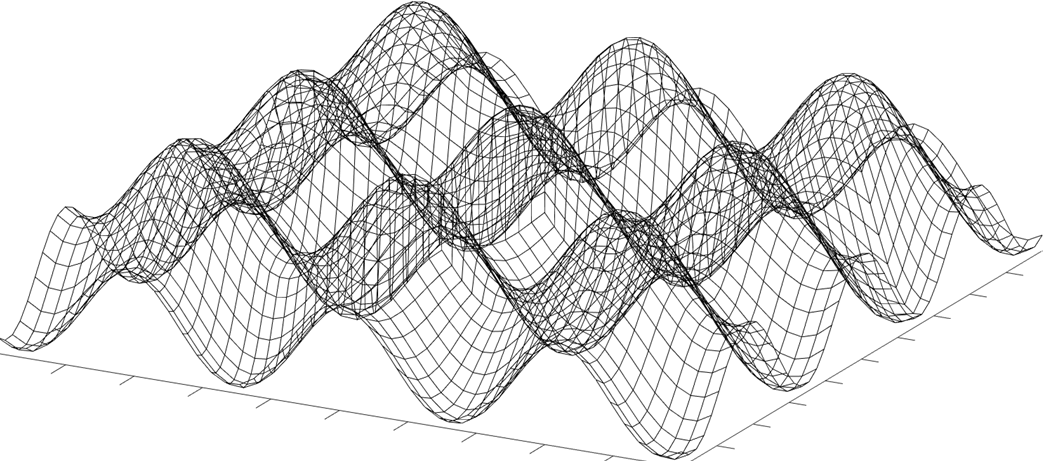}
			\caption{Displacement with 2880 hybrid elements}
		\end{subfigure}
		
		\begin{subfigure}{0.3\linewidth}
			\centering
			\includegraphics[width=1\linewidth]{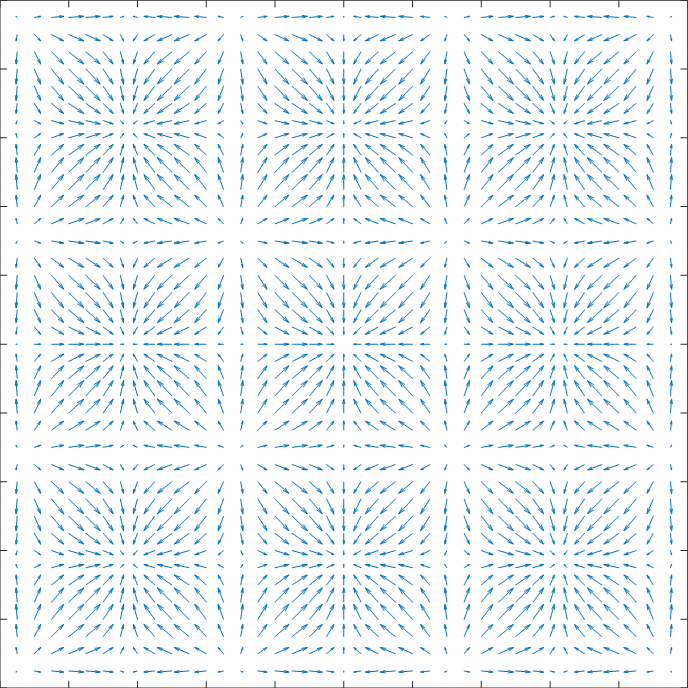}
			\caption{Microdistortion analytical solution}
		\end{subfigure}
		\begin{subfigure}{0.3\linewidth}
			\centering
			\includegraphics[width=1\linewidth]{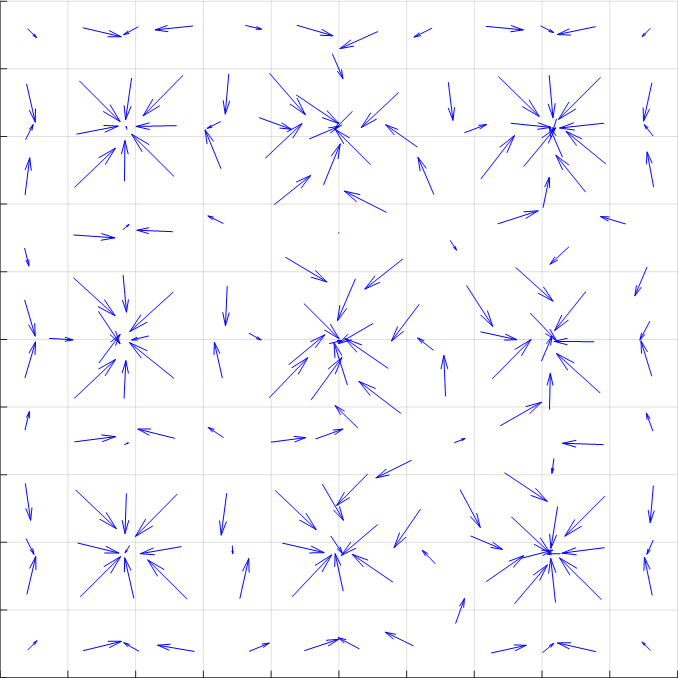}
			\caption{Microdistortion with 180 hybrid elements}
		\end{subfigure}
		\begin{subfigure}{0.3\linewidth}
			\centering
			\includegraphics[width=1\linewidth]{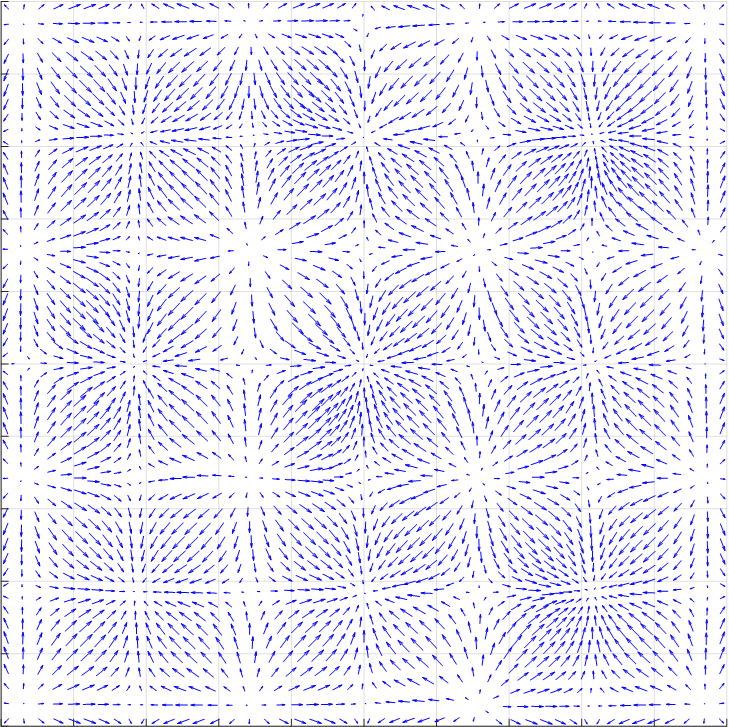}
			\caption{Microdistortion with 2880 hybrid elements}
		\end{subfigure}
		
		\caption{Analytical solutions and finite element solutions on unstructured grids of the displacement and microdistortion fields according to \cref{eq:ana,eq:ana2}.}
		\label{fig:lc0}
	\end{figure}
	
	\begin{figure}
		\centering
		\begin{subfigure}{0.48\linewidth}
			\begin{tikzpicture}
			\begin{loglogaxis}[
			/pgf/number format/1000 sep={},
			title={Error $\| \widetilde{u} - u \|_{\Le}$ over the domain},
			axis lines = left,
			xlabel={Global degrees of freedom $N(\log)$},
			ylabel={Error $N(\log)$},
			xmin=100, xmax=10000,
			ymin=0.1, ymax=10,
			xtick={100,1000,10000},
			ytick={0.1,1,10},
			legend pos=north east,
			ymajorgrids=true,
			grid style=dotted,
			]
			\addplot[
			color=violet,
			mark=diamond*,
			]
			coordinates {
				(161,	5.873579504)
			    (590,	1.890223949)
				(2258,	0.497215006)
				(8834,	0.126110675)
				
			};
			\addlegendentry{Hybrid element}
			\addplot[
			dashed,
			color=black,
			mark=none,
			]
			coordinates {
				(500,1.5*2.5)
				(7000,1.5*0.17857142857142)
			};
			
			\end{loglogaxis}
			\draw (4,3.7) node[anchor=north west] {$\mathcal{O}(h^2)$};
			\end{tikzpicture}
		\end{subfigure}
		\begin{subfigure}{0.48\linewidth}
			\begin{tikzpicture}
			\begin{loglogaxis}[
			/pgf/number format/1000 sep={},
			title={Error $\| \widetilde{\bm{\zeta}} - \bm{\zeta} \|_{\Le}$ over the domain},
			axis lines = left,
			xlabel={Global degrees of freedom $N(\log)$},
			ylabel={Error $N(\log)$},
			xmin=100, xmax=10000,
			ymin=1, ymax=10,
			xtick={100,1000,10000},
			ytick={1,10},
			legend pos=north east,
			ymajorgrids=true,
			grid style=dotted,
			]
			\addplot[
			color=violet,
			mark=diamond*,
			]
			coordinates {
				(161,	6.888790992)
				(590,	4.276703301)
				(2258,	2.179266465)
				(8834,	1.095485665)
				
			};
			\addlegendentry{Hybrid element}
			
			\addplot[
			dashed,
			color=black,
			mark=none,
			]
			coordinates {
				(590,5)
				(7000,1.451)
			};
			
			\end{loglogaxis}
			\draw (4,3.5) node[anchor=north west] {$\mathcal{O}(h)$};
			\end{tikzpicture}
		\end{subfigure}
		\caption{Convergence behaviour of the hybrid element formulation under mesh refinement for the case $\Lc = 0$.}
		\label{fig:mem_con}
	\end{figure}
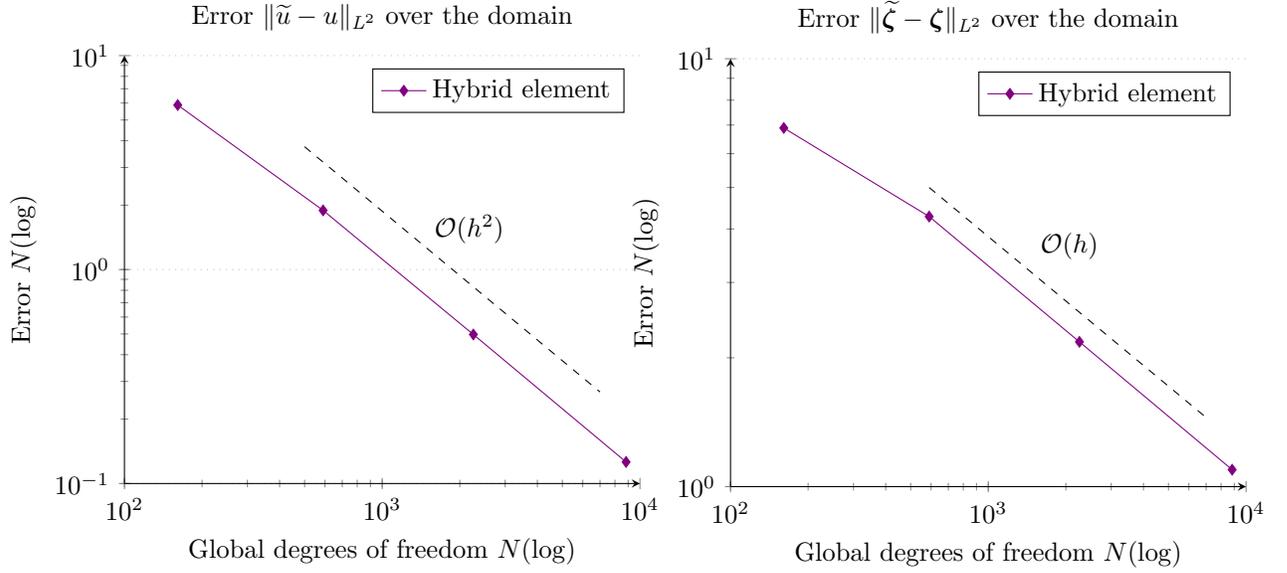

\subsection{Robustness in $\Lc$}
The upper limit of the characteristic length $\Lc$ is defined to be infinity. 
In this example we prove the robustness of our computations for $\Lc \to \infty$.
The analytical solution on $\Omega=[-4,4] \times [-4,4]$ with homogeneous Dirichlet data on $\partial \Omega$ and $\mue=\muma=\mumi=1$ is given by
\begin{subequations}
\label{eq:ex_roust_lc}
\begin{align}
	&\widetilde{u}(x,y) = \cos \left ( \dfrac{\pi \, x}{8} \right ) (y^2 - 16) \exp \left( \dfrac{x+y}{100}\right) \, , \label{eq:ex_roust_lc_a}\\
	&\widetilde{\bm{\zeta}}(x,y) = 2 \begin{bmatrix}
	x (y^2-16) \\
	y(x^2-16) 
	\end{bmatrix} + \dfrac{1}{\Lc^2} \left( \dfrac{x^2}{8} - 2 \right) \left(\dfrac{y^2}{8} -2 \right) \begin{bmatrix}
	-y \\
	x
	\end{bmatrix} \, ,\label{eq:ex_roust_lc_b}
\end{align}
\end{subequations}
from which we can extract the resulting force fields according to \cref{eq:strong1} and \cref{eq:strong2}. We test for convergence using linear elements.

\begin{figure}
	\begin{subfigure}{0.48\linewidth}
		\begin{tikzpicture}
		\begin{loglogaxis}[
		/pgf/number format/1000 sep={},
		title={Primal hybrid method convergence},
		axis lines = left,
		xlabel={Dofs $N(\log)$},
		ylabel={$\|\widetilde{\bm{\zeta}}-\bm{\zeta}\|_{\Hc)}/\|\widetilde{\bm{\zeta}}\|_{\Hc)}$ $N(\log)$},
		xmin=0, xmax=100000,
		ymin=0, ymax=1,
		xtick={0,1,10,100,1000,10000,100000},
		ytick={0.01,0.1,1},
		legend pos=north east,
		ymajorgrids=true,
		grid style=dotted,
		]
		\addplot[
		color=black,
		mark=diamond*,
		]
		coordinates {
			(21,0.5216854447660368)
			(65,0.25421051377935616)
			(225,0.1261947680976178)
			(833,0.06298088577981854)
			(3201,0.03147549422729775)
			(12545,0.015735860700113253)
			(49665,0.007867693930785647)
		};
		\addlegendentry{$\Lc = 10^0$}
		\addplot[
		color=blue,
		mark=o,
		]
		coordinates {
			(21,0.5172795699485528)
			(65,0.2523172272132265)
			(225,0.12529036092846954)
			(833,0.06253627575394381)
			(3201,0.0312545334628986)
			(12545,0.0156255666484233)
			(49665,0.007812570830333642)
		};
		\addlegendentry{$\Lc = 10^2$}
		\addplot[
		color=purple,
		mark=square,
		]
		coordinates {
			(21,0.5172795699038467)
			(65,0.2523173155574729)
			(225,0.12529038826958844)
			(833,0.06253627979219568)
			(3201,0.03125453398671077)
			(12545,0.015625566721171512)
			(49665,0.0078125710156938)
		};
		\addlegendentry{$\Lc = 10^4$}
		\addplot[
		color=magenta,
		mark=+,
		]
		coordinates {
			(21,0.5172795709783622)
			(65,0.25231731613179603)
			(225,0.12529041133745233)
			(833,0.06253697062747193)
			(3201,0.03127123691924457)
			(12545,0.016316343821344368)
			(49665,0.018692260470951332)
		};
		\addlegendentry{$\Lc = 10^6$}
		\addplot[
		color=teal,
		mark=x,
		]
		coordinates {
			(21,0.5172796025169429)
			(65,0.25232104092651275)
			(225,0.12546259698286366)
			(833,0.06745685908391734)
			(3201,0.1418303803906794)
		};
		\addlegendentry{$\Lc = 10^7$}
		\end{loglogaxis}
		\end{tikzpicture}
	\end{subfigure}
	\begin{subfigure}{0.48\linewidth}
		\begin{tikzpicture}
		\begin{loglogaxis}[
		/pgf/number format/1000 sep={},
		title={Mixed hybrid method convergence},
		axis lines = left,
		xlabel={Dofs $N(\log)$},
		ylabel={$\|\widetilde{\bm{\zeta}}-\bm{\zeta}\|_{\Hc)}/\|\widetilde{\bm{\zeta}}\|_{\Hc)}$ $N(\log)$},
		xmin=0, xmax=100000,
		ymin=0, ymax=1,
		xtick={0,1,10,100,1000,10000,100000},
		ytick={0.01,0.1,1},
		legend pos=north east,
		ymajorgrids=true,
		grid style=dotted,
		]
		\addplot[
		color=black,
		mark=diamond*,
		]
		coordinates {
			(26,0.5216854447660368)
			(82,0.25421051377935616)
			(290,0.1261947680976178)
			(1090,0.06298088577981853)
			(4226,0.03147549422729774)
			(16642,0.015735860700113253)
			(66050,0.007867693930785648)
		};
		\addlegendentry{$\Lc = 10^0$}
		\addplot[
		color=blue,
		mark=o,
		]
		coordinates {
			(26,0.5172795699485528)
			(82,0.2523172272132265)
			(290,0.12529036092846954)
			(1090,0.0625362757539438)
			(4226,0.03125453346289859)
			(16642,0.015625566648423295)
			(66050,0.007812570830333633)
		};
		\addlegendentry{$\Lc = 10^2$}
		\addplot[
		color=purple,
		mark=square,
		]
		coordinates {
			(26,0.5172795699039087)
			(82,0.25231731555749004)
			(290,0.12529038826959582)
			(1090,0.06253627979219266)
			(4226,0.031254533986438425)
			(16642,0.015625566713676684)
			(66050,0.007812570838093159)
		};
		\addlegendentry{$\Lc = 10^4$}
		\addplot[
		color=magenta,
		mark=+,
		]
		coordinates {
			(26,0.5172795699039087)
			(82,0.2523173155663307)
			(290,0.12529038828484493)
			(1090,0.06253627981578437)
			(4226,0.03125453402827667)
			(16642,0.015625566815436236)
			(66050,0.00781257103690669)
		};
		\addlegendentry{$\Lc = 10^6$}
		\addplot[
		color=teal,
		mark=x,
		]
		coordinates {
			(26,0.5172795699039088)
			(82,0.25231731556633163)
			(290,0.12529044948170698)
			(1090,0.06253646029642439)
			(4226,0.031254915524514634)
			(16642,0.01562665720001808)
			(66050,0.007816309441558578)
		};
		\addlegendentry{$\Lc = 10^7$}
		\end{loglogaxis}
		\end{tikzpicture}
	\end{subfigure}
	\caption{Convergence behaviour for fixed $\Lc$ on $1\times1$, $2\times2$, $4\times4$, $8\times8$, $16\times16$, $32\times32$, and $64\times64$ structured  quadrilateral grids for the primal and mixed hybrid methods.}
	\label{fig:lcconvergence}
\end{figure}
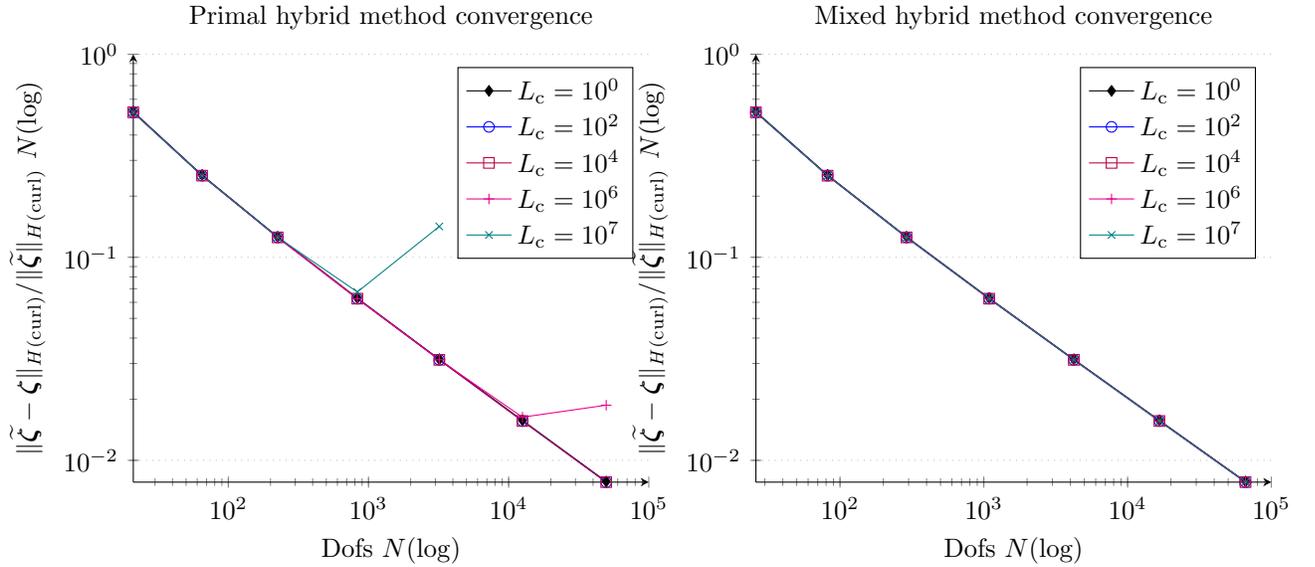

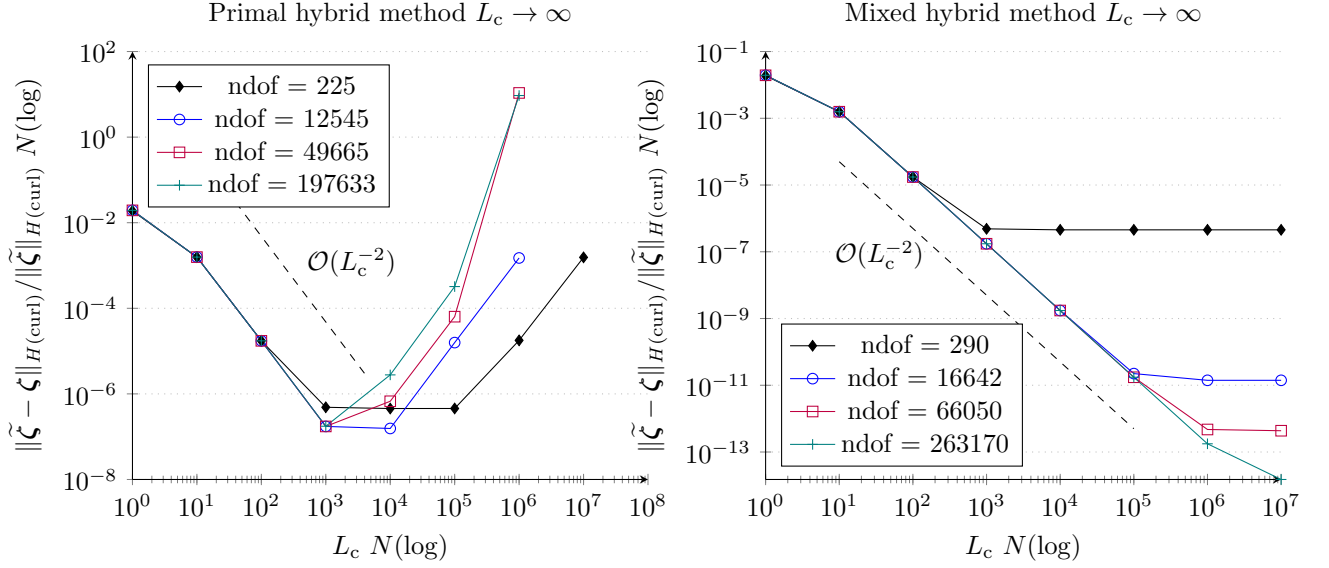
\begin{figure}
	\centering
	\begin{subfigure}{0.48\linewidth}
		\begin{tikzpicture}
		\begin{loglogaxis}[
		/pgf/number format/1000 sep={},
		title={Primal hybrid method $\Lc\to\infty$},
		axis lines = left,
		xlabel={$\Lc$ $N(\log)$},
		ylabel={$\|\widetilde{\bm{\zeta}}-\bm{\zeta}\|_{\Hc)}/\|\widetilde{\bm{\zeta}}\|_{\Hc)}$ $N(\log)$},
		xmin=0, xmax=1e8,
		ymin=0.00000001, ymax=100,
		xtick={0,1,10,100,1000,10000,100000,1000000,10000000,1e8},
		ytick={0.00000001,0.000001,0.0001,0.01,1,100},
		legend pos=north west,
		ymajorgrids=true,
		grid style=dotted,
		]
		\addplot[
		color=black,
		mark=diamond*,
		]
		coordinates {
			(1,0.018824502)
			(1e1,0.001572)
			(1e2,1.737e-5)
			(1e3,4.865e-7)
			(1e4,4.55e-7)
			(1e5,4.56e-7)
			(1e6,1.78e-5)
			(1e7,0.001549)
		};
		\addlegendentry{ndof = 225}
		\addplot[
		color=blue,
		mark=o,
		]
		coordinates {
			(1,0.019429)
			(1e1,0.001572)
			(1e2,1.736e-5)
			(1e3,1.737e-7)
			(1e4,1.556e-7)
			(1e5,1.58e-5)
			(1e6,0.0015)
		};
		\addlegendentry{ndof = 12545}
		\addplot[
		color=purple,
		mark=square,
		]
		coordinates {
			(1,0.019429)
			(1e1,0.001572)
			(1e2,1.736e-5)
			(1e3,1.736e-7)
			(1e4,6.725e-7)
			(1e5,6.37e-5)
			(1e6,10.771)
		};
		\addlegendentry{ndof = 49665}
		\addplot[
		color=teal,
		mark=+,
		]
		coordinates {
			(1,0.019429)
			(1e1,0.001572)
			(1e2,1.736e-5)
			(1e3,1.755e-7)
			(1e4,2.78e-6)
			(1e5,0.000322)
			(1e6,9.371)
		};
		\addlegendentry{ndof = 197633}
		\addplot[
		color=black,
		dashed,
		]
		coordinates {
			(4e1,3e-2)
			(4e2,3e-4)
			(4e3,3e-6)
		};
		\end{loglogaxis}
		\draw (2.2,3.2) node[anchor=north west] {$\mathcal{O}(\Lc^{-2})$};
		\end{tikzpicture}
	\end{subfigure}
	\begin{subfigure}{0.48\linewidth}
		\begin{tikzpicture}
		\begin{loglogaxis}[
		/pgf/number format/1000 sep={},
		title={Mixed hybrid method $\Lc\to\infty$},
		axis lines = left,
		xlabel={$\Lc$ $N(\log)$},
		ylabel={$\|\widetilde{\bm{\zeta}}-\bm{\zeta}\|_{\Hc)}/\|\widetilde{\bm{\zeta}}\|_{\Hc)}$ $N(\log)$},
		xmin=0, xmax=10000000,
		ymin=0, ymax=0.1,
		xtick={0,1,10,100,1000,10000,100000,1000000,10000000},
		ytick={0.0000000000001,0.00000000001,0.000000001,0.0000001,0.00001,0.001,0.1},
		legend pos=south west,
		ymajorgrids=true,
		grid style=dotted,
		]
		\addplot[
		color=black,
		mark=diamond*,
		]
		coordinates {
			(1,0.018824502)
			(1e1,0.001572)
			(1e2,1.737e-5)
			(1e3,4.865e-7)
			(1e4,4.54e-7)
			(1e5,4.54e-7)
			(1e6,4.54e-7)
			(1e7,4.54e-7)
		};
		\addlegendentry{ndof = 290}
		\addplot[
		color=blue,
		mark=o,
		]
		coordinates {
			(1,0.019429)
			(1e1,0.001572)
			(1e2,1.736e-5)
			(1e3,1.737e-7)
			(1e4,1.738e-9)
			(1e5,2.236e-11)
			(1e6,1.407e-11)
			(1e7,1.407e-11)
		};
		\addlegendentry{ndof = 16642}
		\addplot[
		color=purple,
		mark=square,
		]
		coordinates {
			(1,0.019429)
			(1e1,0.001572)
			(1e2,1.736e-5)
			(1e3,1.738e-7)
			(1e4,1.737e-9)
			(1e5,1.737e-11)
			(1e6,4.73e-13)
			(1e7,4.34e-13)
		};
		\addlegendentry{ndof = 66050}
		\addplot[
		color=teal,
		mark=+,
		]
		coordinates {
			(1,0.019429)
			(1e1,0.001572)
			(1e2,1.736e-5)
			(1e3,1.738e-7)
			(1e4,1.738e-9)
			(1e5,1.738e-11)
			(1e6,1.763e-13)
			(1e7,1.503e-14)
		};
		\addlegendentry{ndof = 263170}
		\addplot[
		color=black,
		dashed,
		]
		coordinates {
			(1e1,5e-5)
			(1e2,5e-7)
			(1e3,5e-9)
			(1e4,5e-11)
			(1e5,5e-13)
		};

		\end{loglogaxis}
		\draw (0.8,3.3) node[anchor=north west] {$\mathcal{O}(\Lc^{-2})$};
		\end{tikzpicture}
	\end{subfigure}
	\caption{Convergence behaviour for $\Lc \to \infty$ for fixed $4\times4$, $16\times16$, $32\times32$, $64\times64$ and $128\times128$ grids.}
	\label{fig:lc_inf_test}
\end{figure}
As expected from the theory, we observe uniform convergence up to the point where rounding errors occur in the primal method for very large $\Lc$ terms. The convergences of the mixed formulation remains stable for all values of $\Lc$ as it is not affected by rounding errors, cf. \cref{fig:lcconvergence}. Using lowest order linear nodal elements for $\bm{\zeta}$ leads to non-robust behaviour in $\Lc$ in terms of immense locking. Considering quadratic Lagrange elements overcomes this locking phenomena, however, at the cost of more dofs.

To test the convergence depending on $\Lc$, \cref{eq:mixed_method_conv_Lc}, for the case $\Lc\to\infty$ we use quadratic elements - i.e., quadratic $\Hone$ and N\'{e}d\'{e}lec elements, and linear $\Le$ elements for $m$ in the mixed formulation - in NGSolve and four different structured grids. The same domain as in the previous example is considered and for the limit solution \cref{eq:ex_roust_lc} is used, with $\Lc\to \infty$ in \cref{eq:ex_roust_lc_b}. Again, the primal methods suffers for large values of $\Lc$ from rounding errors, whereas for the mixed method we observe the expected quadratic convergence rate up to the discretization error, compare \cref{eq:conv_lc_h} and \cref{fig:lc_inf_test}.\newline

\subsection{Convergence for $\Lc \to \infty$}
We prove the theoretical result of \cref{thm:limit_gradu_zeta}, with the same domain, boundary conditions, and material constants as in the previous example, by setting the external force and moments  
\begin{align}
f = 0 \, , \quad r = (16-x^2)(16-y^2)(xy-y^2) \, ,\quad \Psi =x^3y^2-xy^2(1-x)-\frac{256}{9}\, ,\quad \bm{\omega} = \nabla r + \Dc(\Psi) \, ,
\end{align}
and testing for convergence $\|\nabla u-\bm{\zeta}\|_{\Hc)}=\mathcal{O}(\Lc^{-2})$ for $\Lc\to\infty$ using NGSolve with linear base functions.

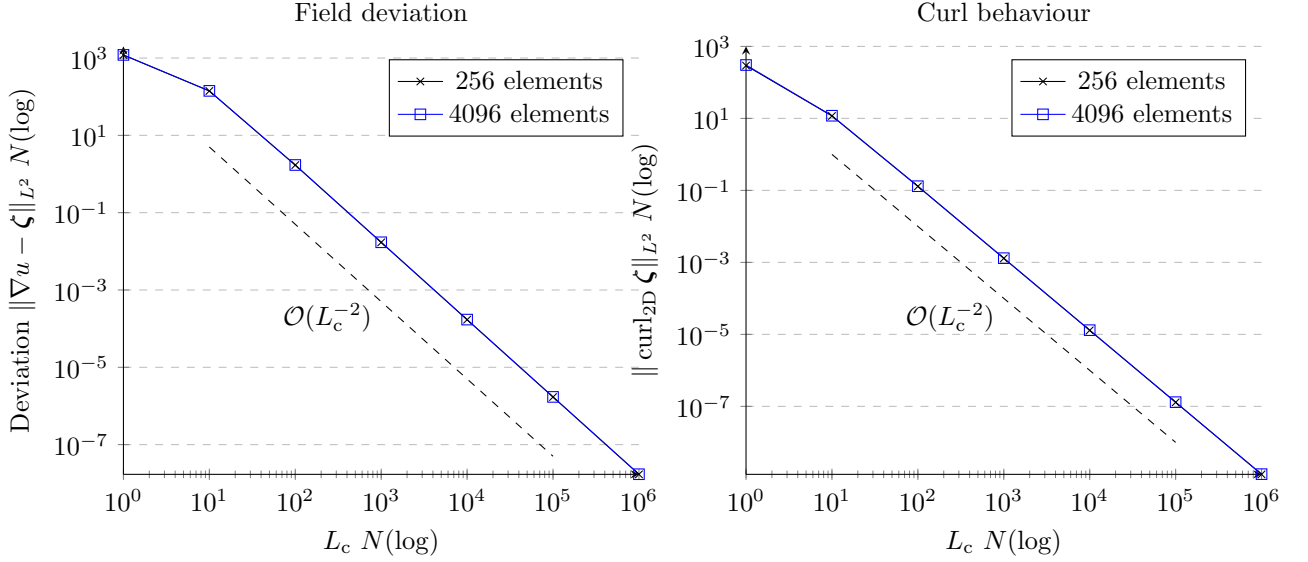
\begin{figure}
	\centering
	\begin{subfigure}{0.48\linewidth}
		\begin{tikzpicture}
		\begin{loglogaxis}[
		/pgf/number format/1000 sep={},
		title={Field deviation},
		axis lines = left,
		xlabel={$\Lc$ $N(\log)$},
		ylabel={Deviation $\| \nabla u - \bm{\zeta} \|_{\Le}$ $N(\log)$},
		xmin=0, xmax=1000000,
		ymin=0, ymax=2000,
		xtick={0,1,10,100,1000,10000,100000,1000000},
		ytick={0.0000001,0.00001,0.001,0.1,10,1000},
		legend pos=north east,
		ymajorgrids=true,
		grid style=dashed,
		]
		\addplot[
		color=black,
		mark=x,
		]
		coordinates {
			(1,1194.7210203536529)
			(10,140.44127647316296)
			(100,1.699082419004753)
			(1000, 0.017027730407461496)
			(10000,0.00017028100422524504)
			(100000,1.702810395631907e-06)
			(1000000,1.7027965508902514e-08)
		};
		\addlegendentry{256 elements}
		\addplot[
		color=blue,
		mark=square,
		]
		coordinates {
			(1,1205.8894192777095)
			(10,141.3888267047627)
			(100, 1.7105569924857034)
			(1000,0.01714274145447953)
			(10000,0.00017143114224072985)
			(100000,1.7143133351772662e-06)
			(1000000,1.714470882962487e-08)
		};
		\addlegendentry{4096 elements}
		\addplot[
		color=black,
		dashed,
		]
		coordinates {
			(10,5)
			(100, 5e-2)
			(1000,5e-4)
			(10000,5e-6)
			(100000,5e-8)
		};

		\end{loglogaxis}
		\draw (2,2.4) node[anchor=north west] {$\mathcal{O}(\Lc^{-2})$};
		\end{tikzpicture}
	\end{subfigure}
	\begin{subfigure}{0.48\linewidth}
		\begin{tikzpicture}
		\begin{loglogaxis}[
		/pgf/number format/1000 sep={},
		title={Curl behaviour},
		axis lines = left,
		xlabel={$\Lc$ $N(\log)$},
		ylabel={$\| \curl \bm{\zeta} \|_{\Le}$ $N(\log)$},
		xmin=0, xmax=1000000,
		ymin=0, ymax=1000,
		xtick={0,1,10,100,1000,10000,100000,1000000},
		ytick={0.00000000001,0.000000001,0.0000001,0.00001,0.001,0.1,10,1000},
		legend pos=north east,
		ymajorgrids=true,
		grid style=dashed,
		]
		\addplot[
		color=black,
		mark=x,
		]
		coordinates {
			(1,292.7315971120642)
			(10,11.69160167992485)
			(100, 0.12909094065427285)
			(1000,0.0012924418982838588)
			(10000,1.2924572649962118e-05)
			(100000,1.292457159147918e-07)
			(1000000,1.2924540080682432e-09)
		};
		\addlegendentry{256 elements}
		\addplot[
		color=blue,
		mark=square,
		]
		coordinates {
			(1,303.86715103001467)
			(10,11.85328550134601)
			(100, 0.13070547546311248)
			(1000,0.0013085868856933859)
			(10000,1.3086022472297733e-05)
			(100000,1.3086022945633506e-07)
			(1000000,1.3086032929161493e-09)
		};
		\addlegendentry{4096 elements}
		
		\addplot[
		color=black,
		dashed,
		]
		coordinates {
			(10,1)
			(100, 1e-2)
			(1000,1e-4)
			(10000,1e-6)
			(100000,1e-8)
		};

		\end{loglogaxis}
		\draw (2,2.4) node[anchor=north west] {$\mathcal{O}(\Lc^{-2})$};
		\end{tikzpicture}
	\end{subfigure}
	\caption{Convergence behaviour of the difference $\nabla u-\bm{\zeta}$ and $\curl\bm{\zeta}$ for $\Lc \to \infty$ with primal hybrid method.}
	\label{fig:lcinf}
\end{figure}
The results are computed using the primal method. By staying within the rounding precision bounds retrieved from our investigation of the robustness in $\Lc$, we are able to find results converging quadratically to the previously derived expectations, see \cref{fig:lcinf}.

\subsection{The consistent coupling condition}
\label{subsec:num_ex_cons}
We conclude our investigation by considering the consistent coupling condition on both the full and relaxed micromorphic continuum models using NGSolve with the primal method. We set the domain $\Omega = [-4,4] \times [-4,4]$ with the material parameters $\mue, \, \mumi , \muma = 1$, the boundary conditions
\begin{align}
	u(x,y) \at_{\partial \Omega} = y^2 -x^2 , \, \qquad \langle \bm{\zeta} , \, \bm{\tau} \rangle \at_{\partial \Omega} = \langle \nabla u , \, \bm{\tau} \rangle \at_{\partial \Omega} = \langle \begin{bmatrix}
	-2x & 2y 
	\end{bmatrix}^T , \, \bm{\tau} \rangle \at_{\partial \Omega} \, ,
	\label{eq:fullvsrel}
\end{align}
and the external forces
\begin{align}
	f = 0 \, , \quad \qquad \bm{\omega} = \begin{bmatrix}
	-y & x
	\end{bmatrix}^T \, ,
\end{align}
and test for convergence in both micromorphic formulations with increasing characteristic lengths $\Lc$. 

\begin{figure}
	\centering
	\begin{subfigure}{0.48\linewidth}
		\begin{tikzpicture}
		\begin{loglogaxis}[
		/pgf/number format/1000 sep={},
		title={Energy in the relaxed micromorphic model},
		axis lines = left,
		xlabel={$\Lc$ $N(\log)$},
		ylabel={Energy $I(u,\bm{\zeta})$ $N(\log)$},
		xmin=1, xmax=1000000,
		ymin=1000, ymax=10000,
		xtick={1,10,100,1000,10000,100000,1000000},
		ytick={1000, 10000},
		legend pos=north east,
		ymajorgrids=true,
		grid style=dashed,
		]
		\addplot[
		color=black,
		mark=x,
		]
		coordinates {
			(1 ,1534.2056043353814)
			(10, 2589.2003333455145)
			(100, 2718.5804252472008 )
			(1000,2720.035133357932 )
			(10000, 2720.0496985876057 )
			(100000,2720.0498442745698 )
			(1000000, 2720.0501444577994)
		};
		\addlegendentry{256 elements}
		\addplot[
		color=blue,
		mark=square,
		]
		coordinates {
			(1 ,1526.8952938735458)
			(10, 2595.743248641167)
			(100,  2726.5183820712004 )
			(1000,  2727.989006761117 )
			(10000, 2728.003731382987 )
			(100000,2728.0038792432097 )
			(1000000, 2728.007944906196)
		};
		\addlegendentry{1024 elements}
		
		\addplot[
		color=violet,
		mark=diamond*,
		]
		coordinates {
			(1 ,1525.0018565734229)
			(10, 2597.384191509069)
			(100,  2728.510742578785 )
			(1000,  2729.9853764556688 )
			(10000, 2730.000141227262 )
			(100000,2730.0002987709795)
			(1000000, 2730.077540112807)
		};
		\addlegendentry{4096 elements}
		
		\end{loglogaxis}
		\end{tikzpicture}
	\end{subfigure}
	\begin{subfigure}{0.48\linewidth}
		\begin{tikzpicture}
		\begin{loglogaxis}[
		/pgf/number format/1000 sep={},
		title={Energy in the full micromorphic model},
		axis lines = left,
		xlabel={$\Lc$ $N(\log)$},
		ylabel={Energy $I(u,\bm{\zeta})$ $N(\log)$},
		xmin=1, xmax=1000000,
		ymin=1000, ymax=1e15,
		xtick={1,10,100,1000,10000,100000,1000000},
		ytick={1e3, 1e7, 1e11, 1e15},
		legend pos=north east,
		ymajorgrids=true,
		grid style=dashed,
		]
		\addplot[
		color=black,
		mark=x,
		]
		coordinates {
			(1 ,1930.0263225276706)
			(10, 29608.370698368657)
			(100, 2756732.387148474 )
			(1000,275467158.6937412 )
			(10000, 27546509768.74214 )
			(100000,2754650770773.381 )
			(1000000, 275465076871236.84)
		};
		\addlegendentry{256 elements}
		\addplot[
		color=blue,
		mark=square,
		]
		coordinates {
			(1 ,2017.2953458784796)
			(10, 40106.943430558196)
			(100,  3807372.788160398 )
			(1000,  380531955.1553537 )
			(10000, 38052990170.984634)
			(100000,3805298811753.7163 )
			(1000000, 380529880970026.25)
		};
		\addlegendentry{1024 elements}
		
		\addplot[
		color=violet,
		mark=diamond*,
		]
		coordinates {
			(1 ,2120.8789480216687)
			(10, 51003.107144379544)
			(100,  4897273.447490505 )
			(1000,  489522298.8198551 )
			(10000, 48952024815.097885 )
			(100000,4895202276442.722)
			(1000000, 489520227439203.8)
		};
		\addlegendentry{4096 elements}
		
		\end{loglogaxis}
		\end{tikzpicture}
	\end{subfigure}
	\caption{Energy convergence in the relaxed and full micromorphic models according to \cref{eq:fullvsrel}.}
	\label{fig:lcinffull}
\end{figure}
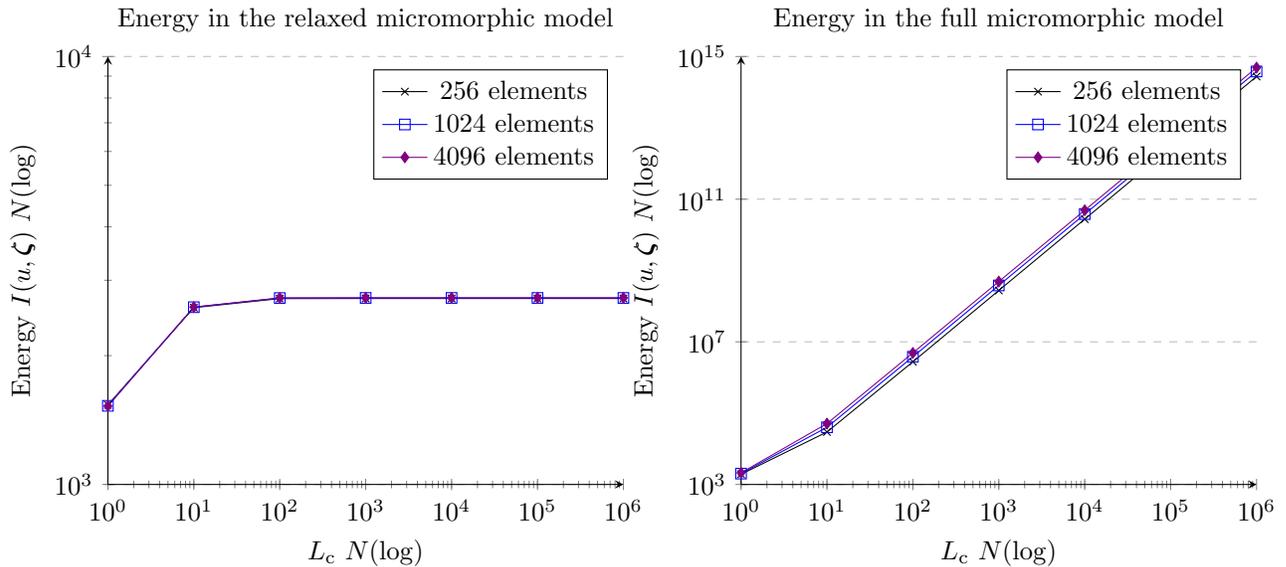
As observed in \cref{fig:lcinffull}, the relaxed micromorphic continuum converges towards a finite energy, whereas the non-trivial boundary conditions on the full micromorphic continuum lead to boundary-layers and consequently, ever-increasing energy for $\Lc \to \infty$. The result is consistent with the problematic mentioned in \cref{rem:full}.

\section{Conclusions and outlook}
The relaxed micromorphic continuum theory introduces the Curl operator in the formulation of the free energy functional. As a result, the solution of the weak form lies in the combined space $\Hone \times \Hc)$. The Lax--Milgram theorem confirms this result by assuring existence and uniqueness for the combined space. Our benchmarks with a completely nodal finite element show its capacity to approximate solutions in the combined space. However, the tests also show its inability to find the exact solution for discontinuous microdistortion fields and the corresponding sub-optimal convergence. A comparison between the linear nodal and hybrid element formulations also reveals the difference in the arising elemental stiffness matrices, namely $\bm{K}_{\text{nodal}} \in \mathbb{R}^{12\times 12}$ and $\bm{K}_{\text{hybrid}} \in \mathbb{R}^{8\times 8}$, resulting in slower computation times for the nodal element. In contrast, the hybrid element yields stable approximations and convergence rates for all tested scenarios, being capable of finding the exact solution also for discontinuous microdistortion fields. The relaxed micromorphic theory aims to capture the mechanical behaviour of metamaterials, highly homogeneous materials and the entire spectrum in between. To that end, the characteristic length $\Lc$ takes the role of a weighting parameter, determining the influence of the energy from the dislocation density (the energy depending on the curl operator). The range of the characteristic length $\Lc$ is an open topic of research into metamaterials. However, from a theoretical point of view, it may vary between zero and infinity. Our tests reveal the arising instability of convergence where increasingly large $\Lc$ parameters are concerned and emergence of locking effects if linear nodal elements are chosen to approximate the microdistortion. For the case of the hybrid element, lost precision can be recovered via the formulation of the corresponding mixed problem. Locking effects in the nodal version of the microdistortion can be alleviated via higher order polynomials at the cost of increased dofs. In addition, also in $\Lc = 0$ setting, where the external moment $\bm{\omega}$ vanishes, we recognize the optimality of using $\Hc)$-elements for the computation of the microdistortion, seeing as it is in fact the natural space for the microdistortion in this setting. Lastly, we recognize the advantage of the relaxed micromorphic continuum with regard to its ability to generate finite energies as $\Lc \to \infty$ for arbitrary boundary conditions. 

These findings build the basis for the extension of the formulation to the fully three-dimensional or a statically condensed two-dimensional version of the \AS{full} relaxed micromorphic continuum.

\section*{Acknowledgements}
\label{sec:acknowledgements}
The support by the Austrian Science Fund (FWF) project W\,1245 is gratefully acknowledged.
P. Neff acknowledges support in the framework of the DFG Priority Programme 2256 Variational Methods for Predicting Complex Phenomena in Engineering Structures and Materials Neff 902/10-1, No: 440935806.

\AS{The authors are very grateful to the referees for their comments.}
	
\bibliographystyle{spmpsci}   

\bibliography{Ref}   

\appendix
\section{Derivation of the strong form} \label{ap:B}
In order to find the strong form of the Euler-Lagrange equations to \cref{eq:weak} we start with the most general setting $\Gamma^u_D\cap \Gamma^u_N=\emptyset$, $\overline{\Gamma^u_D}\cup\overline{\Gamma^u_N}=\partial\Omega$ and $\Gamma^\zeta_D\cap \Gamma^\zeta_N=\emptyset$, $\overline{\Gamma^\zeta_D}\cup\overline{\Gamma^\zeta_N}=\partial\Omega$. The Dirichlet and Neumann boundary parts of $u$ and $\bm{\zeta}$ and assume that $|\Gamma_D^u|>0$ (for Lax--Milgram solvability). We assume smooth fields such that we can integrate by parts. Using the Green identity
\begin{align}
	\int_{\Omega} \di q \, \vb{v} \, \dd X = \oint_{\partial \Omega}   \langle q \, \bm{\nu} ,  \, \vb{v} \rangle \, \dd s - \int_{\Omega} \langle \nabla q , \, \vb{v} \rangle \, \dd X \, , \quad \vb{v} \in \C^1(\Omega,\mathbb{R}^2) \, , \quad q \in \C^1(\Omega,\mathbb{R}) \, , \label{eq:gr1}
\end{align}
where $\bm{\nu}$ is the normal vector on the boundary, and splitting the boundary terms of the first weak form \cref{eq:a}, we find
\begin{align}
	\int_{\Omega}  2 \mue \langle (\nabla u - \bm{\zeta}) , \, \nabla \delta u \rangle - \langle \delta u , \, f \rangle \, \dd X  = &\int_{\Gamma_D^u} \delta u \langle (\nabla u - \bm{\zeta}) , \, \bm{\nu} \rangle \, \dd s  + \int_{\Gamma_N^u} \delta u \langle (\nabla u - \bm{\zeta}) , \, \bm{\nu} \rangle \, \dd s \notag \\ & - \int_{\Omega} \langle \di (\nabla u - \bm{\zeta}) - f, \, \delta u \rangle \, \dd X = 0 \,  \quad \text{ for all } \delta u \in \C^1(\Omega,\mathbb{R}) \, . 
\end{align}
As the Dirichlet data is directly incorporated into the space we have $\delta u=0$ on $\Gamma_D^u$ and thus, for given Dirichlet data $\widetilde{u}$, we obtain the strong form
	\begin{align}
	 -2\mue\di(\nabla u - \bm{\zeta}) &= f & &\text{ in } \Omega \, ,  \notag \\
	 u &= \widetilde{u}&& \text{ on } \Gamma^u_D\, ,  \label{eq:der1}\\
	 \langle \nabla u , \, \bm{\nu} \rangle &= \langle \bm{\zeta} , \, \bm{\nu} \rangle &&  \text{ on } \Gamma^u_N\,. \notag
	\end{align}  
For the second weak form \cref{eq:3} we employ another Green identity
\begin{align}
	\int_{\Omega}  \lambda   \, \curl \vb{q}  \, \dd X = \oint_{\partial \Omega} \langle  \lambda \, \vb{q} , \, \bm{\tau} \rangle \, \dd s + \int_{\Omega} \langle \Dc \lambda, \, \vb{q} \rangle \, \dd X \, , \quad \lambda \in  \C^1(\Omega,\mathbb{R}) \, , \quad \vb{q} \in \C^1(\Omega,\mathbb{R}^2) \, , \label{eq:gr2}
\end{align}
and split the boundary, obtaining for all $\delta \bm{\zeta} \in \C^1(\Omega,\mathbb{R}^2)$ 
\begin{align}
	\int_{\Omega}  2\mue \, & \langle (\nabla u - \bm{\zeta}) , \, (-\delta \bm{\zeta}) \rangle + 2\mumi \,  \langle \bm{\zeta} , \, \delta \bm{\zeta} \rangle + \muma \, \Lc ^2 \, \langle \curl\bm{\zeta} , \, \curl\delta \bm{\zeta} \rangle - \langle \delta \bm{\zeta} , \, \bm{\omega} \rangle \,  \dd X \notag \\& = \int_{\Omega}  2\mue \, \langle (\nabla u - \bm{\zeta}) , \, (-\delta \bm{\zeta}) \rangle  + 2\mumi \,  \langle \bm{\zeta} , \, \delta \bm{\zeta} \rangle \, + \muma \, \Lc^2 \langle  \Dc (\curl \bm{\zeta}) , \, \delta \bm{\zeta} \rangle \,  - \langle \delta \bm{\zeta} , \, \bm{\omega} \rangle \, \dd X  \notag \\&  \quad + \int_{\Gamma_D^\zeta} \muma \, \Lc^2 \, \curl( \bm{\zeta}) \langle \delta \bm{\zeta} , \, \bm{\tau} \rangle \, \dd s +  \int_{\Gamma_N^\zeta} \muma \, \Lc^2 \, \curl( \bm{\zeta}) \langle \delta \bm{\zeta} , \, \bm{\tau} \rangle \, \dd s = 0  \, .
\end{align}
Again, the Dirichlet data is incorporated into the space, such that the following strong formulation arises
	\begin{align}
	 -2\mue (\nabla u - \bm{\zeta}) + 2\mumi \bm{\zeta} + \muma \Lc^2 \, \Dc(  \curl\bm{\zeta} )  &=  \bm{\omega} &&\text{ in } \Omega  \, , \notag \\
	 \curl\bm{\zeta}&=0 && \text{ on }\Gamma_N^\zeta \,, \label{eq:der2} \\
	 \langle \bm{\zeta} , \, \bm{\tau} \rangle &= \langle \nabla \widetilde{\bm{\zeta}} , \, \bm{\tau} \rangle && \text{ on } \Gamma^\zeta_D \,. \notag
	\end{align}  
The complete boundary value problem is given by \cref{eq:der1,eq:der2}.

\section{Constructing analytical solutions} \label{ap:A}
The predefined fields are given by $\widetilde{u}$ and $\widetilde{\bm{\zeta}}$.
	We redefine the variables of the strong form $u^* = u - \widetilde{u}$ and $\bm{\zeta}^* = \bm{\zeta} - \widetilde{\bm{\zeta}}$ and insert them into the partial differential equation
	\begin{subequations}
		\begin{align}
		-2\mue\di(\nabla(u-\widetilde{u}) - (\bm{\zeta}-\widetilde{\bm{\zeta}})) &= 0 \, , \\[2ex]
		-2\mue (\nabla(u-\widetilde{u}) - (\bm{\zeta}-\widetilde{\bm{\zeta}})) + 2\mumi (\bm{\zeta}-\widetilde{\bm{\zeta}}) 
		+ \muma \Lc^2   \, \Dc \curl(\bm{\zeta}-\widetilde{\bm{\zeta}})  &= 0 \, , 
		\end{align}
	\end{subequations}
yielding compositions of additive terms. Therefore, we can rearrange the equations
\begin{subequations}
	\begin{align}
	2\mue\di(\nabla u - \bm{\zeta}) &=  2\mue\di(\nabla \widetilde{u} - \widetilde{\bm{\zeta}}) \, ,
	\\[2ex]
	-2\mue (\nabla u - \bm{\zeta}) + 2\mumi \bm{\zeta} + \muma \Lc^2  \, \Dc (\curl\bm{\zeta}) 
	&= -2\mue (\nabla \widetilde{u} - \widetilde{\bm{\zeta}}) + 2\mumi \widetilde{\bm{\zeta}} + \muma \Lc^2  \, \Dc (\curl\widetilde{\bm{\zeta}})  \, . \notag
	\end{align}
\end{subequations}
It is clear that the solutions of the PDE must be $u = \widetilde{u}$ and $\bm{\zeta} = \widetilde{\bm{\zeta}}$.
Since both $\widetilde{u}$ and $\widetilde{\bm{\zeta}}$ are known a priori, their insertion in the PDE can be calculated. We define the calculated fields
\begin{align}
	f &:= -2\mue\di(\nabla \widetilde{u} - \widetilde{\bm{\zeta}}) \, , \notag \\[2ex]
	\bm{\omega} &:= -2\mue (\nabla \widetilde{u} - \widetilde{\bm{\zeta}}) + 2\mumi \widetilde{\bm{\zeta}} + \muma \Lc^2  \, \Dc (\curl\widetilde{\bm{\zeta}}) \, .
\end{align}
The strong forms with the newly found right-hand sides are multiplied with the corresponding test functions
\begin{subequations}
	\begin{align}
	&\int_{\Omega}  2\mue \langle \di(\nabla u - \bm{\zeta}) , \, \delta u \rangle  \, \dd X = - \int_{\Omega} \langle f , \, \delta u \rangle \, \dd X \, , \\[2ex]
	&\int_{\Omega}  -2\mue \langle (\nabla u - \bm{\zeta}) , \, \delta \bm{\zeta} \rangle + 2\mumi \langle \bm{\zeta} , \, \delta \bm{\zeta} \rangle + \muma \Lc^2 \langle \Dc (\curl\bm{\zeta})  , \,  \delta \bm{\zeta} \rangle \, \dd X = \int_{\Omega} \langle \bm{\omega} , \, \delta \bm{\zeta} \rangle \, \dd X  \, .
	\end{align}
\end{subequations}
Employing Greens' identities \cref{eq:gr1,eq:gr2} we find
\begin{subequations}
	\begin{align}
	&\oint_{\partial \Omega} 2\mue \, \delta u \langle (\nabla u - \bm{\zeta}) , \, \bm{\nu} \rangle \, \dd s - \int_{\Omega} 2 \mue \langle (\nabla u - \bm{\zeta}) , \, \nabla \delta u \rangle \, \dd X  = - \int_{\Omega} \langle f , \, \delta u \rangle \, \dd X \, , \\[2ex]
	&\int_{\Omega} -2\mue  \langle (\nabla u - \bm{\zeta}) , \, \delta \bm{\zeta} \rangle + 2\mumi \, \langle \bm{\zeta}  , \, \delta \bm{\zeta} \rangle + \muma \, \Lc ^2 \, \langle \curl\bm{\zeta} , \,  \curl\delta \bm{\zeta} \rangle \, \dd X 
	- \muma \Lc^2 \,\oint_{\partial \Omega}  \curl\bm{\zeta}  \langle \delta \bm{\zeta} , \, \bm{\tau} \rangle \, \dd s \notag \\ 
	& \qquad \qquad \qquad \qquad \qquad \quad \qquad \qquad \qquad \qquad \qquad \qquad \qquad = \int_{\Omega} \langle \bm{\omega} , \, \delta \bm{\zeta} \rangle \, \dd X  \, .
	\end{align}
\end{subequations}
The latter integrations generate terms for transmissions on the boundary $\partial \Omega$. As the Dirichlet data is directly incorporated into the space and the natural Neumann boundary conditions \cref{eq:neumann,eq:cur} hold, we observe 
\begin{align}
	\int_{\Gamma_N^u} 2\mue \, \delta u \langle (\nabla u - \bm{\zeta}) , \, \bm{\nu} \rangle \, \dd s = 0 \, , \quad \int_{\Gamma_N^\zeta}  \curl\bm{\zeta}  \langle \delta \bm{\zeta} , \, \bm{\tau} \rangle \, \dd s = 0 \, ,
\end{align} 
allowing us to find the original weak formulation with the corresponding force and moment
\begin{align}
	&\int_{\Omega} 2\mue \langle (\nabla u - \bm{\zeta}) , \, \nabla \delta u \rangle \, \dd X = \int_{\Omega} \langle f , \, \delta u \rangle \, \dd X \, , \\[2ex]
	&\int_{\Omega}  - 2\mue  \langle (\nabla u - \bm{\zeta}) , \, \delta \bm{\zeta} \rangle + 2\mumi \, \langle \bm{\zeta} , \, \delta \bm{\zeta} \rangle \nonumber  + \muma \, \Lc ^2 \, \langle \curl\bm{\zeta} , \, \curl\delta \bm{\zeta} \rangle \, \dd X  = \int_{\Omega} \langle \bm{\omega} , \, \delta \bm{\zeta} \rangle \, \dd X  \, .
\end{align}

\end{document}